\documentclass[letterpaper,10pt]{article}

\usepackage[margin=0.5in]{geometry}
\usepackage[pagewise]{lineno}
\usepackage[dvipsnames]{xcolor}
\usepackage[hang,small,bf]{caption}
\usepackage{fullpage}
\usepackage{enumerate}
\usepackage{authblk}
\usepackage[colorinlistoftodos]{todonotes}
\usepackage{verbatim}%comments out anything between \begin{comment} and \end{comment}
\usepackage{section, amsthm, textcase, setspace, amssymb, lineno, amsmath, amssymb, amsfonts, latexsym, fancyhdr, longtable, ulem, mathtools}
\usepackage{epsfig, graphicx, pstricks,pst-grad,pst-text,tikz,colortbl}
\usepackage{epsf}
\usepackage{graphicx, color}
\usepackage{xcolor}
\usepackage{float}
\usepackage[rflt]{floatflt}
\usepackage[most]{tcolorbox}
\usepackage{amsfonts}
\usepackage{latexsym,enumitem}
\usetikzlibrary{fit,matrix,positioning}
\usepackage{pdflscape}
\usetikzlibrary{decorations.pathreplacing}
\usepackage{mathrsfs}
\usepackage{makecell}
\usepackage{tikz}
\usetikzlibrary{decorations.markings}
\usetikzlibrary{arrows.meta}
\usetikzlibrary{arrows}
\usetikzlibrary{shapes.geometric}
\usepackage{yfonts}
\usepackage{faktor}
\usetikzlibrary{patterns}

\newtheorem{theorem}{Theorem}
\newtheorem{lemma}[theorem]{Lemma}
\newtheorem{corollary}[theorem]{Corollary}
\newtheorem{conj}[theorem]{Conjecture}
\newtheorem{definition}[theorem]{Definition}
\newtheorem{ex}[theorem]{Example}
\newtheorem{prop}[theorem]{Proposition}
\newtheorem*{theorem*}{Theorem}
\newtheorem{remark}[theorem]{Remark}

\newcommand{\bl}[1]{{\color{blue} #1}}

\newcommand{\ldownarrow}{\Big\downarrow}

\newcommand{\precdot}{\prec\mathrel{\mkern-5mu}\mathrel{\cdot}}

\tikzstyle{vertex}=[circle, draw, inner sep=0pt, minimum size=13pt]

\tikzstyle{svertex}=[circle, draw, inner sep=0pt, minimum size=8pt]

\usepackage[backend=bibtex]{biblatex}
\begin{document}

\title{Shellability of Kohnert posets}
\author[*]{Celia Kerr}
\author[**]{Nicholas W. Mayers}
\author[*]{Nicholas Russoniello}

\affil[*]{Department of Mathematics, College of William \& Mary, Williamsburg, VA 23185}
\affil[**]{Department of Mathematics, North Carolina State University, Raleigh, NC 27695}

\maketitle

\begin{abstract} 
\noindent
In this paper, we are concerned with identifying among the family of posets associated with Kohnert polynomials, those whose order complex has a certain combinatorial property. In particular, for numerous families of Kohnert polynomials, including key polynomials, we determine when the associated Kohnert posets are (EL-)shellable. Interestingly, under certain diagram restrictions, (EL-)shellability of a Kohnert poset is equivalent to multiplcity freeness of the associated Kohnert polynomial.
\end{abstract}

\section{Introduction}

In his thesis (\textbf{\cite{Kohnert}}, 1990), Kohnert showed that Demazure characters (a.k.a. key polynomials) encode certain collections of diagrams consisting of cells distributed in the first quadrant. Moreover, it was conjectured, and subsequently proven \textbf{\cite{AssafSchu, Winkel2, Winkel1}}, that a similar result applied to Schubert polynomials. Motivated by this, Assaf and Searles (\textbf{\cite{KP1}}, 2022) applied the corresponding polynomial construction paradigm more generally and defined the notion of a ``Kohnert polynomial". Such polynomials encode certain collections of diagrams consisting of cells distributed in the first quadrant that are related to a ``seed" diagram by a sequence of moves, called ``Kohnert moves", that change the position of at most one cell (see also \textbf{\cite{KP2}}). In (\textbf{\cite{KP2}}, 2022), the author notes that one can define a natural poset structure on the collection of diagrams encoded by the terms of a Kohnert polynomial. Moreover, the author of \textbf{\cite{KP2}} illustrates that such ``Kohnert posets" arising from Kohnert polynomials are not generally well-behaved, noting that, in general, they are not lattices, ranked, nor do they have a unique minimal element. In recent work by L. Colmenarjo, et al. (\textbf{\cite{KPoset1}}, 2023), the authors initiate an investigation into the ``not-so-well-behaved" structure of such posets, focusing on identifying when they are ranked and/or bounded. Here, we consider when Kohnert posets are (EL-)shellable and the consequences regarding the associated polynomial.

Starting with modest restrictions, we first consider the Kohnert posets associated with diagrams for which either (1) there is at most one cell per column or (2) the first two rows are empty. Under these restrictions, we are able to find a complete characterization of when the associated Kohnert poset is (EL-)shellable. In fact, for both cases, the indexing diagram (modulo cells in the first row) must be what we call a ``hook diagram" (see Theorem~\ref{thm:hook}). Moreover, for diagrams of the form (1) or (2), we find that (EL-)shellability of the associated Kohnert poset is equivalent to the corresponding Kohnert polynomial being multiplicity free (see Theorem~\ref{thm:hookmf}). With these results, it remains to consider the case of diagrams that contain at least one cell within the first two rows and at least one column with more than one cell.

Based on computational evidence, the case of Kohnert posets associated with diagrams containing at least one cell within the first two rows and at least one column with more than one cell seems much more complicated than that previously considered. Consequently, in this direction we focus on a special case of historical significance: key diagrams. The Kohnert polynomials of key diagrams are Demazure characters; this result forms the motivation for \textbf{\cite{Kohnert}}. In the case of key diagrams, we are able to characterize when the associated Kohnert posets are graded and EL-shellable (see Theorem~\ref{thm:kshell}). Our characterization is in terms of an associated weak composition avoiding three different patterns. Similar to the families of diagrams discussed above, we find that there is a relationship between (EL-)shellability of the Kohnert poset and the Kohnert polynomial being multiplicity-free, though, in this case, the relationship is not as strong. In particular, we find that for key diagrams, if the Kohnert poset is graded and EL-shellable, then the associated Kohnert polynomial is multiplicity-free (see Theorem~\ref{thm:keymf}). On the other hand, we are also able to find an example of a key diagram for which the Kohnert poset is not shellable, but the Kohnert polynomial is multiplicity-free.

The remainder of the paper is organized as follows. In Section~\ref{sec:prelim} we cover the requisite background from the theory of posets and formally define Kohnert posets and polynomials. Following this, in Section~\ref{sec:strucres} we establish some structural results relevant to identifying when Kohnert posets are shellable. Then in Sections~\ref{sec:hook} and~\ref{sec:key}, we apply the aforementioned structural results to give complete characterizations of those diagrams belonging to three different families which generate (EL-)shellable Kohnert posets. Section~\ref{sec:hook} focuses on those diagrams with at most one cell per column as well as those for which the first two rows are empty. In the more complicated case of diagrams containing cells within their first two rows and at least one column with more than one cell, we consider the special case of key diagrams. For such diagrams, in Section~\ref{sec:key} we find a complete characterization of those which generate (EL-)shellable Kohnert posets in the case that the poset is graded. In addition to the characterizations of (EL-)shellability, Sections~\ref{sec:hook} and~\ref{sec:key} also contain results concerning the polynomial consequences of (EL-)shellability for the associated families of diagrams. Finally, in Section~\ref{sec:ep}, we discuss directions for future research.

\section{Preliminaries}\label{sec:prelim}

In this section, we give the requisite preliminaries from the theory of posets and define our posets of interest.

\subsection{Posets}

Recall that a poset $(\mathcal{P},\preceq)$ consists of a set $\mathcal{P}$ along with a binary relation $\preceq$ between the elements of $\mathcal{P}$ which is reflexive, anti-symmetric, and transitive. When no confusion will arise, we simply denote a poset $(\mathcal{P}, \preceq)$ by $\mathcal{P}$. Two posets $\mathcal{P}$ and $\mathcal{Q}$ are \textit{isomorphic}, denoted $\mathcal{P}\cong \mathcal{Q}$, if there exists an order-preserving bijection $\mathcal{P}\to\mathcal{Q}$.  Ongoing, we assume that all posets are finite.

Let $\mathcal{P}$ be a poset and take $x,y\in\mathcal{P}$. If $x\preceq y$ and $x\neq y$, then we call $x\preceq y$ a \textbf{strict relation} and write $x\prec y$. Ongoing, we let $\le$ and $<$ denote the relation and strict relation, respectively, corresponding to the natural ordering on $\mathbb{Z}$. For $x,y\in\mathcal{P}$ satisfying $x\preceq y$, we set $[x,y]=\{z\in\mathcal{P}~|~x\preceq z\preceq y\}$ and treat $[x,y]$ as a poset with the ordering inherited from $\mathcal{P}$; that is, for $z_1,z_2\in [x,y]$, $z_1\prec_{[x,y]}z_2$ if and only if $z_1\prec_{\mathcal{P}}z_2$. If $x\prec y$ and there exists no $z\in \mathcal{P}$ satisfying $x\prec z\prec y$, then $x\prec y$ is a \textbf{covering relation}, denoted $x\precdot y$. Covering relations are used to define a visual representation of $\mathcal{P}$ called the \textbf{Hasse diagram} -- a graph whose vertices correspond to elements of $\mathcal{P}$ and whose edges correspond to covering relations (see Figure~\ref{fig:poset}). We say that $x\in\mathcal{P}$ is a minimal element (resp., maximal element) if there exists no $z\in\mathcal{P}$ such that $z\prec x$ (resp., $z\succ x$). If $\mathcal{P}$ has a unique minimal and maximal element, then we say that $\mathcal{P}$ is \textbf{bounded}. The poset $\mathcal{P}$ is called \textbf{ranked} if there exists a \textit{rank function}, i.e., a function $\rho:\mathcal{P}\to\mathbb{Z}_{\ge 0}$ such that
\begin{enumerate}
    \item if $x\prec y,$ then $\rho(x)<\rho(y),$ and
    \item if $x\prec y$ is a covering relation, then $\rho(y)=\rho(x)+1.$
\end{enumerate}

\begin{ex}\label{ex:poset}
Let $\mathcal{P}_1=\{1,2,3,4\}$ be the poset with $1\prec 2\prec 3,4,$ and let $\mathcal{P}_2=\{1,2,3,4,5\}$ be the poset with $1\prec 2\prec 4\prec 5$ and $1\prec 3\prec 5.$ The Hasse diagrams of $\mathcal{P}_1$ and $\mathcal{P}_2$ are illustrated in Figure~\ref{fig:poset}. Notice that $\mathcal{P}_1$ is ranked but not bounded, while $\mathcal{P}_2$ is bounded but not ranked.

\begin{figure}[H]
$$\begin{tikzpicture}
\def\Node{\node [circle, fill, inner sep=1.5pt]}

\Node[label=below:{\small 1}] (l1) at (-5,0){};
\Node[label=left:{\small 2}] (l2) at (-5,1){};
\Node[label=left:{\small 3}] (l3) at (-6,2){};
\Node[label=right:{\small 4}] (l4) at (-4,2){};

\draw (l1)--(l2)--(l3);
\draw (l2)--(l4);

\Node[label=below:{\small 1}] (r1) at (0,0){};
\Node[label=left:{\small 2}] (r2) at (-1,1){};
\Node[label=right:{\small 3}] (r3) at (1,1.5){};
\Node[label=left:{\small 4}] (r4) at (-1,2){};
\Node[label=above:{\small 5}] (r5) at (0,3){};

\draw (r1)--(r2)--(r4)--(r5)--(r3)--(r1);

\node at (-5,-1){(a)};
\node at (0,-1){(b)};
\end{tikzpicture}$$
\caption{Hasse diagrams of (a) $\mathcal{P}_1$ and (b) $\mathcal{P}_2$}
\label{fig:poset}
\end{figure}
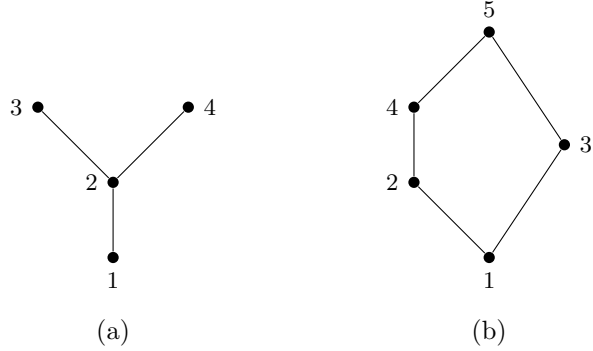
\end{ex}
 
A totally ordered subset of a poset $\mathcal{P}$ is called a \textbf{chain}. We call a chain $\mathcal{C}$ of $\mathcal{P}$ \textbf{maximal} if it is contained in no larger chains of $\mathcal{P}$, and we call $\mathcal{C}$ \textbf{saturated} if there does not exist $u\in\mathcal{P}\backslash \mathcal{C}$ and $s,t\in \mathcal{C}$ such that $s\prec u\prec t$ and $\mathcal{C}\cup\{u\}$ is a chain. Using the chains of a poset $\mathcal{P}$, one can define a simplicial complex $\Delta(\mathcal{P})$ associated with $\mathcal{P}$. Recall that a \textbf{(abstract) simplicial complex} $\Delta$ on a vertex set $V$, is a finite collection of subsets of $V$, called \textbf{faces}, such that $\tau\subseteq\sigma\in\Delta$ implies $\tau\in\Delta$. To define the simplicial complex $\Delta(\mathcal{P})$, we set $V=\mathcal{P}$ and $\Delta(\mathcal{P})=\{\text{chains of }\mathcal{P}\}$. 

\begin{ex}
The simplicial complexes associated with the posets $\mathcal{P}_1$ and $\mathcal{P}_2$ of Example~\ref{ex:poset} are illustarted in Figure~\ref{fig:simpex} \textup{(a)} and \textup{(b)}, respectively. Note that the vertices of the simplices are labelled by the corresponding elements of the posets.
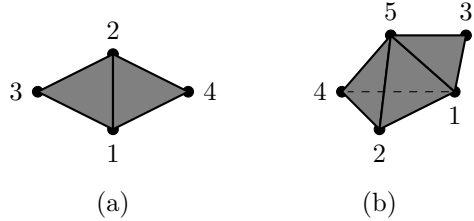
\begin{figure}[H]
    \centering
    $$\begin{tikzpicture}
	\node (1) at (0, 0) [circle, draw = black, fill = black, inner sep = 0.5mm, label=below:{1}]{};
	\node (2) at (-1, 0.5)[circle, draw = black, fill = black, inner sep = 0.5mm,label=left:{3}] {};
    \node (3) at (0, 1) [circle, draw = black, fill = black, inner sep = 0.5mm,label=above:{2}]{};
    \node (4) at (1, 0.5) [circle, draw = black, fill = black, inner sep = 0.5mm,label=right:{4}]{};
    \draw [fill=gray,thick] (0, 0)--(-1, 0.5)--(0, 1)--(0, 0);
    \draw [fill=gray,thick] (0, 0)--(1, 0.5)--(0, 1)--(0, 0);
    \node (5) at (0, -1) {(a)};
\end{tikzpicture}\quad\quad\quad\begin{tikzpicture}
	\node (2) at (0, 0) [circle, draw = black, fill = black, inner sep = 0.5mm, label=below:{2}]{};
	\node (1) at (1, 0.5)[circle, draw = black, fill = black, inner sep = 0.5mm,label=below:{1}] {};
    \node (3) at (1.15, 1.25) [circle, draw = black, fill = black, inner sep = 0.5mm,label=above:{3}]{};
    \node (4) at (-0.5, 0.5) [circle, draw = black, fill = black, inner sep = 0.5mm,label=left:{4}]{};
    \node (5) at (0.15, 1.25) [circle, draw = black, fill = black, inner sep = 0.5mm,label=above:{5}]{};
    \draw [fill=gray,thick] (-0.5, 0.5)--(0, 0)--(0.15, 1.25)--(-0.5, 0.5);
   \draw [fill=gray,thick] (1, 0.5)--(0, 0)--(0.15, 1.25)--(1, 0.5);
   \draw [fill=gray,thick] (1, 0.5)--(1.15, 1.25)--(0.15, 1.25)--(1, 0.5);
   \draw[dashed] (1)--(4);
    \node (6) at (0, -1) {(b)};
\end{tikzpicture}$$
    \caption{$\Delta(\mathcal{P})$}
    \label{fig:simpex}
\end{figure}
\end{ex}

\noindent
Given a simplicial complex $\Delta$, the dimension of a face $\sigma\in\Delta$ is defined as $\dim\sigma=|\sigma|-1$ and the dimension of $\Delta$ is defined by $\dim\Delta=\max_{\sigma\in \Delta}\dim\sigma$. If $\sigma\in\Delta$ satisfies $\dim\sigma=k$, then we refer to $\sigma$ as a \textbf{$\mathbf{k}$-face}; similarly, if $\dim\Delta=k$, we call $\Delta$ a \textbf{$\mathbf{k}$-complex}. Any face $\sigma\in\Delta$ generates a simplicial complex $\bar{\sigma}$ consisting of $\sigma$ and all of its subsets; simplicial complexes of this form are called \textbf{simplices}. A face $\sigma\in\Delta$ is called a \textbf{facet} if it is contained in no other face of $\Delta$. We say that $\Delta$ is \textbf{pure} if all facets of $\Delta$ have the same dimension. 

\begin{definition}
    A simplicial complex $\Delta$ is called \textbf{shellable} if its facets can be arranged into a total order $F_1,\hdots,F_t$ in such a way that the subcomplex $\left(\bigcup_{i=1}^{k-1}\bar{F}_i\right)\cap \bar{F}_k$ is pure and $(\dim F_k-1)$-dimensional for $2\le k\le t$. Such an ordering of facets is called a \textbf{shelling}.
\end{definition}

We call a poset $\mathcal{P}$ \textbf{shellable} (resp., \textbf{pure}) if $\Delta(\mathcal{P})$ is shellable (resp., pure). If a poset is finite, bounded, and pure, then we say that it is \textbf{graded}. As a consequence of the following result, one nice way to identify when a poset $\mathcal{P}$ is not shellable, i.e., when $\Delta(\mathcal{P})$ is not shellable, is by finding a non-shellable interval.

%[Corollary 3.1.9 (b), \textbf{\cite{WachsPT}}]

\begin{theorem}[Bj\"orner and Wachs \textbf{\cite{npshell2}}]\label{thm:shellinterval}
    Every interval of a shellable poset is shellable.
\end{theorem}

%via the ``$h$-vector" of $\Delta(\mathcal{P})$. The \textbf{$f$-vector} of a simplicial complex $\Delta$ is an $(\dim\Delta+1)$-dimensional vector $f(\Delta)=(f_0,f_1,\hdots,f_{\dim\Delta})$, where $f_i$ is the number of $i$-dimension faces of $\Delta$ for $0\le i\le \dim\Delta$. Now, using $f(\Delta)$ one can form the \textbf{$h$-vector} of $\Delta$ defined by $h(\Delta)=(h_0,h_1,\hdots,h_{\dim\Delta})$, where, setting $f_{-1}=1$, $$h_k=\sum_{i=0}^k(-1)^{k-i}\binom{\dim\Delta-i+1}{k-i}f_{i-1}.$$ The following result is a consequence of Theorem 6 in \textbf{\cite{hvec}} \bn{(SOMEONE SHOULD DOUBLE CHECK THIS)}.

%\begin{lemma}\label{lem:hvec}
%Let $\Delta$ be a $d$-dimensional simplicial complex. If $\Delta$ is shellable, then $h(\Delta)\in \mathbb{Z}_{\ge 0}^d$.
%\end{lemma}

%\begin{remark}
%    Lemma~\ref{lem:hvec} can be strengthened to give a complete characterization of graded shellable simplicial complexes in terms of their $h$-vectors. As such a strengthening is not used in this paper, for further details we direct interested readers to the following resources ??.
%\end{remark}

Extending the work of \textbf{\cite{Bjorner,BW1,BW2}}, the authors of \textbf{\cite{npshell,npshell2}} introduce a way of identifying if a bounded poset is shellable without referencing $\Delta(\mathcal{P})$. Given a poset $\mathcal{P}$, let $\mathcal{E}(\mathcal{P})=\{(x,y)\in\mathcal{P}\times\mathcal{P}~|~x\precdot y\}$, i.e., $\mathcal{E}(\mathcal{P})$ is the set of edges in the Hasse diagram of $\mathcal{P}$. An \textbf{edge labeling} of $\mathcal{P}$ is a map $\lambda:\mathcal{E}(\mathcal{P})\to\Lambda$, where $\Lambda$ is some poset. Given a saturated chain $\mathcal{C}=\{x_0\precdot x_1\precdot\cdots\precdot x_n\}$ of $\mathcal{P}$ with an edge labelling $\lambda$, we define the vector $$v(\mathcal{C},\lambda)=(\lambda(x_0,x_1),\lambda(x_1,x_2),\cdots, \lambda(x_{n-1},x_n))\in\mathbb{Z}^n$$ and call $\mathcal{C}$ \textbf{rising} if $\lambda(x_0,x_1)\le\lambda(x_1,x_2)\le\cdots\le \lambda(x_{n-1},x_n)$.

\begin{definition}\label{def:EL}
    A bounded poset is \textbf{EL-shellable} if it admits an edge labeling $\lambda:\mathcal{E}(\mathcal{P})\to \Lambda$ such that for every interval $[x,y]$ of $\mathcal{P}$
    \begin{enumerate}
        \item[\textup{1)}] there is a unique rising unrefinable chain $\mathcal{C}_{[x,y]}=\{x=x_0\precdot x_1\precdot\cdots\precdot x_n=y\}$ and
        \item[\textup{2)}] if $\widetilde{\mathcal{C}}$ is any other unrefinable chain between $x$ and $y$, then $v(\mathcal{C}_{[x,y]})$ is lexicographically less than $v\left(\widetilde{\mathcal{C}}\right)$.
       
    \end{enumerate}
\end{definition}

%[Theorem 5.8, \textbf{\cite{npshell}}]

\begin{theorem}[Bj\"orner and Wachs \textbf{\cite{npshell}}]\label{thm:ELshell}
    Let $\mathcal{P}$ be bounded. If $\mathcal{P}$ is a EL-shellable poset, then $\mathcal{P}$ is shellable.
\end{theorem}

Considering the definition, the following EL version of Theorem~\ref{thm:shellinterval} is immediate.

\begin{theorem}\label{thm:ELinterval}
Let $\mathcal{P}$ be bounded. Every interval of an EL-shellable poset is EL-shellable.
\end{theorem}

Having covered the necessary preliminaries of posets, we now move to defining our posets of interest.

\subsection{Kohnert posets}

As mentioned in the introduction, the underlying sets of Kohnert posets are certain collections of diagrams. Formally, a \textbf{diagram} is an array of finitely many cells in $\mathbb{N}\times\mathbb{N}$. An example diagram is illustrated in Figure~\ref{fig:diagram} below.

\begin{figure}[H]
    \centering
    $$\begin{tikzpicture}[scale=0.65]
  \node at (1.5, 2.5) {$\bigtimes$};
  \node at (0.5, 1.5) {$\bigtimes$};
  \node at (1.5, 1.5) {$\bigtimes$};
  \node at (2.5, 0.5) {$\bigtimes$};
  \draw (0,4)--(0,0)--(4,0);
  \draw (1,3)--(2,3)--(2,2)--(1,2)--(1,3);
  \draw (2,2)--(2,1)--(1,1)--(1,2)--(0,2);
  \draw (0,1)--(1,1);
  \draw (2,0)--(2,1)--(3,1)--(3,0);
\end{tikzpicture}$$
    \caption{Diagram}
    \label{fig:diagram}
\end{figure}

\noindent
We may also think of a diagram as the set of row/column coordinates of the cells defining it, where rows are labeled from bottom to top and columns from left to right. For example, if $D$ is the diagram of Figure~\ref{fig:diagram}, then $D=\{(1,3),(2,1),(2,2),(3,2)\}$. Consequently, if a diagram $D$ contains a cell in position $(r,c)$, then we write $(r,c)\in D$; otherwise, $(r,c)\notin D$.

\begin{remark}
    Ongoing, when illustrating diagrams with a particular form, it will prove helpful to decorate regions to indicate a particular structure. Other than describing the properties of regions in words \textup(usually in parentheses\textup), regions shaded gray represent empty regions containing no cells, and regions shaded with diagonal lines will represent regions that are arbitrary, i.e., the placement of cells can be arbitrary.
\end{remark}

Now, to any diagram $D$ we can apply what are called ``Kohnert moves" defined as follows. For $r>0$, applying a \textbf{Kohnert move} at row $r$ of $D$ results in the rightmost cell in row $r$ of $D$ moving to the first empty position below in the same column (if such a position exists), jumping over other cells as needed. If applying a Kohnert move at row $r>0$ of $D$ causes the cell in position $(r,c)\in D$ to move down to position $(r',c)$, forming the diagram $D'$, then we write $$D'=D\ldownarrow^{(r,c)}_{(r',c)}.$$ We let $KD(D)$ denote the set of all diagrams that can be obtained from $D$ by applying a sequence of Kohnert moves. For example, in Figure~\ref{fig:moves1} we illustrate the diagrams of $KD(D)$ for the diagram $D$ of Figure~\ref{fig:diagram}.

\begin{figure}[H]
$$\begin{tikzpicture}[scale=0.4]
    \draw[thick] (4,0)--(0,0)--(0,4);
        \draw (2,0)--(2,1)--(3,1)--(3,0);
        \draw (0,1)--(2,1)--(2,2)--(0,2);
        \draw (1,2)--(1,3)--(2,3)--(2,2);
        \draw (1,1)--(1,2);

        \node at (2.5,0.5){$\bigtimes$};
        \node at (0.5,1.5){$\bigtimes$};
        \node at (1.5,1.5){$\bigtimes$};
        \node at (1.5,2.5){$\bigtimes$};
\end{tikzpicture}\hspace{1cm}
\begin{tikzpicture}[scale=0.4]
    \draw[thick] (4,0)--(0,0)--(0,4);
        \draw (1,0)--(1,1)--(3,1)--(3,0);
        \draw (2,0)--(2,1);
        \draw (0,1)--(1,1)--(1,2)--(0,2);
        \draw (1,2)--(2,2)--(2,3)--(1,3)--(1,2);

        \node at (1.5,.5){$\bigtimes$};
        \node at (2.5,.5){$\bigtimes$};
        \node at (.5,1.5){$\bigtimes$};
        \node at (1.5,2.5){$\bigtimes$};
\end{tikzpicture}\hspace{1cm}
\begin{tikzpicture}[scale=0.4]
    \draw[thick] (4,0)--(0,0)--(0,4);
        \draw (0,1)--(3,1)--(3,0);
        \draw (1,1)--(1,0);
        \draw (2,1)--(2,0);
        \draw (1,2)--(2,2)--(2,3)--(1,3)--(1,2);

        \node at (.5,.5){$\bigtimes$};
        \node at (1.5,.5){$\bigtimes$};
        \node at (2.5,.5){$\bigtimes$};
        \node at (1.5,2.5){$\bigtimes$};
\end{tikzpicture}\hspace{1cm}
\begin{tikzpicture}[scale=0.4]
    \draw[thick] (4,0)--(0,0)--(0,4);
        \draw (1,0)--(1,1)--(3,1)--(3,0);
        \draw (0,1)--(1,1);
        \draw (2,0)--(2,1);
        \draw (0,2)--(2,2)--(2,1);
        \draw (1,1)--(1,2);

        \node at (1.5,.5){$\bigtimes$};
        \node at (2.5,.5){$\bigtimes$};
        \node at (.5,1.5){$\bigtimes$};
        \node at (1.5,1.5){$\bigtimes$};
\end{tikzpicture}\hspace{1cm}
\begin{tikzpicture}[scale=0.4]
    \draw[thick] (4,0)--(0,0)--(0,4);
        \draw (0,1)--(3,1)--(3,0);
        \draw (1,1)--(1,0);
        \draw (2,1)--(2,0);
        \draw (1,1)--(1,2)--(2,2)--(2,1);

        \node at (.5,.5){$\bigtimes$};
        \node at (1.5,.5){$\bigtimes$};
        \node at (2.5,.5){$\bigtimes$};
        \node at (1.5,1.5){$\bigtimes$};
\end{tikzpicture}$$
\caption{$KD(D).$}
\label{fig:moves1}
\end{figure}

\noindent
The sets $KD(D)$ form the underlying sets of Kohnert posets.

Given a diagram $D$, the authors of \textbf{\cite{KP2}} define an ordering on the elements of $KD(D)$ as follows (see also \textbf{\cite{KP3}}). For $D_1,D_2\in KD(D)$, we say $D_2\prec D_1$ if $D_2$ can be obtained from $D_1$ by applying some sequence of Kohnert moves. For a diagram $D$, we denote the corresponding poset on $KD(D)$ by $\mathcal{P}(D)$ and refer to it as the \textbf{Kohnert poset} associated to $D$. 

% Ongoing, for a diagram $D$ we let $Min(D)$ denote the collection of minimal elements of $\mathcal{P}(D)$, i.e., the collection of diagrams of $KD(D)$ which are invariant under all Kohnert moves.
 
In the sections that follow, we study $\mathcal{P}(D)$ for various collections of diagrams $D$ with our main concern being finding restrictions under which $\mathcal{P}(D)$ is shellable. One important family of diagrams considered below is the family of ``key" diagrams.

A diagram $D$ whose cells are left-justified is called a \textbf{key diagram}. Note that key diagrams are uniquely identified by the weak compositions corresponding to the sequences enumerating the number of cells in each row. Consequently, we denote a key diagram by $\mathbb{D}(\mathbf{a}),$ where $\mathbf{a}$ is the aforementioned weak composition. For example, the diagram in Figure~\ref{fig:swex} is the key diagram $\mathbb{D}(1,0,3,1,2).$

\begin{figure}[H]
    $$\begin{tikzpicture}[scale=0.5]
    \draw[thick] (6,0)--(0,0)--(0,6);
        \draw (0,1)--(1,1)--(1,0);
        \draw (0,2)--(3,2)--(3,3)--(0,3);
        \draw (1,2)--(1,3);
        \draw (2,2)--(2,3);
        \draw (0,4)--(1,4)--(1,3);
        \draw (0,5)--(2,5)--(2,4)--(1,4)--(1,5);

        \node at (.5,.5){$\bigtimes$};
        \node at (.5,2.5){$\bigtimes$};
        \node at (1.5,2.5){$\bigtimes$};
        \node at (2.5,2.5){$\bigtimes$};
        \node at (.5,3.5){$\bigtimes$};
        \node at (.5,4.5){$\bigtimes$};
        \node at (1.5,4.5){$\bigtimes$};
    \end{tikzpicture}$$
    \caption{Key diagram}
    \label{fig:swex}
\end{figure}
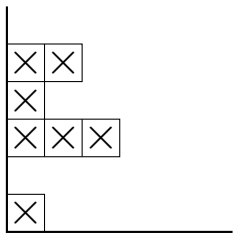

% With this notation, if a diagram $D$ has at most one cell in each column, we write $D=D(r_1,r_2,\dots,r_n),$ where $r_j$ is the row of $D$ containing the cell in column $j.$ As convention, we let $r_j=0$ if there is no cell in column $j$ of $D.$

Now, since in certain cases we will see that shellability of Kohnert posets has polynomial consequences, we briefly recall the definition of ``Kohnert polynomial". Given a diagram $D$, setting $$wt(D)=\prod_{i\ge 0}x_i^{\#\text{cells in row}~i~\text{of D}},$$ we define the \textbf{Kohnert polynomial} associated to $D$ as $\mathfrak{K}_D=\displaystyle\sum_{T\in KD(D)}wt(T)$.

\section{Structural Results}\label{sec:strucres}

In this section, we identify two necessary conditions for a diagram to generate a shellable Kohnert poset. Both conditions are in terms of avoiding certain subdiagrams that generate non-shellable intervals in the associated poset. Specifically, given a diagram that does not avoid one of the aforementioned subdiagrams, we are able to identify an interval that is isomorphic to a poset of the form described in Lemma~\ref{lem:nonshellp} below.

\begin{lemma}\label{lem:nonshellp}
Let $n_1,n_2>1$. If $\mathcal{P}=\{p_i^1\}_{i=1}^{n_1}\cup\{p_i^2\}_{i=1}^{n_2}\cup\{\hat{0},\hat{1}\}$ with 
\begin{itemize}
    \item $\hat{0}\prec p_1^1,p^2_1$,
    \item $p_i^1\prec p_{i+1}^1$ for $1\le i<n_1$,
    \item $p_i^2\prec p_{i+1}^2$ for $1\le i<n_2$, and
    \item $p_{n_1}^1,p_{n_2}^2\prec\hat{1}$,
\end{itemize}
then $\mathcal{P}$ is not shellable.
\end{lemma}
\begin{proof}
Note that $\Delta(\mathcal{P})$ contains exactly two facets $$F_1=\{\hat{0},p_1^1,\hdots,p_{n_1}^1,\hat{1}\}$$ and $$F_2=\{\hat{0},p_1^2,\hdots,p_{n_2}^2,\hat{1}\}.$$ Since $$\dim F_1-1=n_1\neq \dim \bar{F}_1\cap \bar{F}_2=1\neq n_2=\dim F_2-1,$$ by definition $\Delta(\mathcal{P})$ is not shellable.
\end{proof}

\begin{ex}\label{ex:hex}
Consider the poset $\mathcal{P}=\{1,2,3,4,5,6\}$ with $1\prec 2\prec 4\prec 6$ and $1\prec 3\prec 5\prec 6.$ The Hasse diagram and simplicial complex of $\mathcal{P}$ are illustrated in Figure~\ref{fig:hex} \textup(a\textup) and \textup(b\textup), respectively. It follows from Lemma~\ref{lem:nonshellp} \textup(with $n_1=n_2=2$\textup) that $\mathcal{P}$ is not shellable.

\begin{figure}[H]
$$\begin{tikzpicture}
    \node (1) at (0, 0) [circle, draw = black, fill = black, inner sep = 0.5mm, label=below:{1}]{};
    \node (2) at (-1, 1) [circle, draw = black, fill = black, inner sep = 0.5mm, label=left:{2}]{};
    \node (3) at (1, 1) [circle, draw = black, fill = black, inner sep = 0.5mm, label=right:{3}]{};
    \node (4) at (-1, 2) [circle, draw = black, fill = black, inner sep = 0.5mm, label=left:{4}]{};
    \node (5) at (1, 2) [circle, draw = black, fill = black, inner sep = 0.5mm, label=right:{5}]{};
    \node (6) at (0, 3) [circle, draw = black, fill = black, inner sep = 0.5mm, label=above:{6}]{};

    \draw (1)--(2)--(4)--(6)--(5)--(3)--(1);

    \node at (0,-1){(a)};
\end{tikzpicture}\hspace{1in} \begin{tikzpicture}[scale=2]
	\node (1) at (0, 0) [circle, draw = black, fill = black, inner sep = 0.5mm, label=below:{1}]{};
	\node (6) at (0,1)[circle, draw = black, fill = black, inner sep = 0.5mm, label=above:{6}] {};
    \node (5) at (1,0) [circle, draw = black, fill = black, inner sep = 0.5mm, label=right:{5}]{};
    \node (3) at (0.75,-0.4) [circle, draw = black, fill = black, inner sep = 0.5mm, label=below:{3}]{};
    \node (4) at (-1,0) [circle, draw = black, fill = black, inner sep = 0.5mm, label=left:{4}]{};
    \node (2) at (-0.75,-0.4) [circle, draw=black, fill=black, inner sep=0.5mm, label=below:{2}]{};

    \draw [fill=gray, thick] (0,0)--(0.75,-0.4)--(0,1)--(0,0);
    \draw [fill=gray, thick] (0,0)--(-0.75,-0.4)--(0,1)--(0,0);
    \draw [fill=gray, thick] (0.75,-0.4)--(1,0)--(0,1)--(0.75,-0.4);
    \draw [fill=gray, thick] (-0.75,-0.4)--(-1,0)--(0,1)--(-0.75,-0.4);
    \draw [thick, dashed] (4)--(5);

    \node at (0,-1){(b)};
\end{tikzpicture}$$
\caption{(a) Hasse diagram of $\mathcal{P}$ and (b) $\Delta(\mathcal{P})$}
\label{fig:hex}
\end{figure}
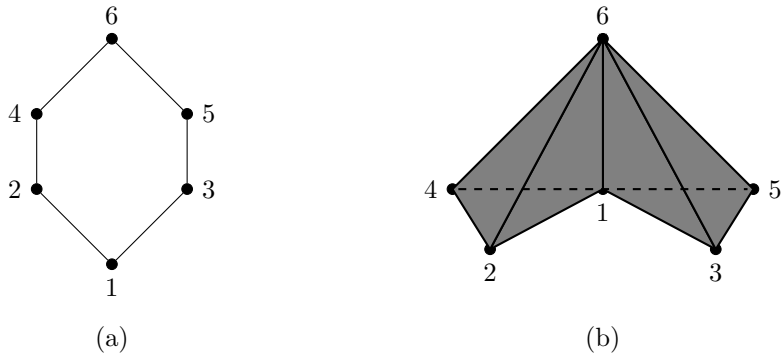
\end{ex}

Now, to aid in identifying the non-shellable intervals in the Kohnert posets of Propositions~\ref{prop:strucasc} and~\ref{prop:strucblock} below, we require the following lemma.

\begin{lemma}\label{lem:strucprophelp}
    Let $D$ be a diagram and $D_1,D_2\in \mathcal{P}(D)$ satisfy $D_2\prec D_1$. If there exist $0<c_1<c_2<\cdots<c_n$ and $0<r^i_1<r^i_2$ for $1\le i\le n$ such that $$S=D_1\backslash\bigcup_{i=1}^n\{(\tilde{r},c_i)~|~r^i_1\le \tilde{r}\le r^i_2\}=D_2\backslash\bigcup_{i=1}^n\{(\tilde{r},c_i)~|~r^i_1\le \tilde{r}\le r^i_2\},$$ then $$\tilde{D}\backslash\bigcup_{i=1}^n\{(\tilde{r},c_i)~|~r^i_1\le \tilde{r}\le r^i_2\}=S$$ for all $\tilde{D}\in [D_2,D_1]$.
\end{lemma}
\begin{proof}
    Let $R=\bigcup_{i=1}^n\{(\tilde{r},c_i)~|~r^i_1\le \tilde{r}\le r^i_2\}$. Assume for a contradiction that there exists $D^*\in [D_2,D_1]$ such that $D^*\backslash R\neq S$. Then there are two cases.
    \bigskip

    \noindent
    \textbf{Case 1:} There exists $c^*,r^*>0$ such that $(r^*,c^*)\in D^*\backslash R$ and $(r^*,c^*)\notin S$. Since $D^*\preceq D_1$, $D_1\backslash R=S,$ and nontrivial Kohnert moves result in cells moving to lower rows, it must be the case that $$|\{(\tilde{r},c^*)\in D_1~|~\tilde{r}>r^*\}|>|\{(\tilde{r},c^*)\in D^*~|~\tilde{r}>r^*\}|.$$ Moreover, since $D_2\prec D_1$, $D_1\backslash R=D_2\backslash R$, and $(r^*,c^*)\notin R$, it follows that $$|\{(\tilde{r},c^*)\in D_2~|~\tilde{r}>r^*\}|=|\{(\tilde{r},c^*)\in D_1~|~\tilde{r}>r^*\}|>|\{(\tilde{r},c^*)\in D^*~|~\tilde{r}>r^*\}|,$$ i.e., there are more cells strictly above row $r^*$ in column $c^*$ of $D_2$ than in $D^*$. As the number of cells in a given column above a given row can only decrease upon applying Kohnert moves, it follows that $D_2\not\preceq D^*$, a contradiction.
    \bigskip

    \noindent
    \textbf{Case 2:} There exists $c^*,r^*>0$ such that $(r^*,c^*)\notin D^*\backslash R$ and $(r^*,c^*)\in S$. Since $D^*\preceq D_1$, $D_1\backslash R=S,$ and nontrivial Kohnert moves result in cells moving to lower rows, it must be the case that $$|\{(\tilde{r},c^*)\in D_1~|~\tilde{r}< r^*\}|<|\{(\tilde{r},c^*)\in D^*~|~\tilde{r}<r^*\}|.$$ As in Case 1, we may conclude further that $$|\{(\tilde{r},c^*)\in D_2~|~\tilde{r}< r^*\}|=|\{(\tilde{r},c^*)\in D_1~|~\tilde{r}< r^*\}|<|\{(\tilde{r},c^*)\in D^*~|~\tilde{r}<r^*\}|,$$ i.e., there are less cells strictly below row $r^*$ in column $c^*$ of $D_2$ than in $D^*$. As the number of cells in a given column below a given row can only increase upon applying Kohnert moves, it follows that $D_2\not\preceq D^*$, a contradiction.
\end{proof}

\begin{remark}
    Lemma~\ref{lem:strucprophelp} is a slight strengthening of Lemma 3.3 in \textup{\textbf{\cite{KPoset1}}}.
\end{remark}

\begin{corollary}\label{cor:cover}
    Let $D$ be a diagram. Suppose that for $D_1,D_2\in \mathcal{P}(D)$ there exists $1\le j<r$ and $c>0$ such that $D_2\prec D_1$ and $D_2=D_1\ldownarrow^{(r,c)}_{(r-j,c)}$. Then $D_2\precdot D_1$ if and only if for each $\tilde{r}$ satisfying $r-j<\tilde{r}<r$ there exists $\tilde{c}>c$ such that $(\tilde{r},\tilde{c})\in D_1$.
\end{corollary}
\begin{proof}
    First, assume that $D_2\precdot D_1$. Then $D_2$ must be formed from $D_1$ by applying a single Kohnert move at row $r$ so that $(\tilde{r},c)\in D_1$ for $r-j+1\le \tilde{r}\le r$. Now, if there exists $\widehat{r}$ satisfying $r-j<\widehat{r}<r$ and $(\widehat{r},\tilde{c})\notin D_1$ for all $\tilde{c}>c$, then applying a single Kohnert move at row $\widehat{r}$ of $D_1$ results in $D_{1.5}=D_1\ldownarrow^{(\widehat{r},c)}_{(r-j,c)}\in \mathcal{P}(D)$. Moreover, applying a Kohnert move at row $r$ of $D_{1.5}$ results in $D_{1.5}\ldownarrow^{(r,c)}_{(\widehat{r},c)}=D_2$, i.e., $D_2\prec D_{1.5}\prec D_1$, contradicting our assumption that $D_2\precdot D_1$. Thus, for each $\tilde{r}$ satisfying $r-j<\tilde{r}<r$, there exists $\tilde{c}>c$ such that $(\tilde{r},\tilde{c})\in D_1$.
    
    Now, assume that, for each $\tilde{r}$ satisfying $r-j<\tilde{r}<r$, there exists $\tilde{c}>c$ such that $(\tilde{r},\tilde{c})\in D_1$. Take $D_{1.5}\in [D_2,D_1]$ such that $D_{1.5}\precdot D_1$. Assume that $D_{1.5}$ can be formed from $D_1$ by applying a Kohnert move at row $\widehat{r}$. Note that $D_1$ and $D_2$ differ only in positions $(r,c)$ and $(r-j,c)$. Thus, if $R=\{(\tilde{r},c)~|~r-j\le \tilde{r}\le r\}$, then $D_1\backslash R=D_2\backslash R=D_{1.5}\backslash R$ by Lemma~\ref{lem:strucprophelp}. Since, for each $\tilde{r}$ satisfying $r-j<\tilde{r}<r$, there exists $\tilde{c}>c$ such that $(\tilde{r},\tilde{c})\in D_1$, it follows that $\widehat{r}=r$ or $r-j$; but applying a Kohnert move at row $r$ of $D_1$ results in $D_2$, while applying a Kohnert move at row $r-j$ either does nothing or affects a cell in a column $\tilde{c}<c$, contradicting $D_1\backslash R=D_{1.5}\backslash R$. Therefore, we may conclude that $D_{1.5}=D_2$, i.e., $D_2\precdot D_1$. The result follows. 
\end{proof}

% \begin{corollary}\label{cor:cover}
%     Let $D$ be a diagram. If $D_1,D_2\in KD(D)$ satisfy $D_2\prec D_2$ and there exists $r>1$ and $c>0$ such that $D_2=D_1\ldownarrow^{(r,c)}_{(r-1,c)}$, then $D_2\precdot D_1$.
% \end{corollary}
% \begin{proof}
%     Take $D_{1.5}\in [D_2,D_1]$. As $D_1\backslash\{(r,c),(r-1,c)\}=D_2\backslash\{(r,c),(r-1,c)\}=S$, applying Lemma~\ref{lem:strucprophelp} it follows that $S=D_{1.5}\backslash \{(r,c),(r-1,c)\}$. Thus, since Kohnert moves preserve the number of cells in a diagram, it follows that $$|D_{1.5}\cap \{(r,c),(r-1,c)\}|=|D_{1.5}\backslash S|=|D_1\backslash S|=|D_1\cap\{(r,c),(r-1,c)\}|=1.$$ Consequently, $D_{1.5}\cap \{(r,c),(r-1,c)\}=\{(r,c)\}$ or $\{(r-1,c)\}$, i.e., $D_{1.5}=D_1$ or $D_{1.5}=D_2$.
% \end{proof}

With Lemma~\ref{lem:strucprophelp} in hand, we can now prove Propositions~\ref{prop:strucasc} and~\ref{prop:strucblock} which provide necessary conditions for a diagram to generate a shellable Kohnert poset.

\begin{prop}\label{prop:strucasc}
    Let $D$ be a diagram. Suppose that there exists $D^*\in \mathcal{P}(D)$ and $r,c_1,c_2\in\mathbb{N}$ such that 
    \begin{itemize}
        \item[\textup{(i)}] $c_1<c_2$,
        \item[\textup{(ii)}] $(r+1,c_1),(r+2,c_2)\in D^*$,
        \item[\textup{(iii)}] $(r+2,\tilde{c}),(r,c_2)\notin D^*$ for $\tilde{c}>c_2$, and
        \item[\textup{(iv)}] $(r+1,\tilde{c}),(r,c_1)\notin D^*$ for $\tilde{c}>c_1$.
    \end{itemize}
    Then $\mathcal{P}(D)$ is not shellable. \textup(See Figure~\ref{fig:shelllem1} for an illustration of $D^*$.\textup)
\end{prop}

    \begin{figure}[H]
    \centering
    $$\scalebox{0.9}{\begin{tikzpicture}[scale=0.65]
    \filldraw[draw=lightgray,fill=lightgray] (0,0) rectangle (1,1);
    \filldraw[draw=lightgray,fill=lightgray] (3,0) rectangle (4,1);
    \filldraw[draw=lightgray,fill=lightgray] (1,1) rectangle (5.5,2);
    \filldraw[draw=lightgray,fill=lightgray] (4,2) rectangle (5.5,3);
    \node at (0.5, 1.5) {$\bigtimes$};
    \node at (3.5, 2.5) {$\bigtimes$};
  \draw (-1,4.5)--(-1,-1)--(6,-1);
  \draw (0,1)--(1,1)--(1,2)--(0,2)--(0,1);
  \draw (3,2)--(4,2)--(4,3)--(3,3)--(3,2);
    \path[pattern=north west lines] (-1,-1)--(5.5,-1) -- (5.5,1)--(4,1)--(4,0)--(3,0)--(3,1)--(1,1)--(1,0)--(0,0)--(0,2)--(3,2)--(3,3)--(5.5,3)--(5.5, 4)--(-1,4)-- cycle;
    \node at (-2,0.5) {$r$};
    \node at (-2,1.5) {$r+1$};
    \node at (-2,2.5) {$r+2$};
    \node at (0.5,-1.5) {$c_1$};
    \node at (3.5,-1.5) {$c_2$};
\end{tikzpicture}}$$
    \caption{Subdiagrams described in Proposition~\ref{prop:strucasc}}
    \label{fig:shelllem1}
\end{figure}
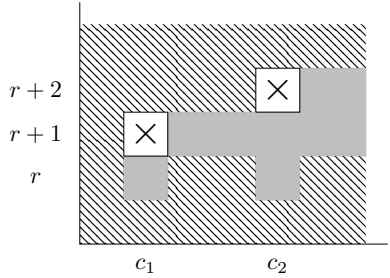

\begin{proof}
Considering Theorem~\ref{thm:shellinterval}, it suffices to show that $\mathcal{P}(D)$ contains an interval which is not shellable. Let $$\widehat{D}=\left(D^*\setminus\{(r+1,c_1),(r+2,c_2)\}\right)\cup\{(r,c_1),(r,c_2)\}.$$ Note that $\widehat{D}\in \mathcal{P}(D)$ since $\widehat{D}$ can be formed from $D^*$ by applying a Kohnert move at row $r+2$ of $D^*$ followed by two Kohnert moves at row $r+1$. We claim that the interval $[\widehat{D},D^*]$ is not shellable. To establish the claim, we first determine the elements of $[\widehat{D},D^*]$.

Note that $D^*$ and $\widehat{D}$ differ only in positions $\{(r+1,c_1),(r+2,c_2),(r,c_1),(r,c_2)\}$; that is, letting $R=\{(\tilde{r},c_1)~|~r\le \tilde{r}\le r+1\}\cup \{(\tilde{r},c_2)~|~r\le \tilde{r}\le r+2\}$, we have that $D^*\backslash R=\widehat{D}\backslash R=S$. Thus, applying Lemma~\ref{lem:strucprophelp}, it follows that $$\tilde{D}\backslash R=S$$ for all $\tilde{D}\in [\widehat{D},D^*]$. Note that this implies that all diagrams $\tilde{D}\in [\widehat{D},D^*]$ must be formed from $D^*$ by applying sequences of Kohnert moves at rows $r+1$ or $r+2$. In particular, it is straightforward to verify, keeping Lemma~\ref{lem:strucprophelp} in mind, that the only diagrams contained in $[\widehat{D},D^*]$ are $$D^*,\quad\quad D^1_1=D^*\ldownarrow^{(r+2,c_2)}_{(r+1,c_2)},\quad\quad D^1_2=D^1_1\ldownarrow^{(r+1,c_2)}_{(r,c_2)},\quad\quad D^2_1=D^*\ldownarrow^{(r+1,c_1)}_{(r,c_1)},$$ $$D_2^2=D_1^2\ldownarrow^{(r+2,c_2)}_{(r+1,c_2)},\quad\quad\text{and}\quad\quad \widehat{D}.$$ Now, to determine how the diagrams of $\left[\widehat{D},D^*\right]$ are related, note that we can form $\widehat{D}$ from $D^*$ by either applying in succession
\begin{enumerate}
    \item[1)] a Kohnert move at row $r+2$ followed by two Kohnert moves at row $r+1$, or
    \item[2)] a Kohnert move at row $r+1$, a Kohnert move at row $r+2$, then a Kohnert move at row $r+1$.
\end{enumerate}
The first sequence of Kohnert moves described above corresponds to the chain $$\widehat{D}\precdot D^1_2\precdot D^1_1 \precdot D^*$$ while the second sequence corresponds to $$\widehat{D}\precdot D^2_2\precdot D^2_1 \precdot D^*.$$ The fact that all relations in the two chains are covering relations follows from Corollary~\ref{cor:cover}. We claim that no other relations exist between the elements of $[\widehat{D},D^*]$. To see this, first note that since $D_1^2,D_1^1\precdot D^*$ and $D_1^2\neq D_1^1$, it follows that $D_1^1$ and $D_1^2$ are unrelated in $[\widehat{D},D^*]$. Similarly, since $\hat{D}\precdot D_2^1,D_2^2$ and $D_2^1\neq D_2^2$, it follows that $D_2^1$ and $D_2^2$ are unrelated in $[\widehat{D},D^*]$. It remains to consider relations between the elements $D_1^1$ and $D_2^2$, and the elements $D_1^2$ and $D_2^1$. Starting with $D_1^1$ and $D_2^2$, note that since $D_1^1\succ\widehat{D}\precdot D_2^2$, if $D_1^1$ and $D_2^2$ are related, then $D_2^2\prec D_1^1$. Now, since $D_1^1\backslash\{(r+1,c_1),(r,c_1)\}=D_2^2\backslash\{(r+1,c_1),(r,c_1)\}$, if $D_2^2\prec D_1^1$, then applying Lemma~\ref{lem:strucprophelp} it follows that there exists a sequence of Kohnert moves at rows $r$ and $r+1$ which takes $D_1^1$ to $D_2^2$. Evidently, such a sequence of Kohnert moves does not exist. Consequently, $D_1^1$ and $D_2^2$ are unrelated in $[\widehat{D},D^*]$. Moving to $D_1^2$ and $D_2^1$, since $D_1^2\succ \widehat{D}\precdot D_2^1$, if $D_1^2$ and $D_2^1$ are related, then $D_2^1\prec D_1^2$; but $$D_1^2\backslash \{(r+2,c_2),(r+1,c_2),(r,c_2)\}=\widehat{D}\backslash \{(r+2,c_2),(r+1,c_2),(r,c_2)\}\neq D_2^1\backslash \{(r+2,c_2),(r+1,c_2),(r,c_2)\},$$ so that applying Lemma~\ref{lem:strucprophelp} we have $D_2^1\notin [\widehat{D},D_1^2]$. Thus, since $\widehat{D}\prec D_2^1$, it follows that $D_2^1\not\prec D_1^2$, i.e., $D_1^2$ and $D_2^1$ are unrelated in $[\widehat{D},D^*]$, establishing the claim.

Now, let $\mathcal{S}$ be the poset of Example~\ref{ex:hex}, and define the map  $f:[\widehat{D},D^*]\to\mathcal{S}$ by $f(\widehat{D})=1$, $f(D^1_2)=2,$ $f(D^1_1)=3,$ $f(D^2_2)=4,$ $f(D_1^2)=5$, and $f(D^*)=6$. Considering our work above, it follows that $f$ forms an order-preserving bijection. Therefore, $[\widehat{D},D^*]$ is not shellable. The result follows.
\end{proof}

\begin{prop}\label{prop:strucblock}
Let $D$ be a diagram. Suppose that there exists $D^*\in \mathcal{P}(D)$ such that for some $r,r^*,c,c^*\in\mathbb{N}$ satisfying $c< c^*-1$ and $r<r^*-1$ we have
\begin{itemize}
    \item[\textup{(i)}] $(r^*,c^*)\in D^*$ and $(r^*, \tilde{c})\notin D^*$ for all $\tilde{c}>c^*$;
    %, i.e., the cell in row $r^*$ and column $c^*$ is the rightmost cell in its row; 
    \item[\textup{(ii)}] $|\{(r^*,\tilde{c})\in D^*~|~c<\tilde{c}<c^*\}|>0$;
    %, i.e., there are at least three cells contained in row $r^*$ and columns $c$ through $c^*$; \rn{(we haven't said yet that $(r^*,c)\in D,$ so this condition only gives us two cells in row $r^*$ between $c$ and $c^*.$)}
    \item[\textup{(iii)}] for $r < \tilde{r}\le r^*$, $(\tilde{r},c)\in D^*$;
    %, i.e., the cells between rows $r$ and $r^*$ in column $c$ are in the diagram as well as the cell in row $r^*$ and column $c$;
    \item[\textup{(iv)}] $(r,c),(r^*-1,\tilde{c})\notin D^*$ for $\tilde{c}>c$; and 
    \item[\textup{(v)}] for each $\tilde{r}$ satisfying $r<\tilde{r}<r^*-1$ there exists $\tilde{c}>c$ such that $(\tilde{r},\tilde{c})\in D^*$.
    %, i.e., in every row between rows $r$ and $r^*-1$ there is a cell in a column to the right of column $c$.
\end{itemize}
Then $\mathcal{P}(D)$ is not shellable. \textup(See Figure~\ref{fig:strucblock} for an illustration of $D^*$.\textup)
\end{prop}

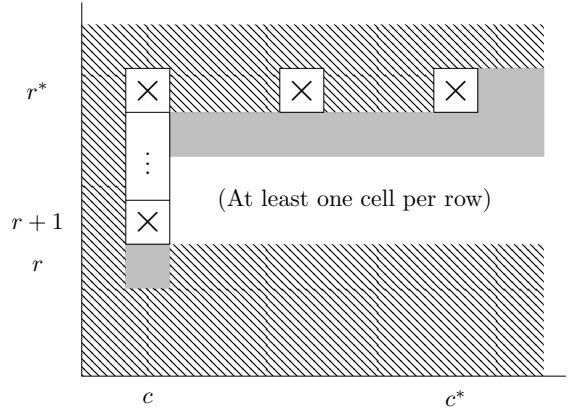
\begin{figure}[H]
    \centering
    $$\scalebox{0.9}{\begin{tikzpicture}[scale=0.65]
    \node at (8.5, -0.5) {$c^*$};
    \node at (1.5,6.5) {$\bigtimes$};
    \node at (8.5,6.5) {$\bigtimes$};
    \node at (-1,6.5) {$r^*$};
    \node at (-1,3.5) {$r+1$};
    \node at (-1,2.5) {$r$};
    \node at (1.5,5) {$\vdots$};
    \node at (1.5,3.5) {$\bigtimes$};
    \node at (5,6.5) {$\bigtimes$};
    \node at (1.5, -0.5) {$c$};
    \node at (6.2, 4) {(At least one cell per row)};
    \filldraw[draw=lightgray,fill=lightgray] (1,2) rectangle (2,3);
    \filldraw[draw=lightgray,fill=lightgray] (9,6) rectangle (10.5,7);
    \filldraw[draw=lightgray,fill=lightgray] (2,5) rectangle (10.5,6);
    \draw (0,8.5)--(0,0)--(11,0);
    \draw (8,6)--(9,6)--(9,7)--(8,7)--(8,6);
    \draw (4.5,6)--(5.5,6)--(5.5,7)--(4.5,7)--(4.5,6);
    \draw (1,3)--(2,3)--(2,7)--(1,7)--(1,3);
    \draw (1,4)--(2,4);
    \draw (1,6)--(2,6);
    \path[pattern=north west lines] (0,0)--(10.5,0)--(10.5,3)--(2,3)--(2,2)--(1,2) -- (1,7) -- (2,7) -- (2,6) -- (4.5,6) -- (4.5,7) -- (5.5,7) -- (5.5,6) -- (8,6)--(8,7) -- (10.5,7) -- (10.5,8) -- (0,8) -- cycle;
\end{tikzpicture}}$$
    \caption{Subdiagrams described in Proposition~\ref{prop:strucblock}}
    \label{fig:strucblock}
\end{figure}

\begin{proof}
As in the proof of Proposition~\ref{prop:strucasc}, we show that there exists a diagram $\widehat{D}\in \mathcal{P}(D)$ such that $[\widehat{D}, D^*]$ is isomorphic to a poset of the form described in Lemma~\ref{lem:nonshellp} so that, by Theorem~\ref{thm:shellinterval}, $\mathcal{P}(D)$ is not shellable. To define $\widehat{D}$, let
$$C= \left\{\tilde{c}~|~ c\leq \tilde{c} \leq c^*\ \text{ and } \ (r^*,\tilde{c})\in D_{max} \right\} = \left\{ c=c_0 < c_1 < \cdots < c_{m-1} = c^*\right\},$$ where, by assumption, $m\ge 3$.
Then $\widehat{D}$ is the diagram obtained from $D^*$ by applying $m$ Kohnert moves at row $r^*$; that is, $\widehat{D}$ is the diagram obtained from $D^*$ by moving the rightmost $m-1$ cells in row $r^*$ down to row $r^*-1$, and moving the $m^{th}$ cell from right to left in row $r^*$ down to row $r$. 

First, to determine the elements of $[\widehat{D},D^*]$, consider the following two chains from $\widehat{D}$ to $D^*$, both defined by the sequences of Kohnert moves applied to form $\widehat{D}$ from $D^*$.
\begin{itemize}
    \item[1)] Form the chain $\mathcal{C}_1$ by applying $m$ Kohnert moves at row $r^*$, i.e.,
    $$ 
    \mathcal{C}_1: \qquad D^* \succ D^1_1  \succ 
    D^1_2 \succ \cdots \succ D^1_{m-1} \succ \widehat{D},
    $$
    where 
    $D^1_1 := D^* \ldownarrow^{(r^*,c_{m-1})}_{(r^*-1,c_{m-1})}$, $D^1_{i+1}:= D^1_{i} \ldownarrow^{(r^*,c_{m-i-1})}_{(r^*-1,c_{m-i-1})}$ for $1\leq i \leq m-2$, and $\widehat{D} = D^1_{m-1} \ldownarrow^{(r^*,c_0)}_{(r,c_0)}$.

    \item[2)] Form the chain $\mathcal{C}_2$ by applying one Kohnert move at row $r^*-1$ followed by $m$ Kohnert moves at row $r^*$, i.e.,
      $$ 
    \mathcal{C}_2:\qquad D^* \succ D_1^2  \succ 
    D_2^2 \succ \cdots \succ D_{m}^2 \succ \widehat{D},
    $$
    where
 $D_1^2 := D^* \ldownarrow^{(r^*-1,c_0)}_{(r,c_0)}$, $D_{i+1}^2:= D_{i}^2 \ldownarrow^{(r^*,c_{m-i})}_{(r^*-1,c_{m-i})}$ for $1\leq i \leq m-1$, and $\widehat{D} = D_{m}^2\ldownarrow^{(r^*,c_0)}_{(r^*-1,c_0)}$.
\end{itemize}
We claim that the elements contained in $\mathcal{C}_1$ and $\mathcal{C}_2$ constitute the elements of $[\widehat{D},D^*]$. To see this, set $R=\{(\tilde{r},c)~|~r\le \tilde{r}\le r^*\}\cup\bigcup_{i=1}^{m-1}\{(r^*,c_i),(r^*-1,c_i)\}$ and note that, applying Lemma~\ref{lem:strucprophelp}, we have $D^*\backslash R=\widehat{D}\backslash R=\tilde{D}\backslash R$ for all $\tilde{D}\in [\widehat{D},D^*]$. Thus, every $\tilde{D}\in [\widehat{D},D^*]$ must satisfy property (v) of $D^*$. Consequently, applying a nontrivial Kohnert move to $\tilde{D}\in [\widehat{D},D^*]$ at any row other than $r^*-1$ or $r^*$ must affect the position of a cell outside of $R$; but, since $D^*\backslash R=\widehat{D}\backslash R=\tilde{D}\backslash R$, this implies that $\tilde{D}\in [\widehat{D},D^*]$ must be formed from $D^*$ by a sequence of Kohnert moves at rows $r^*-1$ and $r^*$. Under these restrictions, it is straightforward to verify the claim.

Now, we determine the relations defining the poset $[\widehat{D},D^*]$. Applying Corollary~\ref{cor:cover}, it follows that $$\widehat{D}\precdot D^1_{m-1}\precdot\cdots\precdot D^1_1\precdot D^*$$ and $$\widehat{D}\precdot D^2_{m}\precdot\cdots\precdot D_1^2\precdot D^*.$$ It remains to consider relations between $D^1_i$ and $D_j^2$ for $1\le i\le m-1$ and $1\le j\le m$. We show that there are no such relations. Let $$R^1_i=\{(r^*,c_k),(r^*-1,c_k)~|~m-i\le k\le m-1\}$$ for $1\le i\le m-1$. Since $$(r,c)\notin D^*\backslash R^1_i=D^1_i\backslash R^1_i\neq D_j^2\backslash R^1_i\ni (r,c)$$ for $1\le i\le m-1$ and $1\le j\le m$, it follows from Lemma~\ref{lem:strucprophelp} that $D_j^2\notin [D^1_i,D^*]$ for $1\le i\le m-1$ and $1\le j\le m$; that is, $D^1_i\not\prec D_j^2$ for $1\le i\le m-1$ and $1\le j\le m$. Consequently, if $D_j^2$ and $D^1_i$ are related for some choice of $1\le i\le m-1$ and $1\le j\le m$, then $D_j^2\prec D^1_i$. For a contradiction, assume that $D_j^2\prec D^1_i$ for some choice of $1\le i\le m-1$ and $1\le j\le m$. Then $D_m^2\prec D^1_i$. Letting $$R^2_i=\{(r^*,c_k),(r^*-1,c_k)~|~1\le k< m-i\}\cup\{(\tilde{r},c)~|~r\le \tilde{r}\le r^*-1\}$$ for $1\le i\le m-1$, we have that $$D_m^2\backslash R^2_i=D^1_i\backslash R^2_i$$ for $1\le i\le m-1$. Thus, since $D^1_i$ satisfies property (v) of $D^*$ as noted above, it follows from Lemma~\ref{lem:strucprophelp} that all diagrams in $[D_m^2,D^1_i]$ -- in particular, $D_m^2$ -- must be formed from $D^1_i$ by applying sequences of Kohnert moves at rows $r^*$ or $r$; but applying a Kohnert move at row $r^*$ results in $D_{min}$ which is not contained in $[D_m^2,D^1_i]$, while applying a Kohnert move at row $r$ either does nothing or affects a cell in some column $\tilde{c}<c$, resulting in a diagram not contained in $[D_m^2,D^1_i]$, by Lemma~\ref{lem:strucprophelp}. Therefore, $D_m^2\not\prec D^1_i$ and we conclude that there are no relations between $D^1_i$ and $D_j^2$ for $1\le i\le m-1$ and $1\le j\le m$.

Finally, above we showed that the poset $[\widehat{D},D^*]$ is completely defined by the relations $$D^* \succ D^1_1  \succ 
    D^1_2 \succ \cdots \succ D^1_{m-1} \succ \widehat{D}$$ and $$D^* \succ D_1^2  \succ 
    D_2^2 \succ \cdots \succ D_{m}^2 \succ \widehat{D}.$$ Let $\mathcal{P}_m$ be a poset of the form described in Lemma~\ref{lem:nonshellp} with $n_1=m-1$ and $n_2=m$, i.e., $\mathcal{P}_m=\{p_i^1\}_{i=1}^{m-1}\cup\{p_i^2\}_{i=1}^{m}\cup\{\hat{0},\hat{1}\}$ with 
\begin{itemize}
    \item $\hat{0}\prec p_1^1,p^2_1$,
    \item $p_i^1\prec p_{i+1}^1$ for $1\le i<m-1$,
    \item $p_i^2\prec p_{i+1}^2$ for $1\le i<m$, and
    \item $p_{m-1}^1,p_{m}^2\prec\hat{1}$.
\end{itemize}
Define the map $f:[\widehat{D},D^*]\to\mathcal{P}_m$ by $f(\widehat{D})=\hat{0}$, $f(D^1_i)=p_{m-i}^1$ for $1\le i\le m-1$, $f(D_j^2)=p_{m-j+1}^2$ for $1\le j\le m$, and $f(D^*)=\hat{1}$. Considering our work above, it follows that $f$ is an order-preserving bijection. Therefore, $[\widehat{D},D^*]$ is not shellable. The result follows.
\end{proof}

Ongoing, we will not require the full strength of Proposition~\ref{prop:strucblock}, but instead, we make use of the following special case.

\begin{corollary}\label{cor:strucblock}
    Let $D$ be a diagram. Suppose that there exists $D^*\in \bl{\mathcal{P}}(D)$ such that for some $r,c_1,c_2\in\mathbb{N}$ satisfying $c_1<c_2$ we have
    \begin{itemize}
        \item[\textup{(i)}] $(r+1,c_1),(r+2,c_1),(r+2,c_2)\in D^*$,
        \item[\textup{(ii)}] $|\{(r+2,\tilde{c})\in D^*~|~c_1<c<c_2\}|>0$
        \item[\textup{(iii)}] $(r+2,\tilde{c})\notin D^*$ for $\tilde{c}>c_2$,
        \item[\textup{(iv)}] $(r+1,\tilde{c})\notin D^*$ for $\tilde{c}>c_1$, and
        \item[\textup{(v)}] $(r,c_1)\notin D^*$.
    \end{itemize}
     Then $\mathcal{P}(D)$ is not shellable. \textup(See Figure~\ref{fig:corstrucblock} for an illustration of $D^*$.\textup)
\end{corollary}

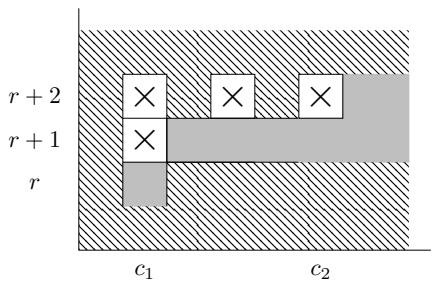
\begin{figure}[H]
    \centering
    $$\scalebox{0.9}{\begin{tikzpicture}[scale=0.65]
    \filldraw[fill=lightgray] (1,1) rectangle (4,2);
    \filldraw[draw=lightgray,fill=lightgray] (4,1) rectangle (5,2);
    \filldraw[draw=lightgray,fill=lightgray] (0,0) rectangle (1,1);
    
    \filldraw[draw=lightgray,fill=lightgray] (5,1) rectangle (6.5,3);
    \node at (0.5, 1.5) {$\bigtimes$};
    \node at (0.5, 2.5) {$\bigtimes$};
  \node at (4.5, 2.5) {$\bigtimes$};
  \node at (0.5,-1.5) {$c_1$};
  \node at (4.5,-1.5) {$c_2$};
  \node at (2.5, 2.5) {$\bigtimes$};
  \node at (-2, 0.5) {$r$};
  \node at (-2, 1.5) {$r+1$};
  \node at (-2, 2.5) {$r+2$};
  \draw (-1,4.5)--(-1,-1)--(7,-1);
  \draw (0,3)--(0,0);
  \draw (0,3)--(1,3)--(1,2);
  \draw (0,1)--(1,1)--(1,2)--(0,2);
  \draw (2,2)--(3,2)--(3,3)--(2,3)--(2,2);
  \draw (4,2)--(5,2)--(5,3)--(4,3)--(4,2);
  \draw (1,0)--(1,1)--(2,1);
  \draw (2,1)--(3,1);
  \path[pattern=north west lines] (-1,-1)--(6.5,-1) -- (6.5,1) --(1,1)--(1,0)--(0,0)--(0,3)--(1,3)--(1,2)--(2,2)--(2,3)--(3,3)--(3,2)--(4,2)--(4,3)--(6.5,3)--(6.5, 4)--(-1,4)-- cycle;
\end{tikzpicture}}$$
    \caption{Subdiagrams described in Corollary~\ref{cor:strucblock}}
    \label{fig:corstrucblock}
\end{figure}

\begin{proof}
    This result corresponds to taking $r^*=r+2$ in Proposition~\ref{prop:strucblock}.
\end{proof}

\begin{remark}
    The families of subdiagrams considered in Proposition~\ref{prop:strucblock} and Corollary~\ref{cor:strucblock} form a subset of those considered in Theorem 3.5 and Corollary 3.6, respectively, of \textup{\textbf{\cite{KPoset1}}}.
\end{remark}

In the sections that follow, we utilize the above results to characterize when diagrams belonging to restricted families are associated with (EL-)shellable Kohnert posets.

\section{Hook diagrams}\label{sec:hook}

In this section, we give a complete classification of diagrams $D$ for which $\mathcal{P}(D)$ is (EL-)shellable and either
\begin{enumerate}
    \item[1)] each nonempty column of $D$ contains exactly one cell or
    \item[2)] the first two rows of $D$ are empty.
\end{enumerate}
In both cases, the classification is given in terms of certain diagrams that we call ``hook diagrams".

For $r_1\le r_2\in\mathbb{Z}_{>0}$ and $C=\{c_1,\cdots,c_m\}\subset\mathbb{Z}_{>0}$, if $c_m=\max C$, then $$H(r_1,r_2;C)=\{(r_2,c_i)~|~1\le i\le m\}\cup\{(j,c_m)~|~r_1\le j\le r_2\}.$$ A \textbf{hook diagram} $D$ is a diagram for which there exists $r_1,r_2\in \mathbb{Z}_{>0}$ and $C\subset\mathbb{Z}_{>0}$ such that $D\in KD(H(r_1,r_2;C))$.

\begin{ex}\label{ex:hook}
    In Figure~\ref{fig:hook} \textup{(a)} we illustrate $H(4,6;\{1,2,4,7\})$ and in Figure~\ref{fig:hook} \textup{(b)} a hook diagram $D\in KD(H(4,6;\{1,2,4,7\}))$.
    \begin{figure}[H]
        \centering
        $$\begin{tikzpicture}[scale=0.4]
    \draw[thick] (7.5,-1)--(0,-1)--(0,5.5);
        \draw (0,4)--(2,4)--(2,5)--(0,5);
        \draw (1,4)--(1,5);
        \node at (1.5,4.5){$\bigtimes$};
        \node at (0.5,4.5){$\bigtimes$};

        \draw (3,4)--(4,4)--(4,5)--(3,5)--(3,4);
        \node at (3.5,4.5){$\bigtimes$};

        \draw (6,2)--(7,2)--(7,5)--(6,5)--(6,2);
        \draw (6,3)--(7,3);
        \draw (6,4)--(7,4);
        \node at (6.5,2.5){$\bigtimes$};
        \node at (6.5,3.5){$\bigtimes$};
        \node at (6.5,4.5){$\bigtimes$};
        \node at (3.5, -2) {(a)};
\end{tikzpicture}\quad\quad\quad\quad\begin{tikzpicture}[scale=0.4]
    \draw[thick] (7.5,-1)--(0,-1)--(0,5.5);
        \draw (0,4)--(1,4)--(1,5)--(0,5);
        \node at (0.5,4.5){$\bigtimes$};

        \draw (1,3)--(2,3)--(2,4)--(1,4)--(1,3);
        \node at (1.5,3.5){$\bigtimes$};

        \draw (3,2)--(4,2)--(4,3)--(3,3)--(3,2);
        \node at (3.5,2.5){$\bigtimes$};

        \draw (7,1)--(7,3)--(6,3)--(6,1)--(7,1);
        \draw (6,2)--(7,2);

        \draw (6,-1)--(6,0)--(7,0)--(7,-1);
        \node at (6.5,-0.5){$\bigtimes$};
        \node at (6.5,1.5){$\bigtimes$};
        \node at (6.5,2.5){$\bigtimes$};
        \node at (3.5, -2) {(b)};
\end{tikzpicture}$$
        \caption{(a) $H(4,6;\{1,2,4,7\})$ and (b) a hook diagram $D\in KD(H(4,6;\{1,2,4,7\}))$}
        \label{fig:hook}
    \end{figure}
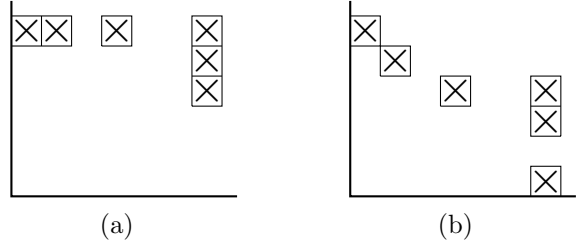
\end{ex}

The following lemma provides a characterization of hook diagrams in terms of the distribution of their cells.

\begin{prop}\label{prop:formhook}
    Let $D$ be a diagram. Then $D$ is a hook diagram if and only if there exists $c^*>0$ such that
\begin{enumerate}
    \item[\textup{(i)}] column $c^*$ of $D$ is nonempty and for all $\tilde{c}>c^*$, column $\tilde{c}$ of $D$ is empty;
    \item[\textup{(ii)}] if $D$ contains more than one cell in column $c$, then $c=c^*$;
    \item[\textup{(iii)}] if $r$ is maximal such that $(r,c^*)\in D$ and $\tilde{c}<c$, then $(\tilde{r},\tilde{c})\in D$ implies that $\tilde{r}\ge r$; and
    \item[\textup{(iv)}] if $(\tilde{r}_1,\tilde{c}_1),(\tilde{r}_2,\tilde{c}_2)\in D$ with $\tilde{c}_1<\tilde{c}_2<c^*$, then $\tilde{r}_1\ge \tilde{r}_2$.
\end{enumerate}
\end{prop}
\begin{proof}
    Assume that $D$ is a hook diagram. Then there exists $r_1\le r_2\in\mathbb{Z}_{>0}$ and $C=\{c_1,\cdots,c_m\}\subset \mathbb{Z}_{>0}$ such that $D\in KD(H(r_1,r_2;C))$. Evidently, if $c^*=\max C$, then properties (i) and (ii) hold in $D$. It remains to show that properties (iii) and (iv) also hold in $D$. We show that $D$ satisfies (iii), as (iv) follows via a similar argument. For a contradiction, assume that $D$ does not satisfy property (iii). Ongoing, for $\tilde{D}\in KD(H(r_1,r_2;C))$, we let $r^*(\tilde{D})=\max\{r~|~(r,c^*)\in \tilde{D}\}$. Since $D$ does not satisfy property (iii), there exists $(r,c)\in D$ such that $c<c^*$ and $r<r^*(D)$. Consequently, since $(r_2,c)\in H(r_1,r_2;C)$ and $r^*(H(r_1,r_2;C))=r_2$, it follows that there must exist $D_1,D_2\in KD(H(r_1,r_2;C))$ such that $(\tilde{r}_1,c)\in D_1$ with $\tilde{r}_1\ge r^*(D_1)$, $(\tilde{r}_2,c)\in D_2$ with $\tilde{r}_2<r^*(D_2)$, and $D_2$ can be formed from $D_1$ by applying a single Kohnert move; that is, in forming $D$ from $H(r_1,r_2;C)$, there must be a point at which a diagram that does not satisfy property (iii) in column $c$, namely $D_2$, is formed from one that does, namely $D_1$. Now, because $\tilde{r}_1\geq r^*(D_1)\geq r^*(D_2)>\tilde{r}_2$ and $D_2$ is obtained from $D_1$ by applying a single Kohnert move, it follows that $D_2=D_1\ldownarrow^{(\tilde{r}_1,c)}_{(\tilde{r}_2,c)}$. Further, it must be the case that $\tilde{r}_2=\tilde{r}_1-1$ since all diagrams in $KD(H(r_1,r_2;C))$ contain a single cell in column $c.$ However, this implies that $\tilde{r}_1=r^*(D_1)=r^*(D_2);$ that is, $(\tilde{r}_1,c)$ is not the rightmost cell in row $\tilde{r}_1$ of $D_1,$ contradicting that $D_2=D_1\ldownarrow^{(\tilde{r}_1,c)}_{(\tilde{r}_2,c)}.$
    % We claim that $D_2=D_1\ldownarrow^{(\tilde{r}_1,c)}_{(\tilde{r}_2,c)}$. Since all diagrams of $KD(H(r_1,r_2;C))$ contain a single cell in column $c$, if $D_2\neq D_1\ldownarrow^{(\tilde{r}_1,c)}_{(\tilde{r}_2,c)}$, then $D_2= D_1\ldownarrow^{(\hat{r}_1,\hat{c})}_{(\hat{r}_2,\hat{c})}$ for $\hat{c}\neq c$ so that $\tilde{r}_1=\tilde{r}_2$ and either $r^*(D_1)=r^*(D_2)$ or $r^*(D_2)=r^*(D_1)-1$. In either case, $\tilde{r}_2=\tilde{r}_1\ge r^*(D_1)\ge r^*(D_2)$, contradicting our choice of $D_2$. Thus, $D_1=D_2\ldownarrow^{(\tilde{r}_1,c)}_{(\tilde{r}_2,c)}$, as claimed. Now, once again invoking the fact that all diagrams of $KD(H(r_1,r_2;C))$ contain a single cell in column $c$, we have $\tilde{r}_2=\tilde{r}_1-1$. \bl{Since $\tilde{r}_1\ge r^*(D_1)=r^*(D_2)>\tilde{r}_2=\tilde{r}_1-1$, it follows that $\tilde{r}_1= r^*(D_1)$; that is, $(\tilde{r}_1,c)$ is not rightmost in row $\tilde{r}_1$ of $D_1$, contradicting $D_2=D_1\ldownarrow^{(\tilde{r}_1,c)}_{(\tilde{r}_2,c)}$.}
    Thus, $D$ must satisfy property (iii). For property (iv), one can use almost the exact same argument as that given above, replacing $(\tilde{r}_1,c),(r^*(D_1),c^*)\in D_1$ and $(\tilde{r}_2,c),(r^*(D_2),c^*)\in D_2$ with $(\tilde{r}_1,\tilde{c}_1),(\tilde{r}'_1,\tilde{c}_2)\in D_1$ and $(\tilde{r}_2,\tilde{c}_1),(\tilde{r}'_2,\tilde{c}_2)\in D_2$, respectively, where $\tilde{c}_1<\tilde{c}_2$, $\tilde{r}_1\ge\tilde{r}'_1$, and $\tilde{r}_2<\tilde{r}'_2$.

    Now, for the backward direction, assume that $D$ contains $m>1$ nonempty columns and that column $c^*$ of $D$ contains more than one cell; the cases where $m=1$ and/or nonempty columns of $D$ contain exactly one cell follow via similar -- but simpler -- arguments. Let $r^1_1\le\cdots\le r^1_{n_1}$ be the nonempty rows of columns $\tilde{c}<c^*$ of $D$ and $r^2_1<\cdots<r^2_{n_2}$ be the nonempty rows of $D$ in column $c^*$. Note that $r^1_{1}\ge r^2_{n_2}$ by condition (iii). If $C$ denotes the set of nonempty columns of $D$, then we claim that $D\in KD(H(r^1_{n_1}-n_2+1,r_{n_1}^1;C))$. To see this, note that we can form $D$ from $H(r^1_{n_1}-n_2+1,r_{n_1}^1;C)$ as follows. 
    \begin{enumerate}
        \item For $1\le i\le n_2$ in increasing order, if $r^1_{n_1}-n_2+i\neq r^2_i$, then apply in succession one Kohnert move at rows $r^1_{n_1}-n_2+i$ down to $r^2_i+1$ in decreasing order; otherwise, apply no Kohnert moves.
        \item For $1\le i\le n_1$ in increasing order, if $r^1_{n_1}\neq r^1_i$, then apply in succession one Kohnert move at rows $r^1_{n_1}$ down to $r^1_i+1$ in decreasing order; otherwise apply no Kohnert moves.
    \end{enumerate}
    Thus, $D\in KD(H(r^1_{n_1}-n_2+1,r_{n_1}^1;C))$ and the result follows.
\end{proof}

% \bigskip

    % \noindent
    % \textbf{Case 1:} $m=1$. Let $r_1<\cdots<r_n$ be the nonempty rows of $D$. We claim that $D\in KD(H(r_n-n+1,r_n;\{c^*\}))$. To see this, note that we can form $D$ from $H(r_n-n+1,r_n;\{c^*\})$ as follows. For $1\le i\le n$ in increasing order, if $r_n-n+i\neq r_i$, then apply in succession a single Kohnert move at rows $r_n-n+i$ down to $r_i+1$ in decreasing order; otherwise, apply no Kohnert moves. Thus, $D\in KD(H(r_n-n+1,r_n;\{c^*\}))$ and the result follows.
    % \bigskip

Theorem~\ref{thm:hook} below is the main result of this section, and its proof -- along with some noteworthy consequences regarding the corresponding Kohnert polynomials -- constitutes the remainder of this section.

% As alluded to above, in the remainder of this section we establish the following.

\begin{theorem}\label{thm:hook}~
    \begin{enumerate}
        \item[\textup{(a)}] Let $D$ be a diagram for which all nonempty columns contain exactly one cell. Then $\mathcal{P}(D)$ is \textup(EL-\textup) shellable if and only if $D\backslash\{(1,\tilde{c})~|~\tilde{c}>0\}$ is a hook diagram.
        \item[\textup{(b)}] Let $D$ be a diagram for which the first two rows are empty. Then $\mathcal{P}(D)$ is \textup(EL-\textup)shellable if and only if $D$ is a hook diagram.
    \end{enumerate}
\end{theorem}

Before proceeding with the proof of Theorem~\ref{thm:hook}, we include the following lemma which relates the Kohnert posets of the two diagrams occurring in Theorem~\ref{thm:hook} (a).

\begin{lemma}\label{lem:0bpciso}
    If $D$ is a diagram for which all nonempty columns contain exactly one cell, then $\mathcal{P}(D)\cong \mathcal{P}(D\backslash\{(1,\tilde{c})~|~\tilde{c}>0\})$.
\end{lemma}
\begin{proof}
    Immediate from the definitions of Kohnert move and Kohnert poset.
\end{proof}

\subsection{Sufficiency}\label{sec:suffhook}

In this section, we establish the backward directions of Theorem~\ref{thm:hook} (a) and (b). Considering Lemma~\ref{lem:0bpciso}, it suffices to show that if $D$ is a hook diagram, then $\mathcal{P}(D)$ is EL-shellable. 

To show that hook diagrams generate EL-shellable Kohnert posets, we first establish that such posets are bounded.

\begin{lemma}\label{lem:hookbound}
    If $D$ is a hook diagram, then $\mathcal{P}(D)$ is bounded. 
\end{lemma}
\begin{proof}
    It suffices to establish the result for $D=H(r_1,r_2;C)$ with $r_1\le r_2\in\mathbb{Z}_{>0}$ and $C=\{c_1,\cdots,c_m\}~\subset~\mathbb{Z}_{>0}$. We claim that $D_{min}=H(1,r_2-r_1+1;C)$ is the unique minimal element of $\mathcal{P}(D)$. To see this, take $\tilde{D}\in \bl{\mathcal{P}}(H(r_1,r_2;C))$. If $|C|=1$ or $r_1=r_2$, then the result follows by Corollary 6.2 of \textbf{\cite{KPoset1}}. So, assume that $|C|=m>1$, $r_2>r_1$, and $c_1<\cdots<c_m$. Considering Proposition~\ref{prop:formhook}, there exists $\tilde{r}_{m-1}\ge\cdots\ge \tilde{r}_1$ and $r^*_1<\cdots<r^*_{r_2-r_1+1}$ such that $$\tilde{D}=\{(\tilde{r}_{m-i},c_i)~|~1\le i\le m-1\}\cup\{(r^*_i,c_m)~|~1\le i\le r_2-r_1+1\}.$$ Note that one can form $D_{min}$ from $\tilde{D}$ by applying successively
    \begin{enumerate}
        \item[1)] for $1\le i\le r_2-r_1+1$ in increasing order, one Kohnert move at rows $r^*_i$ through $i+1$ in decreasing order, followed by
        \item[2)] for $1\le i\le m-1$ in increasing order, one Kohnert move at rows $\tilde{r}_i$ through $r_2-r_1+2$ in decreasing order.
    \end{enumerate}
    Consequently, $D_{min}\preceq \tilde{D}$. As $\tilde{D}$ was arbitrary, the claim and, hence, the result follows.
\end{proof}

\begin{ex}\label{ex:hookmin}
    In Figure~\ref{fig:hookmin} below we illustrate the unique minimal element $D_{min}$ of the Kohnert posets associated with the two hook diagrams of Example~\ref{ex:hook}.

    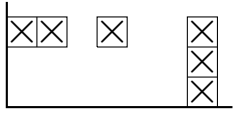
\begin{figure}[H]
        \centering
        $$\begin{tikzpicture}[scale=0.4]
    \draw[thick] (7.5,0)--(0,0)--(0,3.5);
        \draw (0,2)--(2,2)--(2,3)--(0,3);
        \draw (1,2)--(1,3);
        \node at (1.5,2.5){$\bigtimes$};
        \node at (0.5,2.5){$\bigtimes$};

        \draw (3,2)--(4,2)--(4,3)--(3,3)--(3,2);
        \node at (3.5,2.5){$\bigtimes$};

        \draw (7,0)--(7,3)--(6,3)--(6,0);
        \draw (6,1)--(7,1);
        \draw (6,2)--(7,2);
        \node at (6.5,0.5){$\bigtimes$};
        \node at (6.5,1.5){$\bigtimes$};
        \node at (6.5,2.5){$\bigtimes$};
\end{tikzpicture}$$
        \caption{Minimal element associated with a hook diagram}
        \label{fig:hookmin}
    \end{figure}
\end{ex}

% Using Lemma~\ref{lem:hookbound} we can now establish the backwards directions of Theorem~\ref{thm:hook} (a) and (b) by showing that for a hook diagram $D$, $\mathcal{P}(D)$ is EL-shellable.

% \begin{prop}\label{prop:hshell}
%     If $D$ is a hook diagram, then $\mathcal{P}(D)$ is EL-shellable. 
% \end{prop}
% \begin{proof}
%     Considering Theorem~\ref{thm:ELinterval} and Lemma~\ref{lem:hookbound}, it suffices to show that for $r_1\le r_2\in\mathbb{Z}_{>0}$ and $C\subset \mathbb{Z}_{>0}$, the poset $H(r_1,r_2;C)$ is EL-shellable. To see this, note given any hook diagram $D$, by definition there must exist $r_1\le r_2\in\mathbb{Z}_{>0}$ and $C\subset \mathbb{Z}_{>0}$ such that $D\in KD(H(r_1,r_2;C))$. Moreover, by Lemma~\ref{lem:hookbound}, we have that if $D_{min}$ is the unique minimal element of $\mathcal{P}(H(r_1,r_2;C))$, then $\mathcal{P}(D)$ is the interval $[D_{min},D]$ of $\mathcal{P}(H(r_1,r_2;C))$. Consequently, applying Theorem~\ref{thm:ELinterval}, if $\mathcal{P}(H(r_1,r_2;C))$ is EL-shellable, then so is $\mathcal{P}(D)$.
% \end{proof}

In order to establish that Kohnert posets associated with hook diagrams are EL-shellable, we describe an edge labeling that satisfies Definition~\ref{def:EL}. First, however, we require the following lemma.

\begin{lemma}\label{lem: hook covering} Let $D$ be a hook diagram. If $D'=D\ldownarrow^{(r,c)}_{(r-k,c)}$ and $D'\precdot D$, then $k=1$.
\end{lemma}
\begin{proof} Assume otherwise. Without loss of generality, let $D'=D\ldownarrow^{(r,c)}_{(r-k,c)}$ with $k=2.$ Note that, since $D'\precdot D$, it must be the case that $D'$ can only be formed from $D$ by applying exactly one Kohnert move. Consequently, it follows that $(r-1,c)\in D$. Now, since $(r,c),(r-1,c)\in D$, we may conclude that $c$ is the unique column in $D$ with more than one cell. So, by condition (i) of Proposition~\ref{prop:formhook}, for all $\tilde{c}>c,$ column $\tilde{c}$ of $D$ is empty. Thus, $D'$ can be formed from $D$ by applying a single Kohnert move at row $r-1$ followed by a single Kohnert move at row $r;$ that is $$D'\prec D\ldownarrow^{(r-1,c)}_{(r-2,c)}\precdot D,$$ a contradiction. The result follows.

% Then for $(r,c)\in D_2$ to be moved to $(r-2,c)\in D_1$, we must have that the cell in location $(r,c)$ ``jumped" over the cell in location $(r-1,c)$. Because $c$ contains more than one cell and $D_2$ is a hook diagram, all columns $c_i>c$ are empty \rn{(so what can you conclude from this?)}. Then, since $(r-2,c)\in D_2$ is empty, \rn{applying} a Kohnert move to \rn{$D_2$ at row $r-1$} moves the cell $(r-1,c)\in D_2$ to $(r-2,c)$ followed by another Kohnert move to move the cell in $(r,c)\in D_2$ to $(r-1,c)$. The resulting diagram from these two moves is equivalent to $D_1$, therefore there exists some $D'$ such that $(r,c),(r-2,c)\in D'$ and $D_1\preceq D'\preceq D_2$ which is a contradiction.
\end{proof}

To define our edge labelings for Kohnert posets arising from hook diagrams, we utilize the following labeling of the cells in the associated diagrams. Let $D$ be a hook diagram where $C=\{c_1,\dots,c_m\}$ is the set of nonempty columns in $D$ with $c_1<\dots<c_m,$ and let $R=\{r_i~|~(r_i,c_m)\in D\text{ for }1\leq i\leq n\}$ with $r_1>\dots>r_n$. Decorate the cell located in column $c_j$ of $D$ with the label $j$ for $1\leq j<m$,  then decorate the cell in column $c_m$ and row $r_i$ of $D$ with the label $m-i+1$ for $1\leq i\leq n$. To extend this labeling to remaining diagrams of $\mathcal{P}(D)$, if $D_1,D_2\in\mathcal{P}(D)$ are such that 
\begin{equation}\label{eq:3}
    D_2=D_1\ldownarrow^{(r,c)}_{(r-1,c)}\precdot D_1
\end{equation}
and the label of $(r,c)\in D_1$ is $L$, then decorate the cell $(r-1,c)\in D_2$ with the label $L$ and maintain the labels of the cells in $D_1\cap D_2$. Now, we define our edge labeling $\lambda:\mathcal{E}(\mathcal{P}(D))\to\mathbb{Z}_{>0}$ by $\lambda(D_2,D_1)=L$ if (\ref{eq:3}) holds and the label of $(r,c)\in D_1$ is $L$. See Example~\ref{ex:shelling}.

\begin{ex}\label{ex:shelling}
In Figure~\ref{fig:shelling} below we illustrate our cell decoration and edge labeling on an interval $[D_2,D_1]\subset\mathcal{P}(D),$ where $D$ is a hook diagram.

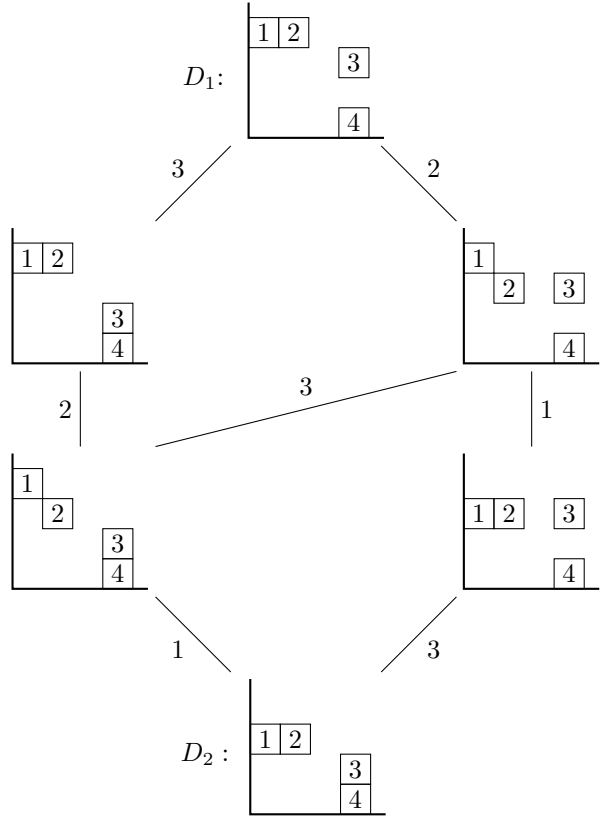
\begin{figure}[H]
$$\begin{tikzpicture}

    \node at (0,4) {\begin{tikzpicture}[scale=0.4]
    \draw[thick] (4.5,0)--(0,0)--(0,4.5);
    \draw (0,3) rectangle (1,4);
    \draw (1,3) rectangle (2,4);
    \draw (3,2) rectangle (4,3);
    \draw (3,0) rectangle (4,1);

    \node at (0.5,3.5) {1};
    \node at (1.5,3.5) {2};
    \node at (3.5,2.5) {3};
    \node at (3.5,0.5) {4};

    \node at (-1.5,2){$D_1$:};
    \node at (6,2){};
        
\end{tikzpicture}};

    \node at (-3,1) {\begin{tikzpicture}[scale=0.4]
    \draw[thick] (4.5,0)--(0,0)--(0,4.5);
    \draw (0,3) rectangle (1,4);
    \draw (1,3) rectangle (2,4);
    \draw (3,1) rectangle (4,2);
    \draw (3,0) rectangle (4,1);

    \node at (0.5,3.5) {1};
    \node at (1.5,3.5) {2};
    \node at (3.5,1.5) {3};
    \node at (3.5,0.5) {4};
        
\end{tikzpicture}};

    \node at (3,1) {\begin{tikzpicture}[scale=0.4]
    \draw[thick] (4.5,0)--(0,0)--(0,4.5);
    \draw (0,3) rectangle (1,4);
    \draw (1,2) rectangle (2,3);
    \draw (3,2) rectangle (4,3);
    \draw (3,0) rectangle (4,1);

    \node at (0.5,3.5) {1};
    \node at (1.5,2.5) {2};
    \node at (3.5,2.5) {3};
    \node at (3.5,0.5) {4};
        
\end{tikzpicture}};

\node at (-3,-2) {\begin{tikzpicture}[scale=0.4]
    \draw[thick] (4.5,0)--(0,0)--(0,4.5);
    \draw (0,3) rectangle (1,4);
    \draw (1,2) rectangle (2,3);
    \draw (3,1) rectangle (4,2);
    \draw (3,0) rectangle (4,1);

    \node at (0.5,3.5) {1};
    \node at (1.5,2.5) {2};
    \node at (3.5,1.5) {3};
    \node at (3.5,0.5) {4};
        
\end{tikzpicture}};

\node at (3,-2) {\begin{tikzpicture}[scale=0.4]
    \draw[thick] (4.5,0)--(0,0)--(0,4.5);
    \draw (0,2) rectangle (1,3);
    \draw (1,2) rectangle (2,3);
    \draw (3,2) rectangle (4,3);
    \draw (3,0) rectangle (4,1);

    \node at (0.5,2.5) {1};
    \node at (1.5,2.5) {2};
    \node at (3.5,2.5) {3};
    \node at (3.5,0.5) {4};
        
\end{tikzpicture}};

\node at (0,-5) {\begin{tikzpicture}[scale=0.4]
    \draw[thick] (4.5,0)--(0,0)--(0,4.5);
    \draw (0,2) rectangle (1,3);
    \draw (1,2) rectangle (2,3);
    \draw (3,1) rectangle (4,2);
    \draw (3,0) rectangle (4,1);

    \node at (0.5,2.5) {1};
    \node at (1.5,2.5) {2};
    \node at (3.5,1.5) {3};
    \node at (3.5,0.5) {4};

    \node at (-1.5,2) {$D_2:$};
    \node at (6,2) {};
        
\end{tikzpicture}};

\draw (-1,3)--(-2,2);
\draw (1,3)--(2,2);

\draw (-3,0)--(-3,-1);
\draw (3,0)--(3,-1);
\draw (2,0)--(-2,-1);

\draw (-2,-3)--(-1,-4);
\draw (2,-3)--(1,-4);

\node at (-1.7,2.7) {3};
\node at (1.7, 2.7) {2};

\node at (-3.2,-0.5) {2};
\node at (0,-0.2) {3};
\node at (3.2,-0.5) {1};

\node at (-1.7,-3.7) {1};
\node at (1.7,-3.7) {3};

\end{tikzpicture}$$
\caption{Edge-labeled interval of $\mathcal{P}(D)$}
\label{fig:shelling}
\end{figure}
\end{ex}

\begin{lemma}\label{lem: hook multiset} Let $D$ be a hook diagram and $I=[D_2,D_1]$ be an interval in $\mathcal{P}(D)$ equipped with the edge labeling described above. If $\mathcal{C}_1,\mathcal{C}_2$ are maximal chains in $I$, then the multiset of edge labels corresponding to $\mathcal{C}_1$ is equal to that of $\mathcal{C}_2$.
\end{lemma}
\begin{proof}

Let $\mathcal{C}_1,\mathcal{C}_2$ be two maximal chains in $I$. If $(r,c)\in D_1$ has label $L$ and is moved, via a sequence of Kohnert moves, to position $(r-k,c)\in D_2$ for some $k\in\mathbb{Z}_{>0}$, then it follows from Lemma~\ref{lem: hook covering} that among the edges of $\mathcal{C}_1$ and $\mathcal{C}_2$ in the Hasse diagram of $\mathcal{P}(D)$, exactly $k$ are labeled $L$. Since $(r,c)\in D_1$ was arbitrary, the result follows.
\end{proof}

\begin{remark}\label{rem:unique}
As a consequence of Lemma~\ref{lem: hook covering}, given an interval $I=[D_2,D_1]\subseteq\mathcal{P}(D)$, each maximal chain in $I$ has its own unique ordered list of edge labels. To see this, note that each edge label in a given chain in $I$ is defined by a particular cell being moved down one row. Thus, if two chains have the same ordered list of edge labels, then they correspond to the same sequence of Kohnert moves, i.e., they are equal chains. Combining this observation with Lemma~\ref{lem: hook multiset}, we conclude that if $\mathcal{C}_1\neq\mathcal{C}_2$ are two maximal chains in $I,$ then the ordered list of edge labels corresponding to $\mathcal{C}_1$ is a nontrivial permutation of the ordered list of edge labels corresponding to $\mathcal{C}_2.$
\end{remark}

\begin{theorem}\label{thm:HrrCshell}
If $D$ is a hook diagram, then $\mathcal{P}(D)$ is EL-shellable.
\end{theorem}
\begin{proof}
Let $\mathcal{P}=\mathcal{P}(D)$ and assume that $D$ contains $n>k$ cells. By Lemma~\ref{lem:hookbound}, $\mathcal{P}$ is bounded. Thus, it remains to show that $\mathcal{P}$ admits an edge labeling that satisfies Definition~\ref{def:EL}. Decorate the cells of all $\widetilde{D}\in \mathcal{P}$ and label the Hasse diagram of $\mathcal{P}$ as described above. Take $D_1,D_2\in\mathcal{P}$ satisfying $D_2\prec D_1$. Assume that for $1\leq i\leq k$, if the cell in location $(r,c)$ of $D_1$ is decorated with label $L_i$, then the cell $(r-\mu_i,c)\in D_2$, $\mu_i> 0$, is decorated with label $L_i$; that is, $D_2$ is formed from $D_1$ by moving each cell in $D_1$ with label $L_i$ down $\mu_i$ rows. By Lemmas~\ref{lem: hook covering} and \ref{lem: hook multiset}, it follows that each chain in the interval $I=[D_2,D_1]$ has the multiset of edge labels $\{L_1^{\mu_1},\dots,L_k^{\mu_k}\},$ where, without loss of generality, we may assume $L_i<L_j$ whenever $i<j.$

We describe the sequence of Kohnert moves that corresponds to the unique rising unrefinable chain $\mathcal{C}^*$ in $I.$ In short, $\mathcal{C}^*$ is the chain obtained by moving each cell individually from its location in $D_1$ to its position in $D_2$ in decreasing order of the cells' labels. Explicitly, define $\mathcal{C}^*$ to be the chain of diagrams obtained via the following:
\begin{description}
\item[Step 1:] If the cell labeled $L_k$ is located in row $\tilde{r}_k$ of $D_1$, then apply, in succession, a single Kohnert move to $D_1$ at rows $\tilde{r}_k$ through $\tilde{r}_k-\mu_k+1$ in decreasing order. Call the resulting diagram $D^k.$
\item[Step i:] If the cell labeled $L_{k-i+1}$ is located in row $\tilde{r}_{k-i+1}$ of $D_1$, then apply, in succession, a single Kohnert move to $D^{k-i+2}$ at rows $\tilde{r}_{k-i+1}$ through $\tilde{r}_{k-i+1}-\mu_i+1$ in decreasing order. Call the resulting diagram $D^{k-i+1}.$
\end{description}
\noindent Considering the method of labeling the cells of $D_1$ along with Proposition~\ref{prop:formhook}, it follows that applying the sequence of Kohnert moves outlined above has the desired effect. Note that the ordered list of labels corresponding to $\mathcal{C}^*$ is $(L_1^{\mu_1},\dots,L_k^{\mu_k}).$ It follows from Remark~\ref{rem:unique} that this list is lexicographically minimal with respect to all chains in $I$ and that $\mathcal{C}^*$ is the unique rising chain in $I.$
% Moreover, in light of Lemma~\ref{lem: hook multiset}, \bl{(I feel like both remaining pieces should follow from this. Maybe just note the form of any other chain. The labels must differ in order in some place since otherwise it would correspond to the same chain, evidently I think we can say. Find the first place where they differ moving from bottom to top. I feel like it shouldn't be hard to argue that the increasing one must be smaller there.)} \rn{this sequence of labels is clearly lexicographically minimal with respect to all maximal chains in $I$. It remains to show that $\mathcal{C}^*$ is the unique maximal chain in $I$ that is weakly increasing. To see this, notice that any maximal chain $\tilde{\mathcal{C}}\neq\mathcal{C}^*$ in $I$ has a corresponding ordered list of edge labels which is a nontrivial permutation of $(L_1^{\mu_1},\dots,L_k^{\mu_k}),$ by Lemma~\ref{lem: hook multiset}. Thus, there exists a sequence of covering relations $\tilde{D}_1\precdot\tilde{D}_2\precdot\tilde{D}_3$ in $\tilde{\mathcal{C}}$ such that $\lambda(\tilde{D}_1,\tilde{D}_2)>\lambda(\tilde{D}_2,\tilde{D}_3);$ that is, the ordered list of edge labels corresponding to $\tilde{\mathcal{C}}$ is not weakly increasing. The result follows.}
\end{proof}

\subsection{Necessity}

In this section, we finish the proof of Theorem~\ref{thm:hook}. We consider each part separately, starting with Theorem~\ref{thm:hook} (a).

\begin{proof}[Proof of Theorem~\ref{thm:hook} \textup(a\textup)] The backward direction was established in Section~\ref{sec:suffhook}.

Let $\widehat{D}=D\backslash\{(1,\tilde{c})~|~\tilde{c}>0\}=\{(r_1,c_1),\hdots,(r_n,c_n)\}$ with $c_1<\hdots<c_n$. Note that $r_i>1$ for $1\le i\le n$. Assume that $\widehat{D}$ is not a hook diagram. Then, applying Proposition~\ref{prop:formhook}, there exists $1\le i<j\le n$ such that $r_i<r_j$. Let $$n(k)=\begin{cases}
    |\{\tilde{c}~|~\tilde{c}>c_i,~(r_i,\tilde{c})\in \widehat{D}\}|, & k=i \\
    |\{\tilde{c}~|~\tilde{c}>c_j,~(r_k,\tilde{c})\in \widehat{D}\}|, & i< k\le j,
\end{cases}$$
for $k$ satisfying $i\le k\le j$, and form $D^*\in \mathcal{P}(\widehat{D})$ from $\widehat{D}$ as follows.
\begin{enumerate}
    \item[1)] For $k$ satisfying $i\le k\le j$ in increasing order, successively apply $n(k)$ Kohnert moves at rows $r_k$ through $2$ in decreasing order.
    \item[2)] If $r_i<r_j-1$, then apply, in succession, one Kohnert move at rows $r_j$ through $r_i+2$ in decreasing order; otherwise, do nothing.
\end{enumerate}
By our assumptions on $D$, it follows that
\begin{itemize}
    \item $(r_i,c_i), (r_i+1,c_j)\in D^*$,
    \item $(r_i+1,\tilde{c}),(r_i-1,c_j)\notin D^*$ for $\tilde{c}>c_j$, and
    \item $(r_i,\tilde{c}),(r_i-1,c_i)\notin D^*$ for $\tilde{c}>c_i$.
\end{itemize}
Thus, applying Proposition~\ref{prop:strucasc} with $r=r_i-1, c_1=c_i,$ and $c_2=c_j$, $\mathcal{P}(\widehat{D})$ is not shellable. The result now follows from Lemma~\ref{lem:0bpciso}.
\end{proof}

For the proof of Theorem~\ref{thm:hook} (b), Lemmas~\ref{lem:behindblock} through~\ref{lem:onecol} below identify necessary conditions for diagrams whose first two rows are empty to be associated with shellable Kohnert posets. Utilizing the aforementioned lemmas along with Proposition~\ref{prop:formhook}, we will be able to show that the only diagrams that remain are hook diagrams, finishing the proof of Theorem~\ref{thm:hook} (b).

\begin{lemma}\label{lem:behindblock}
    Let $D$ be a diagram for which all cells of $D$ are contained in rows $2<r_1<r_2<\hdots<r_n$ and $\mathcal{P}(D)$ is shellable. If $c^*$ is the rightmost nonempty column of $D$ and $r^*$ is maximal such that $(r^*,c^*)\in D$, then $(\tilde{r},\tilde{c})\in D$ for $\tilde{c}<c^*$ implies that $\tilde{r}\ge r^*$.
\end{lemma}
\begin{proof}
    Assume otherwise. Then there exists a maximal $c<c^*$ such that, for some $i$ and $j$ satisfying $1\leq i<j\leq n$, we have $(r_i,c),(r_j,c^*)\in D$. Assume that $j$ is chosen to be minimal with the aforementioned properties, i.e., $(r_j,c^*)$ is the lowest cell in column $c^*$ for which there exists a cell in column $c$ lying in a strictly lower row. Note that, by the maximality of $c$, there are no cells below row $r_j$ in columns strictly between $c$ and $c^*$. Further, by the minimality of $j$, there are no cells in column $c^*$ strictly between rows $r_i$ and $r_j$. Now, form $\widehat{D}\in \mathcal{P}(D)$ as follows. 
    \begin{enumerate}
        \item[1)] Letting $k=|\{\tilde{c}~|~\tilde{c}=c\text{ or }c^*,~(r_i-1,\tilde{c})\in D\}|$, apply $k$ Kohnert moves at row $r_i-1$.
        \item[2)] Next, if $(r_i,c^*)\in D$, then apply a single Kohnert move at row $r_i$ followed by a single Kohnert move at row $r_i-1$\bl{;} otherwise, do nothing.
        \item[3)] Finally, if $r_i<r_j-1$, then apply, in succession, one Kohnert move at rows $r_j$ through $r_i+2$ in decreasing order\bl{;} otherwise, do nothing.
    \end{enumerate}
    Note that since $r_1>2$, any Kohnert moves applied in steps 1 and 2 are nontrivial. By our assumptions on $D$, it follows that
    \begin{itemize}
        \item $(r_i,c),(r_i+1,c^*)\in \widehat{D}$,
        \item $(r_i+1,\tilde{c}),(r_i-1,c^*)\notin \widehat{D}$ for $\tilde{c}>c^*$, and
        \item $(r_i,\tilde{c}),(r_i-1,c)\notin\widehat{D}$ for $\tilde{c}>c$.
    \end{itemize}
    Therefore, applying Proposition~\ref{prop:strucasc} with $r=r_i-1$, $c_1=c$, and $c_2=c^*$, it follows that $\mathcal{P}(D)$ is not shellable, a contradiction. The result follows.
\end{proof}

\begin{lemma}\label{lem:infrontblock}
    Let $D$ be a diagram for which all cells of $D$ are contained in rows $2<r_1<r_2<\hdots<r_n$, there exists a column in which $D$ contains more than one cell, and $\mathcal{P}(D)$ is shellable.
    If $c^*$ is the rightmost column of $D$ containing more than one cell, then all columns $\tilde{c}>c^*$ of $D$ are empty.
\end{lemma}
\begin{proof}
    Assume otherwise. Let $c>c^*$ be minimal such that column $c$ of $D$ is nonempty. Since $c^*$ is the rightmost column of $D$ containing more than one cell, all nonempty columns $\tilde{c}>c^*$ of $D$ must contain exactly one cell. Let $D_1$ denote the diagram formed by bottom-justifying the cells in columns $\tilde{c}>c$ of $D$. Evidently, $D_1\in \mathcal{P}(D)$. Assume that $(r_i,c),(r_j,c^*)\in D_1$ where $r_j$ is minimal, i.e., $(r_j,c^*)$ is the lowest cell in column $c^*$ of $D_1$. Note that since column $c^*$ contains more than one cell and $r_j$ is the row occupied by the lowest such cell, it follows that $j<n$. There are two cases.
    \bigskip

    \noindent
    \textbf{Case 1:} $i>j$. In this case, form the diagram $D_2$ from $D_1$ as follows. If $r_j<r_i-1$, then apply, in succession, one Kohnert move at rows $r_i$ through $r_j+2$ in decreasing order; otherwise, do nothing. By our assumptions on $D$,
    \begin{itemize}
        \item $(r_j,c^*),(r_j+1,c)\in D_2$,
        \item $(r_j+1,\tilde{c}),(r_j-1,c)\notin D_2$ for $\tilde{c}>c$, and
        \item $(r_j,\tilde{c}),(r_j-1,c^*)\notin D_2$ for $\tilde{c}>c^*$.
    \end{itemize}
    Thus, applying Proposition~\ref{prop:strucasc} with $r=r_j-1$, $c_1=c^*$, and $c_2=c$, it follows that $\mathcal{P}(D)$ is not shellable, a contradiction.
    \bigskip

    \noindent
    \textbf{Case 2:}  $i\le j$. Assume that $r_k$ is minimal such that $(r_k,c^*)\in D_1$ and $r_k>r_j$, i.e., $r_k$ is the second lowest nonempty row in column $c^*$ of $D$. Form the diagram $D_2$ from $D_1$ as follows. 
    \begin{enumerate}
        \item[1)] If $i<j$, then apply, in succession, one Kohnert move at rows $r_j$ through $r_i+1$ in decreasing order; otherwise, do nothing.
        \item[2)] Apply, in succession, one Kohnert move at rows $r_k$ through $r_i+1$ in decreasing order.
    \end{enumerate}
    By our assumptions on $D$, it follows that
    \begin{itemize}
        \item $(r_i-1,c^*),(r_i,c)\in D_2$,
        \item $(r_i,\tilde{c}),(r_i-2,c)\notin D_2$ for $\tilde{c}>c$, and
        \item $(r_i-1,\tilde{c}),(r_i-2,c^*)\notin D_2$ for $\tilde{c}>c^*$.
    \end{itemize}
    Applying Proposition~\ref{prop:strucasc} with $r=r_i-2$, $c_1=c^*$, and $c_2=c$, it follows that $\mathcal{P}(D)$ is not shellable, a contradiction.
    \bigskip

    \noindent
    The result follows.
\end{proof}

\begin{lemma}\label{lem:onecol}
    Let $D$ be a diagram for which all cells of $D$ are contained in rows $2<r_1<r_2<\hdots<r_n$ and $\mathcal{P}(D)$ is shellable. If there exists a column in which $D$ contains more than one cell, then it is unique.
\end{lemma}
\begin{proof}
    Assume otherwise. Then there exists at least two columns of $D$ each of which contains more than one cell. Let $c^*_1<c^*_2$ be the rightmost two such columns of $D$ and assume that $r_i$ is maximal such that $(r_i,c_2^*)\in D$. Then 
    \begin{enumerate}
        \item[(i)] for all $\tilde{c}$ satisfying $c^*_1<\tilde{c}<c^*_2$, column $\tilde{c}$ of $D$ contains at most one cell;
        \item[(ii)] applying Lemma~\ref{lem:infrontblock}, all columns $\tilde{c}>c^*_2$ of $D$ are empty; and
        \item[(iii)] applying Lemma~\ref{lem:behindblock}, $(\tilde{r},\tilde{c})\in D$ for $\tilde{c}<c_2^*$ implies that $\tilde{r}\ge r_i$.
    \end{enumerate}
    Now, assume that $r_j$ and $r_k$ are minimal such that $r_j<r_k$ and $(r_j,c^*_1),(r_k,c^*_1)\in D$, i.e., $(r_j,c^*_1),(r_k,c^*_1)$ are the two lowest cells in column $c_1^*$ of $D$. Note that, considering (iii) above, $r_i\le r_j<r_k$. Define $$n(t)=|\{\tilde{c}~|~(r_t,\tilde{c})\in D,~c^*_1\le\tilde{c}<c^*_2\}|$$ for $i< t\le k$. Form $\widehat{D}\in \mathcal{P}(D)$ as follows.
    \begin{enumerate}
        \item[1)] For $i<t\le k$ in increasing order, apply, in succession, $n(t)$ Kohnert moves at rows $r_t$ through $r_i+1$ in decreasing order.
        \item[2)] If $(r_i-1,c_2^*)\in D$, then apply a single Kohnert move at row $r_i-1$; otherwise, do nothing.
        \item[3)] If $(r_i-2,c_2^*)\in D$, then apply a single Kohnert move at row $r_i-2$; otherwise, do nothing.
    \end{enumerate}
    By our assumptions on $D$, it follows that
    \begin{itemize}
        \item $(r_i-1,c^*_1),(r_i,c^*_2)\in \widehat{D}$,
        \item $(r_i,\tilde{c}), (r_i-2,c^*_2)\notin \widehat{D}$ for $\tilde{c}>c^*_2$, and
        \item $(r_i-1,\tilde{c}),(r_i-2,c^*_1)\notin \widehat{D}$ for $\tilde{c}>c^*_1$.
    \end{itemize}
    Therefore, applying Proposition~\ref{prop:strucasc} with $r=r_i-2$, $c_1=c^*_1$, and $c_2=c^*_2$, it follows that $\mathcal{P}(D)$ is not shellable, a contradiction. The result follows.
\end{proof}

Combining the results above, we can now finish the proof of Theorem~\ref{thm:hook}.

\begin{proof}[Proof of Theorem~\ref{thm:hook} \textup(b\textup)] The backward direction was established in Section~\ref{sec:suffhook}.

Let $D$ be a diagram for which all cells are contained in rows $2<r_1<r_2<\hdots<r_n$. Note that if $D$ has no columns with more than one cell, then the result follows by (a). Consequently, we assume that there exists a column in which $D$ contains more than one cell. 

Assume that $\mathcal{P}(D)$ is shellable. Applying Lemma~\ref{lem:onecol}, there exists a unique column $c^*$ in which $D$ contains more than one cell. By Lemma~\ref{lem:infrontblock}, all columns $\tilde{c}>c^*$ of $D$ are empty. Moreover, by Lemma~\ref{lem:behindblock}, if $r_i$ is maximal such that $(r_i,c^*)\in D$ and $\tilde{c}<c^*$, then $(\tilde{r},\tilde{c})\in D$ implies $\tilde{r}\ge r_i$. Consequently, $D$ satisfies all the properties listed in Proposition~\ref{prop:formhook} except possibly property (iv). For a contradiction, assume that $D$ does not satisfy property (iv) of Proposition~\ref{prop:formhook}; that is, assume that there exists $\widehat{c}_1<\widehat{c}_2<c^*$ such that 
$(r_i,\widehat{c}_1),(r_j,\widehat{c}_2)\in D$ with $r_i<r_j.$
% \rn{By Lemma~\ref{lem:onecol}, if $c<c^*,$ then there is at most one cell in column $c,$ so we may take $\widehat{c}_1$ and $\widehat{c}_2$ to be maximal with respect to property (\ref{eqn:ascent}). Further, if $r^*$ is maximal such that $(r^*,c^*)\in D,$ then it follows from Lemma~\ref{lem:behindblock} that $r_j>r_i\geq r^*.$ Consequently, the cell $(r_j,\widehat{c}_2)\in D$ is the rightmost in row $r_j$, and for all $\tilde{r},\tilde{c}$ satisfying $r_i<\tilde{r}<r_j$ and $\widehat{c}_1<\tilde{c}<\widehat{c}_2,$ we have that $(\tilde{r},\tilde{c})\not\in D.$ Form $\widehat{D}\in KD(D)$ from $D$ as follows.
% \begin{enumerate}
%     \item[1)] 
% \end{enumerate}}
By Lemma~\ref{lem:behindblock}, if $r^*$ is maximal such that $(r^*,c^*)\in D$ then $r_j>r_i\ge r^*$, and it follows that all cells in rows $r$ satisfying $r_i<r\le r_j$ must be unique in their respective columns. If $r_i>r^*,$ then it also follows that each cell in row $r_i$ is unique in its respective column; however, if $r_i=r^*,$ then each cell in row $r_i$, excluding the cell $(r^*,c^*),$ is unique in its respective column. Define $$n(k)=\begin{cases}
        |\{\tilde{c}~|~(r_k,\tilde{c}),~\tilde{c}>\widehat{c}_1\}|, & i\le k<j \\
        |\{\tilde{c}~|~(r_k,\tilde{c}),~\tilde{c}>\widehat{c}_2\}|,& k=j
    \end{cases}$$
    for $k$ satisfying $i\le k\le j$. Form $\widehat{D}\in \mathcal{P}(D)$ from $D$ as follows.
    \begin{enumerate}
        \item[1)] For $k$ satisfying $i\le k \le j$ in increasing order, apply, in succession, $n(k)$ Kohnert moves at rows $r_k$ through $r_i$ in decreasing order; note that since $r_1>2$, any such Kohnert moves are nontrivial.
        \item[2)] If $r_i<r_j-1$, then apply, in succession, one Kohnert move at rows $r_j$ through $r_i+2$ in decreasing order; otherwise, do nothing.
    \end{enumerate}
    By our assumptions on $D$, it follows that
    \begin{itemize}
        \item $(r_i,\widehat{c}_1),(r_i+1,\widehat{c}_2)\in \widehat{D}$,
        \item $(r_i+1,\tilde{c}),(r_i-1,\widehat{c}_2)\notin \widehat{D}$ for $\tilde{c}>\widehat{c}_2$, and
        \item $(r_i,\tilde{c}),(r_i-1,\widehat{c}_1)\notin \widehat{D}$ for $\tilde{c}>\widehat{c}_1$.
    \end{itemize}
    Thus, applying Proposition~\ref{prop:strucasc} with $r=r_i-1$, $c_1=\widehat{c}_1$, and $c_2=\widehat{c}_2$, it follows that $\mathcal{P}(D)$ is not shellable, a contradiction. Consequently, $D$ satisfies all of the properties listed in Proposition~\ref{prop:formhook}; that is, $D$ is a hook diagram.
\end{proof}

Interestingly, among the families of diagrams considered in Theorem~\ref{thm:hook}, the corresponding Kohnert posets are (EL-)shellable precisely when the associated Kohnert polynomials are multiplicity free.

\begin{theorem}\label{thm:hookmf}
    Let $D$ be a diagram for which either there is at most one cell per column or the first two rows are empty. Then $\mathcal{P}(D)$ is \textup(EL-\textup)shellable if and only if $\mathfrak{K}_D$ is multiplicity-free.
\end{theorem}
\begin{proof}
    For the backward direction, assume that $D$ is not (EL-)shellable, i.e., $D$ is not a hook diagram. Note that in establishing each of the backward directions of Theorem~\ref{thm:hook}, we found that if $D$ did not satisfy at least one of the defining properties (i)--(iv) of hook diagrams given in Proposition~\ref{prop:formhook}, then there existed $D^*\in \mathcal{P}(D)$ which contained a subdiagram of the form described in Proposition~\ref{prop:strucasc}. Given such a diagram $D^*$, let $r,c_1,c_2\in\mathbb{N}$ be such that 
    \begin{itemize}
        \item[\textup{(i)}] $1\le r$ and $c_1<c_2$,
        \item[\textup{(ii)}] $(r+1,c_1),(r+2,c_2)\in D^*$,
        \item[\textup{(iii)}] $(r+2,\tilde{c}),(r,c_2)\notin D^*$ for $\tilde{c}>c_2$, and
        \item[\textup{(iv)}] $(r+1,\tilde{c}),(r,c_1)\notin D^*$ for $\tilde{c}>c_1$.
    \end{itemize} 
    Let $D_1$ denote the diagram formed by applying a Kohnert move at row $r+2$ of $D^*$ followed by row $r+1$ and $D_2$ denote the diagram formed by applying a Kohnert move at row $r+1$ of $D^*$ followed by row $r+2$. Then we have that $$D_1=(D^*\backslash\{(r+2,c_2)\})\cup\{(r,c_2)\}\neq (D^*\backslash\{(r+1,c_1),(r+2,c_2)\})\cup\{(r,c_1),(r+1,c_2)\}=D_2$$ and $wt(D_1)=wt(D_2)$. Thus, $\mathfrak{K}_D$ is not multiplicity free.

    Now, for the forward direction, we break the proof into two cases. Assume that $n>0$ is maximal such that row $n$ of $D$ is nonempty.
    \bigskip

    \noindent
    \textbf{Case 1:} $D$ contains at most one cell per column. Let $$C_1=\{c~|~(1,c)\in D\}=\{c^1_1,\hdots,c^1_{m_1}\}$$ and $$C_2=\{c~|~(r,c)\in D~\text{and}~c\notin C_1\}=\{c^2_1<\hdots<c^2_{m_2}\}.$$ Note that for $\tilde{D}\in \mathcal{P}(D)$, $wt(\tilde{D})=\prod_{i=1}^nx_i^{a_i}$ where $a_1\ge m_1$ and $\sum_{i=1}^na_i-m_1=m_2$. Now, it is straightforward to show that if $\mathcal{P}(D)$ is (EL-)shellable and $(r_i,c^2_i)\in\tilde{D}\in \mathcal{P}(D)$ for $1\le i\le m_2$, then $r_1\ge r_2\ge\cdots\ge r_{m_2}$. Thus, if $\prod_{i=1}^nx_i^{a_i}$ is a monomial in $\mathfrak{K}_D$, then it corresponds uniquely to the diagram $$\tilde{D}=\{(1,c)~|~c\in C_1\}\cup\{(1,c^2_{m_2-i+1})~|~1\le i\le a_1-m_1+1\}\cup\bigcup_{j=2}^n\left\{(j,c^2_{m_2-i+1})~|~\sum_{k=1}^{j-1}a_k< i\le\sum_{k=1}^{j}a_k\right\}.$$ Consequently, $\mathfrak{K}_D$ is multiplicity free.
    \bigskip

    \noindent
    \textbf{Case 2:} The first two rows of $D$ are empty. Assume that $D\in H(r_1,r_2;C)$ with $C=\{c_1<\hdots~<~c_m\}$. Since we have already considered the situation where $D$ contains at most one cell per column in Case 1, we may assume that $L=r_2-r_1>0$. Take $\tilde{D}\in \mathcal{P}(D)$ and assume that $r$ is maximal such that $(r,c_m)\in \tilde{D}$. Then considering Proposition~\ref{prop:formhook} (i)--(iii), there exists $1\le i_1<\cdots<i_{L}< n$ such that $wt(\tilde{D})=\prod_{j=1}^Lx_{i_j}\prod_{k=i_L+1}^nx^{a_k}_k$, where $\sum_{k=i_L+1}^na_k=|C|$. Now, considering Proposition~\ref{prop:formhook} (iv), if $\prod_{j=1}^Lx_{i_j}\prod_{k=i_L+1}^nx^{a_k}_k$ is a monomial of $\mathfrak{K}_D$, then it  corresponds uniquely to the diagram $$\tilde{D}=\{(i_j,c_m)~|~1\le j\le L\}\cup\{(i_L+1,c_{m-i+1})~|~0< i\le a_{i_L+1}\}$$ $$\cup\bigcup_{j=2}^{n-i_L}\left\{(i_L+j,c_{m-i+1})~|~\sum_{k=i_L+1}^{i_L+j-1}a_k< i\le\sum_{k=i_L+1}^{i_L+j}a_k\right\}.$$ Consequently, $\mathfrak{K}_D$ is multiplicity free.
\end{proof}

\begin{remark}
    In the following section, we find, in particular, that the equivalence between \textup(EL-\textup)shellability of a Kohnert poset and the associated Kohnert polynomial being multiplicity free does not hold in general. More specifically, for key diagrams we find that if a Kohnert poset is graded and EL-shellable, then the associated Kohnert polynomial is multiplicity free; however, there are examples of key diagrams whose Kohnert polynomials are multiplicity free and whose Kohnert posets are not shellable.
\end{remark}

It remains to consider diagrams that contain at least one cell within the first two rows and for which at least one column contains more than one cell. While we do not obtain a complete classification in this case, in the following section we consider the special case of key diagrams.

\section{Key diagrams}\label{sec:key}

In this section, we consider Kohnert posets associated with key diagrams. Recall from Section~\ref{sec:prelim} that key diagrams are defined by weak compositions: given a weak composition $\mathbf{a}=(a_1,\hdots,a_n)$, the associated key diagram  is defined by $$\mathbb{D}(\mathbf{a})=\bigcup_{i=1}^n\{(i,j)~|~1\le j\le a_i\}.$$
Ongoing, we let $|\mathbf{a}|=\sum_{i=1}^na_i$.

The main result of this section is a characterization of the key diagrams that generate pure, shellable Kohnert posets in terms of their associated compositions. To state the main result, we require the notion of ``pure composition" introduced in \textbf{\cite{KPoset1}}. 

\begin{definition}
    A weak composition $\mathbf{a}=(a_1,\hdots,a_n)$ is called a \textbf{pure composition} if there exists no $1\le j_1<j_2<j_3\le n$ such that 
    \begin{itemize}
        \item $a_{j_1}<a_{j_2}<a_{j_3}$, 
        \item $a_{j_1}<a_{j_3}<a_{j_2}$, or
        \item $a_{j_1}+1<a_{j_2}=a_{j_3}$.
    \end{itemize}
\end{definition} 

In \textbf{\cite{KPoset1}}, the following properties of pure compositions and Kohnert posets associated with key diagrams are established. To set the notation, given a weak composition $\mathbf{a}=(a_1,\hdots,a_n)$, ongoing we set $\max(\mathbf{a})=\max\{a_i~|~1\le i\le n\}$ and $\min(\mathbf{a})=\min\{a_i~|~1\le i\le n\}$, i.e., $\max(\mathbf{a})$ (resp., $\min(\mathbf{a})$) is the maximal (resp., minimal) entry of $\mathbf{a}$.

\begin{lemma}[Lemma 6.8, \textbf{\cite{KPoset1}}]\label{lem:puredecomp}
    If $\mathbf{a}=(a_1,\hdots,a_n)$ is a pure composition, then there exists $i_1=1<\hdots<i_{m}<n=i_{m+1}-1$ such that $\alpha_j=(a_{i_j},\hdots,a_{i_{j+1}-1})$ for $1\le j\le m$ satisfies $\min(\alpha_{j-1})\ge \max(\alpha_j)$ for $1<j\le m$. Moreover, each $\alpha_j=(a_{i_j},\hdots,a_{i_{j+1}-1})$ is of one of the following forms:
    \begin{enumerate}
         \item[\textup{(i)}] $a_{i_j}\ge \hdots\ge a_{i_{j+1}-1}$; that is, $\alpha_j$ is a weakly decreasing sequence. 
        \item[\textup{(ii)}] There exists $p\in\mathbb{Z}_{\ge 0}$ such that $a_{i_j}=p$ and $\{a_{i_j},\hdots,a_{i_{j+1}-1}\}=\{p,p+1\}$; that is, all entries of $\alpha_j$ are $p$ or $p+1$ for some $p\in \mathbb{Z}_{\ge 0}$, the first entry is $p$, and some other entry must be equal to $p+1$. 
        \item[\textup{(iii)}] $a_{i_j}\ge\hdots\ge a_{i_{j+1}-2}<a_{i_{j+1}-1}-1$; that is, the entries of $\alpha_j$ are in decreasing order, except the last one which is at least two larger than the penultimate one. 
        \item[\textup{(iv)}] There exist $p\in\mathbb{Z}_{\ge 0}$ and $i_j^*\in\mathbb{Z}_{> 0}$ with $i_j+1<i_j^*<i_{j+1}-1$ such that $a_{i_j}=p$, $\{a_{i_j},\hdots,a_{i_j^*-1}\}=\{p,p+1\}$, $p>a_{i_j^*}\ge \hdots\ge a_{i_{j+1}-2}$, and $a_{i_{j+1}-1}=p+1$.
    \end{enumerate}
\end{lemma}

\begin{theorem}[Theorem 6.1, \textbf{\cite{KPoset1}}]\label{thm:keybound}
    For all weak compositions $\mathbf{a}$, the Kohnert poset $\mathcal{P}(\mathbb{D}(\mathbf{a}))$ is bounded.
\end{theorem}

\begin{theorem}[Theorem 6.5, \textbf{\cite{KPoset1}}]\label{thm:kgrade}
    Let $\mathbf{a}$ be a weak composition. Then $\mathcal{P}(\mathbb{D}(\mathbf{a}))$ is pure if and only if $\mathbf{a}$ is pure.
\end{theorem}

\begin{remark}
    To be precise, Theorem 6.5 of \textup{\textbf{\cite{KPoset1}}} actually establishes that for a weak composition $\mathbf{a}$, the Kohnert poset $\mathcal{P}(\mathbb{D}(\mathbf{a}))$ is ranked if and only if $\mathbf{a}$ is pure. However, it is straight-forward to verify that a bounded poset is pure if and only if it is ranked.
\end{remark}

Ongoing, given a pure composition $\mathbf{a}$, we will refer to a decomposition $(\alpha_1,\hdots,\alpha_m)$ of $\mathbf{a}$ as shown to exist in Lemma~\ref{lem:puredecomp} as a \textbf{pure decomposition} of $\mathbf{a}$. Moreover, we refer to the weak compositions of the forms described in $(i)-(iv)$  of Lemma~\ref{lem:puredecomp}, i.e., those weak compositions that form the building blocks of pure compositions, as \textbf{basic pure compositions}.

\begin{ex}
    Consider the pure composition $$\mathbf{a}=(15,15,15,14,14,15,14,15,13,11,10,7,15,7,6,5,4,6,3,3,4,3,4,3,2,1,0,1).$$ One choice of pure decomposition of $\mathbf{a}$ is given by $$\alpha_1=(15,15,15,14,14,15,14,15,13,11,10,7,15),\quad\quad\alpha_2=(7,6,5,4,6),\quad\quad\alpha_3=(3,3,4,3,4),$$ $$\alpha_4=(3,2,1),\quad\quad\text{and}\quad\quad\alpha_5=(0,1).$$ 
    Note that $\alpha_1$ is of type \textup{(iv)}, $\alpha_2$ is of type \textup{(iii)}, $\alpha_3$ and $\alpha_5$ are of type \textup{(ii)}, and $\alpha_4$ is of type \textup{(i)}.
\end{ex}

\begin{remark}
    In \textup{\textbf{\cite{KPoset1}}}, given a pure composition $\mathbf{a}$, a procedure for finding a choice of pure decomposition of $\mathbf{a}$ is provided within the proof of Lemma 6.8.
\end{remark}

We claim that Theorem~\ref{thm:kgrade} can be extended to give a characterization of pure, (EL-)shellable Kohnert posets associated with key diagrams. In particular, in the remainder of this section we establish the following.

\begin{theorem}\label{thm:kshell}
     Let $\mathbf{a}$ be a weak composition. Then $\mathcal{P}(\mathbb{D}(\mathbf{a}))$ is pure and \textup(EL-\textup)shellable if and only if $\mathbf{a}$ is pure.
\end{theorem}

\subsection{Necessity}

In this section we find necessary conditions for a key diagram to be associated with a shellable Kohnert poset in terms of the corresponding composition avoiding certain patterns. In particular, we establish the following.

\begin{prop}\label{prop:keynecrank}
    Let $\mathbf{a}=(a_1,\hdots,a_n)$ be a weak composition. Suppose that either 
    \begin{enumerate}
        \item[\textup{(a)}] there exist $1\le i_1<i_2<i_3\le n$ for which one of the following holds
    \begin{itemize}
    \item[\textup{(i)}] $a_{i_1}<a_{i_2}<a_{i_3}$
    \item[\textup{(ii)}] $a_{i_1}\le a_{i_3}-3\le a_{i_2}-3$;
    \end{itemize}
    \end{enumerate}
    or
    \begin{enumerate}
        \item[\textup{(b)}] there exist $1\le j_1<j_2<j_3<j_4\le n$ for which one of the following holds
    \begin{itemize}
    \item[\textup{(i)}] $a_{j_1}\le a_{j_2}<a_{j_3}-1\le a_{j_4}-1$
    \item[\textup{(ii)}] $a_{j_1}\le a_{j_2}<a_{j_4}< a_{j_3}$
    \item[\textup{(iii)}] $a_{j_2}< a_{j_1}<a_{j_4}< a_{j_3}$
    \item[\textup{(iv)}] $a_{j_2}< a_{j_1}<a_{j_3}\le a_{j_4}.$
\end{itemize}
    \end{enumerate}
Then $\mathcal{P}(\mathbb{D}(\mathbf{a}))$ is not shellable.
\end{prop}

\begin{remark}
    Note that Proposition~\ref{prop:keynecrank} establishes more than is needed to prove Theorem~\ref{thm:kshell}; in particular, the proof of Theorem~\ref{thm:kshell} does not require part $(b$-$iv)$ of Proposition~\ref{prop:keynecrank}. We include $(b$-$iv)$ as, altogether, we conjecture that a weak composition $\mathbf{a}$ avoiding the patterns of Proposition~\ref{prop:keynecrank} is equivalent to the Kohnert poset $\mathcal{P}(\mathbb{D}(\mathbf{a}))$ being \textup(EL-\textup)shellable, with no restrictions involving purity \textup(see Conjecture~\ref{conj:key}\textup).
\end{remark}

To prove Proposition~\ref{prop:keynecrank}, we first show that the result holds if we assume that the patterns described therein are followed by consecutively occurring terms in the composition.

\begin{lemma}\label{lem:key1}
Let $\mathbf{a}=(a_1,\hdots,a_n)$ be a weak composition. Suppose that either
\begin{enumerate}
    \item[\textup{(a)}] there exists $1\le i<n-2$ for which one of the following holds 
\begin{itemize}
    \item[\textup{(i)}] $a_i<a_{i+1}<a_{i+2}$
    \item[\textup{(ii)}] $a_i\le a_{i+2}-3\le a_{i+1}-3$;
\end{itemize}
\end{enumerate}
or
\begin{enumerate}
    \item[\textup{(b)}] there exists $1\le j<n-3$ for which one of the following holds
\begin{itemize}
    \item[\textup{(i)}] $a_j\le a_{j+1}<a_{j+2}-1\le a_{j+3}-1$
    \item[\textup{(ii)}] $a_j\le a_{j+1}<a_{j+3}< a_{j+2}$
    \item[\textup{(iii)}] $a_{j+1}< a_{j}<a_{j+3}< a_{j+2}$
    \item[\textup{(iv)}] $a_{j+1}< a_{j}<a_{j+2}\le a_{j+3}$.
\end{itemize}
\end{enumerate}
Then $\mathcal{P}(\mathbb{D}(\mathbf{a}))$ is not shellable.
\end{lemma}

    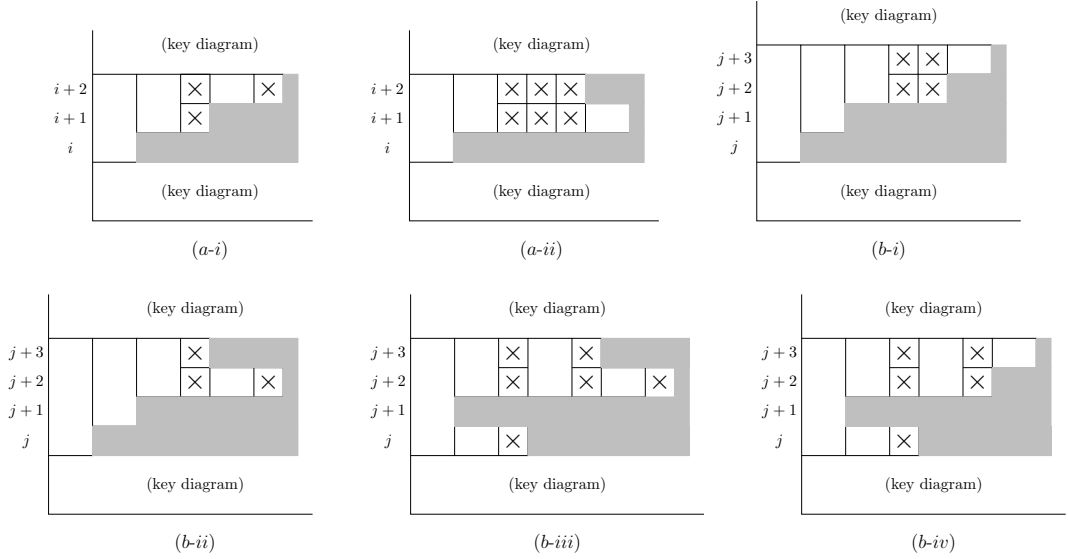
\begin{figure}[H]
        \centering
        $$\scalebox{0.6}{\begin{tikzpicture}[scale=0.65]
        \node at (-0.75,4.5) {$i+2$};
    \node at (-0.75,3.5) {$i+1$};
    \node at (-0.75,2.5) {$i$};
    \draw (0,2)--(1.5,2)--(1.5,5)--(0,5);
    \draw (1.5,3)--(3,3)--(3,5)--(1.5,5);
    \draw (3,3)--(4,3)--(4,5)--(3,5);
    \draw (0,6.5)--(0,0)--(7.5,0);
    \draw (3,4)--(4,4);
    \draw (4,4)--(5.5,4)--(5.5,5)--(3,5);
    \draw (5.5,4)--(6.5,4)--(6.5,5)--(5.5,5);
    \node at (3.5,3.5) {$\bigtimes$};
    \node at (3.5,4.5) {$\bigtimes$};
    \node at (6,4.5) {$\bigtimes$};
    \filldraw[draw=lightgray,fill=lightgray] (1.5,2) rectangle (7,3);
    \filldraw[draw=lightgray,fill=lightgray] (4,3) rectangle (7,4);
    \filldraw[draw=lightgray,fill=lightgray] (6.5,4) rectangle (7,5);
    \node at (4,6) {(key diagram)};
    \node at (4,1) {(key diagram)};
    \node at (4,-1) {\large$(a$-$i)$};
\end{tikzpicture}}\quad\quad \scalebox{0.6}{\begin{tikzpicture}[scale=0.65]
    \node at (-0.75,4.5) {$i+2$};
    \node at (-0.75,3.5) {$i+1$};
    \node at (-0.75,2.5) {$i$};
    \draw (0,2)--(1.5,2)--(1.5,5)--(0,5);
    \draw (1.5,3)--(3,3)--(3,5)--(1.5,5);
    \draw (3,3)--(4,3)--(4,5)--(3,5);
    \draw (4,3)--(5,3)--(5,5)--(4,5);
    \draw (5,3)--(6,3)--(6,5)--(5,5);
    \draw (6,3)--(7.5,3)--(7.5,4)--(6,4);
    \draw (0,6.5)--(0,0)--(8.5,0);
    \draw (3,4)--(6,4);
    
    \node at (3.5,3.5) {$\bigtimes$};
    \node at (3.5,4.5) {$\bigtimes$};
    \node at (4.5,3.5) {$\bigtimes$};
    \node at (4.5,4.5) {$\bigtimes$};
    \node at (5.5,3.5) {$\bigtimes$};
    \node at (5.5,4.5) {$\bigtimes$};
    \filldraw[draw=lightgray,fill=lightgray] (1.5,2) rectangle (8,3);
    \filldraw[draw=lightgray,fill=lightgray] (7.5,3) rectangle (8,4);
    \filldraw[draw=lightgray,fill=lightgray] (6,4) rectangle (8,5);
    \node at (4.5,6) {(key diagram)};
    \node at (4.5,1) {(key diagram)};
    \node at (4.5,-1) {\large$(a$-$ii)$};
\end{tikzpicture}}\quad\quad \scalebox{0.6}{\begin{tikzpicture}[scale=0.65]
    \node at (-0.75,5.5) {$j+3$};
    \node at (-0.75,4.5) {$j+2$};
    \node at (-0.75,3.5) {$j+1$};
    \node at (-0.75,2.5) {$j$};
    \draw (0,2)--(1.5,2)--(1.5,6)--(0,6);
    \draw (1.5,3)--(3,3)--(3,6)--(1.5,6);
    \draw (3,4)--(4.5,4)--(4.5,6)--(3,6);
    \draw (4.5,4)--(5.5,4)--(5.5,5)--(6.5,5)--(6.5,6)--(4.5,6);
    \draw (4.5,5)--(5.5,5)--(5.5,6);
    \draw (5.5,4)--(6.5,4)--(6.5,5);
    \node at (5,4.5) {$\bigtimes$};
    \node at (6,4.5) {$\bigtimes$};
    \node at (5,5.5) {$\bigtimes$};
    \node at (6,5.5) {$\bigtimes$};
    \draw (6.5,5)--(8,5)--(8,6)--(6.5,6);
    \node at (4.5,7) {(key diagram)};
    \node at (4.5,1) {(key diagram)};
    \draw (0,7.5)--(0,0)--(9,0);
    \filldraw[draw=lightgray,fill=lightgray] (1.5,2) rectangle (8.5,3);
    \filldraw[draw=lightgray,fill=lightgray] (3,3) rectangle (8.5,4);
    \filldraw[draw=lightgray,fill=lightgray] (6.5,4) rectangle (8.5,5);
    \filldraw[draw=lightgray,fill=lightgray] (8,5) rectangle (8.5,6);
    
    \node at (4.5,-1) {\large$(b$-$i)$};
\end{tikzpicture}}$$ $$\scalebox{0.6}{\begin{tikzpicture}[scale=0.65]
    \node at (-0.75,5.5) {$j+3$};
    \node at (-0.75,4.5) {$j+2$};
    \node at (-0.75,3.5) {$j+1$};
    \node at (-0.75,2.5) {$j$};

    \draw (0,2)--(1.5,2)--(1.5,6)--(0,6);
    \draw (1.5,3)--(3,3)--(3,6)--(1.5,6);
    \draw (3,4)--(4.5,4)--(4.5,6)--(3,6);
    \draw (4.5,4)--(5.5,4)--(5.5,6)--(4.5,6);
    \draw (4.5,5)--(5.5,5);
    \draw (5.5,4)--(7,4)--(7,5)--(5.5,5);
    \draw (7,4)--(8,4)--(8,5)--(7,5);
    \node at (5,4.5) {$\bigtimes$};
    \node at (5,5.5) {$\bigtimes$};
    \node at (7.5,4.5) {$\bigtimes$};
    \node at (5,7) {(key diagram)};
    \node at (5,1) {(key diagram)};

    \draw (0,7.5)--(0,0)--(9,0);
    \filldraw[draw=lightgray,fill=lightgray] (1.5,2) rectangle (8.5,3);
    \filldraw[draw=lightgray,fill=lightgray] (3,3) rectangle (8.5,4);
    \filldraw[draw=lightgray,fill=lightgray] (8,4) rectangle (8.5,5);
    \filldraw[draw=lightgray,fill=lightgray] (5.5,5) rectangle (8.5,6);
    
    \node at (5,-1) {\large$(b$-$ii)$};
\end{tikzpicture}}\quad\quad\scalebox{0.6}{\begin{tikzpicture}[scale=0.65]
    \node at (-0.75,5.5) {$j+3$};
    \node at (-0.75,4.5) {$j+2$};
    \node at (-0.75,3.5) {$j+1$};
    \node at (-0.75,2.5) {$j$};
    \node at (5,7) {(key diagram)};
    \node at (5,1) {(key diagram)};
    \draw (0,7.5)--(0,0)--(10,0);

    \draw (0,2)--(1.5,2)--(1.5,6)--(0,6);
    \draw (1.5,2)--(3,2)--(3,3)--(1.5,3);
    \draw (3,2)--(4,2)--(4,3)--(3,3);
    \node at (3.5,2.5) {$\bigtimes$};
    \draw (1.5,4)--(3,4)--(3,6)--(1.5,6);
    \draw (3,4)--(4,4)--(4,6)--(3,6);
    \draw (3,5)--(4,5);
    \node at (3.5,4.5) {$\bigtimes$};
    \node at (3.5,5.5) {$\bigtimes$};
    \draw (4,4)--(5.5,4)--(5.5,6)--(4,6);
    \draw (5.5,4)--(6.5,4)--(6.5,6)--(5.5,6);
    \node at (6,4.5) {$\bigtimes$};
    \node at (6,5.5) {$\bigtimes$};
    \draw (6.5,4)--(8,4)--(8,5)--(6.5,5);
    \draw (9,4)--(9,4)--(9,5)--(8,5);
    \node at (8.5,4.5) {$\bigtimes$};
    \draw (5.5,5)--(6.5,5);

    \filldraw[draw=lightgray,fill=lightgray] (4,2) rectangle (9.5,3);

    \filldraw[draw=lightgray,fill=lightgray] (1.5,3) rectangle (9.5,4);

    \filldraw[draw=lightgray,fill=lightgray] (9,4) rectangle (9.5,5);

    \filldraw[draw=lightgray,fill=lightgray] (6.5,5) rectangle (9.5,6);
    
    \node at (5,-1) {\large$(b$-$iii)$};
\end{tikzpicture}}\quad\quad \scalebox{0.6}{\begin{tikzpicture}[scale=0.65]
    \node at (-0.75,5.5) {$j+3$};
    \node at (-0.75,4.5) {$j+2$};
    \node at (-0.75,3.5) {$j+1$};
    \node at (-0.75,2.5) {$j$};
    \node at (4.5,7) {(key diagram)};
    \node at (4.5,1) {(key diagram)};
    \draw (0,7.5)--(0,0)--(9,0);

    \draw (0,2)--(1.5,2)--(1.5,6)--(0,6);
    \draw (1.5,2)--(3,2)--(3,3)--(1.5,3);
    \draw (3,2)--(4,2)--(4,3)--(3,3);
    \node at (3.5,2.5) {$\bigtimes$};
    \draw (1.5,4)--(3,4)--(3,6)--(1.5,6);
    \draw (3,4)--(4,4)--(4,6)--(3,6);
    \draw (3,5)--(4,5);
    \node at (3.5,4.5) {$\bigtimes$};
    \node at (3.5,5.5) {$\bigtimes$};
    \draw (4,4)--(5.5,4)--(5.5,6)--(4,6);
    \draw (5.5,4)--(6.5,4)--(6.5,6)--(5.5,6);
    \node at (6,4.5) {$\bigtimes$};
    \node at (6,5.5) {$\bigtimes$};
    \draw (6.5,5)--(8,5)--(8,6)--(6.5,6);
    \draw (5.5,5)--(6.5,5);

    \filldraw[draw=lightgray,fill=lightgray] (4,2) rectangle (8.5,3);

    \filldraw[draw=lightgray,fill=lightgray] (1.5,3) rectangle (8.5,4);

    \filldraw[draw=lightgray,fill=lightgray] (6.5,4) rectangle (8.5,5);

    \filldraw[draw=lightgray,fill=lightgray] (8,5) rectangle (8.5,6);
    
    \node at (4.5,-1) {\large$(b$-$iv)$};
\end{tikzpicture}}$$
        \caption{Forms of key diagrams for weak compositions described in Lemma~\ref{lem:key1}}
        \label{fig:lemkeyconseca}
    \end{figure}

\begin{proof}
    $(a$-$i)$ By assumption, we have that
    \begin{itemize}
        \item $(i+1,a_{i+1}),(i+2,a_{i+2})\in \mathbb{D}(\mathbf{a})$ with $a_{i+1}<a_{i+2}$,
        \item $(i,a_{i+2}),(i+2,\tilde{c})\notin \mathbb{D}(\mathbf{a})$ for $\tilde{c}>a_{i+2}$, and
        \item $(i,a_{i+1}),(i+1,\tilde{c})\notin \mathbb{D}(\mathbf{a})$ for $\tilde{c}>a_{i+1}$.
    \end{itemize}
    Thus, applying Proposition~\ref{prop:strucasc} with $r=i$, $c_1=a_{i+1}$, and $c_2=a_{i+2}$, the result follows.

    $(a$-$ii)$ Let $D$ denote the diagram obtained from $\mathbb{D}(\mathbf{a})$ by applying $a_{i+1}-a_{i+2}+2$ Kohnert moves at row $i+1$ (see Figure~\ref{fig:key1aii})
    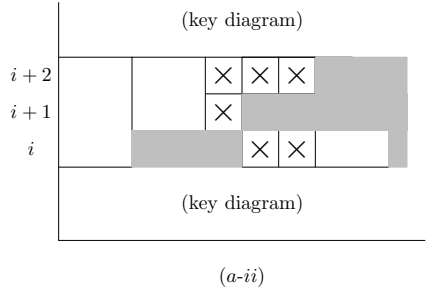
\begin{figure}[H]
        \centering
        $$\scalebox{0.75}{\begin{tikzpicture}[scale=0.65]
        \node at (-0.75,4.5) {$i+2$};
    \node at (-0.75,3.5) {$i+1$};
    \node at (-0.75,2.5) {$i$};
    \draw (0,2)--(2,2)--(2,5)--(0,5);
    \draw (2,3)--(4,3)--(4,5)--(2,5);
    \draw (4,3)--(5,3)--(5,5)--(4,5);
    \draw (5,4)--(6,4)--(6,5)--(5,5);
    \draw (6,4)--(7,4)--(7,5)--(6,5);
    \draw (7,2)--(9,2)--(9,3)--(7,3);
    \draw (0,6.5)--(0,0)--(10,0);
    \draw (5,2)--(7,2)--(7,3)--(5,3)--(5,2);
    \draw (6,2)--(6,3);
    \draw (4,4)--(5,4);
    \draw (5,4)--(7,4)--(7,5)--(5,5);
    \draw (7,4)--(8,4)--(8,5)--(7,5);
    \node at (4.5,3.5) {$\bigtimes$};
    \node at (4.5,4.5) {$\bigtimes$};
    \node at (5.5,2.5) {$\bigtimes$};
    \node at (5.5,4.5) {$\bigtimes$};
    \node at (6.5,2.5) {$\bigtimes$};
    \node at (6.5,4.5) {$\bigtimes$};
    \filldraw[draw=lightgray,fill=lightgray] (2,2) rectangle (5,3);
    \filldraw[draw=lightgray,fill=lightgray] (9,2) rectangle (9.5,3);
    \filldraw[draw=lightgray,fill=lightgray] (5,3) rectangle (9.5,4);
    \filldraw[draw=lightgray,fill=lightgray] (7,4) rectangle (9.5,5);
    \node at (5,6) {(key diagram)};
    \node at (5,1) {(key diagram)};
    \node at (5,-1) {$(a$-$ii)$};
\end{tikzpicture}}$$
        \caption{Diagram related to key diagram in case $(a$-$ii)$}
        \label{fig:key1aii}
    \end{figure}
    Note that
    \begin{itemize}
        \item $(i+1,a_{i+2}-2),(i+2,a_{i+2}-2),(i+2,a_{i+2})\in D$,
        \item $|\{(i+2,\tilde{c})\in D~|~a_{i+2}-2<\tilde{c}<a_{i+2}\}|=1>0,$
        \item $(i+2,\tilde{c})\notin D$ for $\tilde{c}>a_{i+2}$,
        \item $(i+1,\tilde{c})\notin D$ for $\tilde{c}>a_{i+2}-2$, and
        \item $(i,a_{i+2}-2)\notin D$.
    \end{itemize}
    Thus, applying Corollary~\ref{cor:strucblock} with $r=i$, $c_1=a_{i+2}-2$, and $c_2=a_{i+2}$, the result follows.

    $(b$-$i)$,$(b$-$iv)$ If $a_{j+2}<a_{j+3}$, then the result follows as in $(a$-$i)$ taking $i=j+1$. So, assume that $a_{j+2}=a_{j+3}$. Let $D$ denote the diagram obtained from $\mathbb{D}(\mathbf{a})$ by applying a single Kohnert move at row $j+2$ followed by a single Kohnert move at row $j+1$ (see Figure~\ref{fig:key1bi}).

    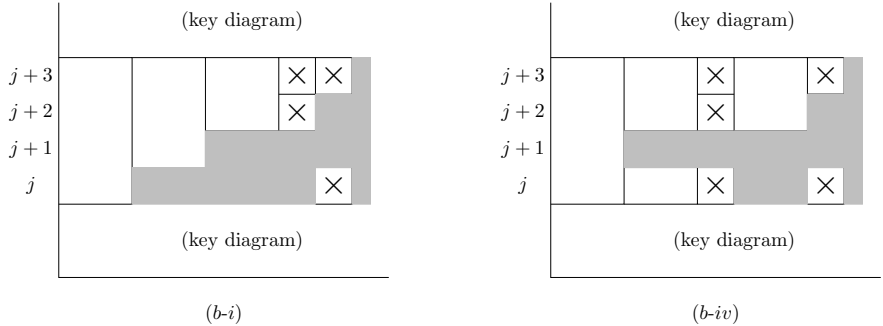
\begin{figure}[H]
        \centering
        $$\scalebox{0.75}{\begin{tikzpicture}[scale=0.65]
    \node at (-0.75,5.5) {$j+3$};
    \node at (-0.75,4.5) {$j+2$};
    \node at (-0.75,3.5) {$j+1$};
    \node at (-0.75,2.5) {$j$};
    \draw (0,2)--(2,2)--(2,6)--(0,6);
    \draw (2,3)--(4,3)--(4,6)--(2,6);
    \draw (4,4)--(6,4)--(6,6)--(4,6);
    \draw (6,4)--(7,4)--(7,5)--(8,5)--(8,6)--(6,6);
    \draw (6,5)--(7,5)--(7,6);
    \draw (7,2)--(8,2)--(8,3)--(7,3)--(7,2);
    \node at (7.5,2.5) {$\bigtimes$};
    \node at (6.5,4.5) {$\bigtimes$};
    \node at (6.5,5.5) {$\bigtimes$};
    \node at (7.5,5.5) {$\bigtimes$};
    \node at (5,7) {(key diagram)};
    \node at (5,1) {(key diagram)};
    \draw (0,7.5)--(0,0)--(9,0);
    \filldraw[draw=lightgray,fill=lightgray] (2,2) rectangle (7,3);
    \filldraw[draw=lightgray,fill=lightgray] (8,2) rectangle (8.5,3);
    \filldraw[draw=lightgray,fill=lightgray] (4,3) rectangle (8.5,4);
    \filldraw[draw=lightgray,fill=lightgray] (7,4) rectangle (8.5,5);
    \filldraw[draw=lightgray,fill=lightgray] (8,5) rectangle (8.5,6);
    
    \node at (4.5,-1) {$(b$-$i)$};
\end{tikzpicture}}\quad\quad\quad\quad \scalebox{0.75}{\begin{tikzpicture}[scale=0.65]
    \node at (-0.75,5.5) {$j+3$};
    \node at (-0.75,4.5) {$j+2$};
    \node at (-0.75,3.5) {$j+1$};
    \node at (-0.75,2.5) {$j$};
    \node at (5,7) {(key diagram)};
    \node at (5,1) {(key diagram)};
    \draw (0,7.5)--(0,0)--(9,0);

    \draw (0,2)--(2,2)--(2,6)--(0,6);
    \draw (2,2)--(4,2)--(4,3)--(2,3);
    \draw (4,2)--(5,2)--(5,3)--(4,3);
    \node at (4.5,2.5) {$\bigtimes$};
    \node at (4.5,4.5) {$\bigtimes$};
    \draw (2,4)--(4,4)--(4,6)--(2,6);
    \draw (4,4)--(5,4)--(5,6)--(4,6);
    \draw (4,5)--(5,5);
    \node at (4.5,5.5) {$\bigtimes$};
    \draw (5,4)--(7,4)--(7,6)--(5,6);
    \draw (7,5)--(8,5)--(8,6)--(7,6);
    \node at (7.5,5.5) {$\bigtimes$};

    \draw (7,2)--(8,2)--(8,3)--(7,3)--(7,2);
    \node at (7.5,2.5) {$\bigtimes$};
    \draw (7,5)--(8,5);

    \filldraw[draw=lightgray,fill=lightgray] (5,2) rectangle (7,3);
    \filldraw[draw=lightgray,fill=lightgray] (8,2) rectangle (8.5,3);

    \filldraw[draw=lightgray,fill=lightgray] (2,3) rectangle (8.5,4);

    \filldraw[draw=lightgray,fill=lightgray] (7,4) rectangle (8.5,5);

    \filldraw[draw=lightgray,fill=lightgray] (8,5) rectangle (8.5,6);
    
    \node at (4.5,-1) {$(b$-$iv)$};
\end{tikzpicture}}$$
        \caption{Diagrams related to key diagram in cases $(b$-$i)$ (left) and $(b$-$iv)$ (right)}
        \label{fig:key1bi}
    \end{figure}
    \noindent
    By assumption, 
    \begin{itemize}
        \item $(j+3,a_{j+3}),(j+2,a_{j+3}-1)\in D$,
        \item $(j+1,a_{j+3}),(j+3,\tilde{c})\notin D$ for $\tilde{c}>a_{j+3}$, and
        \item $(j+1,a_{j+3}-1),(j+2,\tilde{c})\notin D$ for $\tilde{c}>a_{j+3}-1$.
    \end{itemize}
    Thus, applying Proposition~\ref{prop:strucasc} with $r=j+1$, $c_1=a_{j+3}-1$, and $c_2=a_{j+3}$, the result follows.

    $(b$-$ii)$,$(b$-$iii)$ Let $D$ denote the diagram obtained from $\mathbb{D}(\mathbf{a})$ by applying a single Kohnert move at row $j+3$ (see Figure~\ref{fig:key1bii}).
    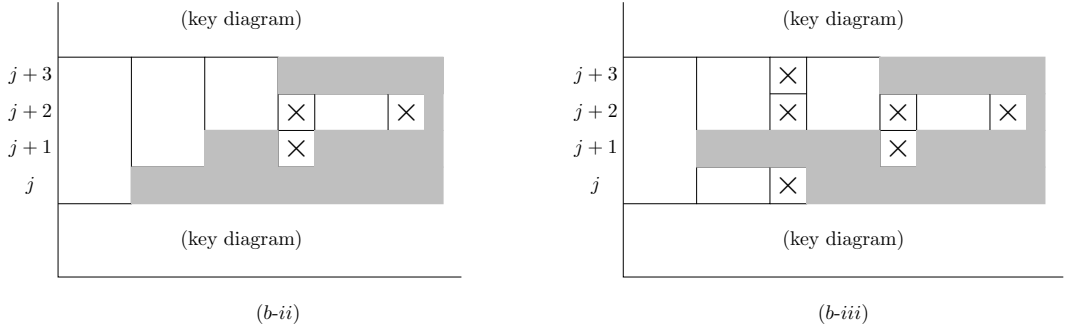
\begin{figure}[H]
        \centering
        $$\scalebox{0.75}{\begin{tikzpicture}[scale=0.65]
    \node at (-0.75,5.5) {$j+3$};
    \node at (-0.75,4.5) {$j+2$};
    \node at (-0.75,3.5) {$j+1$};
    \node at (-0.75,2.5) {$j$};

    \draw (0,2)--(2,2)--(2,6)--(0,6);
    \draw (2,3)--(4,3)--(4,6)--(2,6);
    \draw (4,4)--(6,4)--(6,6)--(4,6);
    \draw (6,4)--(7,4)--(7,5)--(6,5);
    \draw (6,5)--(7,5);
    \draw (7,4)--(9,4)--(9,5)--(7,5);
    \draw (9,4)--(10,4)--(10,5)--(9,5);
    \node at (6.5,3.5) {$\bigtimes$};
    \draw (6,4)--(6,3)--(7,3)--(7,4);
    \node at (6.5,4.5) {$\bigtimes$};
    \node at (9.5,4.5) {$\bigtimes$};
    \node at (5,7) {(key diagram)};
    \node at (5,1) {(key diagram)};

    \draw (0,7.5)--(0,0)--(11,0);
    \filldraw[draw=lightgray,fill=lightgray] (2,2) rectangle (10.5,3);
    \filldraw[draw=lightgray,fill=lightgray] (4,3) rectangle (6,4);
    \filldraw[draw=lightgray,fill=lightgray] (7,3) rectangle (10.5,4);
    \filldraw[draw=lightgray,fill=lightgray] (10,4) rectangle (10.5,5);
    \filldraw[draw=lightgray,fill=lightgray] (6,5) rectangle (10.5,6);
    
    \node at (6,-1) {$(b$-$ii)$};
\end{tikzpicture}}\quad\quad\quad\quad \scalebox{0.75}{\begin{tikzpicture}[scale=0.65]
    \node at (-0.75,5.5) {$j+3$};
    \node at (-0.75,4.5) {$j+2$};
    \node at (-0.75,3.5) {$j+1$};
    \node at (-0.75,2.5) {$j$};
    \node at (6,7) {(key diagram)};
    \node at (6,1) {(key diagram)};
    \draw (0,7.5)--(0,0)--(12,0);

    \draw (0,2)--(2,2)--(2,6)--(0,6);
    \draw (2,2)--(4,2)--(4,3)--(2,3);
    \draw (4,2)--(5,2)--(5,3)--(4,3);
    \node at (4.5,2.5) {$\bigtimes$};
    \draw (2,4)--(4,4)--(4,6)--(2,6);
    \draw (4,4)--(5,4)--(5,6)--(4,6);
    \draw (4,5)--(5,5);
    \node at (4.5,4.5) {$\bigtimes$};
    \node at (4.5,5.5) {$\bigtimes$};
    \draw (5,4)--(7,4)--(7,6)--(5,6);
    \draw (7,4)--(8,4)--(8,5)--(7,5);
    \node at (7.5,4.5) {$\bigtimes$};
    \node at (7.5,3.5) {$\bigtimes$};
    \draw (7,4)--(7,3)--(8,3)--(8,4);
    \draw (8,4)--(10,4)--(10,5)--(8,5);
    \draw (10,4)--(11,4)--(11,5)--(10,5);
    \node at (10.5,4.5) {$\bigtimes$};
    \draw (7,5)--(8,5);

    \filldraw[draw=lightgray,fill=lightgray] (5,2) rectangle (11.5,3);

    \filldraw[draw=lightgray,fill=lightgray] (2,3) rectangle (7,4);
    \filldraw[draw=lightgray,fill=lightgray] (8,3) rectangle (11.5,4);

    \filldraw[draw=lightgray,fill=lightgray] (11,4) rectangle (11.5,5);

    \filldraw[draw=lightgray,fill=lightgray] (7,5) rectangle (11.5,6);
    
    \node at (6,-1) {$(b$-$iii)$};
\end{tikzpicture}}$$
        \caption{Diagrams related to key diagram in cases $(b$-$ii)$ (left) and $(b$-$iii)$ (right)}
        \label{fig:key1bii}
    \end{figure}
    \noindent
    By assumption, 
    \begin{itemize}
        \item $(j+1,a_{j+3}),(j+2,a_{j+2})\in D$ with $a_{j+3}<a_{j+2}$,
        \item $(j,a_{j+2}),(j+2,\tilde{c})\notin D$ for $\tilde{c}>a_{j+2}$, and
        \item $(j,a_{j+3}),(j+1,\tilde{c})\notin D$ for $\tilde{c}>a_{j+3}$.
    \end{itemize}
    Thus, applying Proposition~\ref{prop:strucasc} with $r=j$, $c_1=a_{j+3}$, and $c_2=a_{j+2}$, the result follows.
\end{proof}

Now, to prove Proposition~\ref{prop:keynecrank}, we make use of the following result from \textbf{\cite{KPoset1}}. For notation, given a weak composition $\mathbf{a}=(a_1,\hdots,a_n)$, we denote by $\mathbf{a}s_{i,j}$ the weak composition obtained from $\mathbf{a}$ by exchanging the entries $a_i$ and $a_j$; that is, $\mathbf{a}s_{i,j} =(a_1,\hdots,a_{i-1}, a_j,a_{i+1}, \hdots, a_{j-1}, a_i, a_{j+1}, \hdots, a_n).$

\begin{lemma}[Lemma 6.15, \textbf{\cite{KPoset1}}]\label{lem:permute}
    Let $\mathbf{a}=(a_1,\hdots,a_n)$ be a weak composition. If there exist $i<j$ such that $a_i<a_j$, then $\mathbb{D}(\mathbf{a}s_{i,j})\in \mathcal{P}(\mathbb{D}(\mathbf{a}))$.
\end{lemma}

\begin{proof}[Proof of Proposition~\ref{prop:keynecrank}]
We include only the proofs for cases $(a$-$i)$ and $(b$-$i)$ as the remaining cases following via very similar arguments. In both cases, we show that there exists a key diagram $T\in \mathcal{P}(\mathbb{D}(\mathbf{a}))$ such that $T$ has one of the patterns described in Lemma~\ref{lem:key1}. Consequently, applying Lemma~\ref{lem:key1}, it will follow that $\mathcal{P}(\mathbb{D}(\mathbf{a}))$ contains an interval that is not shellable -- namely, the interval between $T$ and the unique minimal element of $\mathcal{P}(\mathbb{D}(\mathbf{a})).$ Considering Theorem~\ref{thm:shellinterval}, we may thus conclude that $\mathcal{P}(\mathbb{D}(\mathbf{a}))$ is not shellable.
\bigskip

\noindent
$(a$-$i)$ Let $i_1$ be maximal, $i_2$ be arbitrary given our choice of $i_1$, and $i_3$ minimal given our choice of $i_2$. If $i_1<i_2-1$, then $a_{i_2-1}\ge a_{i_2}>a_{i_1}$ by maximality of $i_1$. Thus, applying Lemma~\ref{lem:permute}, $\mathbb{D}(\mathbf{a}s_{i_1,i_2-1})\in \mathcal{P}(\mathbb{D}(\mathbf{a}))$. Now, if $i_3>i_2+1$, then $a_{i_2+1}\le a_{i_2}<a_{i_3}$ by minimality of $i_3$. Consequently, applying Lemma~\ref{lem:permute}, we find that $$\mathbb{D}(\mathbf{a}s_{i_1,i_2-1}s_{i_2+1,i_3})\in \mathcal{P}(\mathbb{D}(\mathbf{a}s_{i_1,i_2-1}))\subseteq \mathcal{P}(\mathbb{D}(\mathbf{a})),$$ i.e., $\mathbb{D}(\mathbf{a}s_{i_1,i_2-1}s_{i_2+1,i_3})\in \mathcal{P}(\mathbb{D}(\mathbf{a}))$. As the values $a_{i_1},a_{i_2},a_{i_3}$ occur as terms $i_2-1, i_2,$ and $i_2+1,$ respectively, in $\mathbf{a}s_{i_1,i_2-1}s_{i_2+1,i_3}$, the result follows.
\bigskip

\noindent
$(b$-$i)$ Let $j_1$ be maximal, $j_2$ be arbitrary given our choice of $j_1$, $j_3$ be minimal given our choice of $j_2$, and $j_4$ be arbitrary given our choice of $j_3$. Note that if $j_1<j_2-1$, then $a_{j_2-1}>a_{j_2}\ge a_{j_1}$ by maximality of $j_1$; and if $j_2<j_3-1$, then for $j_2<i<j_3$ we have $a_i\neq a_{j_3}$ by minimality of $j_3$. Let $j_3^*$ be minimal such that $j_2<j_3^*\le j_3$ and $a_{j_3^*}\le a_{j_3}$. By our choice of $j_3^*,$ for $j_2<j<j_3^*$ we have $a_j>a_{j_3}>a_{j_2}$; in particular, if $j_3^*-1\neq j_2,$ then $a_{j_3^*-1}>a_{j_2}$. Thus, if we set $\mathbf{b}=\mathbf{a}s_{j_2,j_3^*-1}s_{j_3^*,j_3}$ and apply Lemma~\ref{lem:permute} twice, it follows that $\mathbb{D}(\mathbf{b})\in \mathcal{P}(\mathbb{D}(\mathbf{a})).$ Now, by construction, we have that $b_i=a_i>a_{j_1}$ for $j_1<i<j_2$, $b_{j_2}=a_{j_3^*-1}\ge a_{j_2}\ge a_{j_1}$, and $b_{i}>a_{j_3}> a_{j_1}$ for $j_2<i<j_3^*-1$. Consequently, if we set $\mathbf{c}=\mathbf{b}s_{j_1,j_3^*-2}$, then an application of Lemma~\ref{lem:permute} shows that $\mathbb{D}(\mathbf{c})\in \mathcal{P}(\mathbb{D}(\mathbf{a}))$. Finally, if $c_{j^*_3+1}\ge a_{j_4}$, then set $\mathbf{d}=\mathbf{c}$; otherwise, set $\mathbf{d}=\mathbf{c}s_{j^*_3+1,j_4}$ and note that since $j_3^*+1< j_4$ and $c_{j_4}=a_{j_4}$, Lemma~\ref{lem:permute} implies that $\mathbb{D}(\mathbf{d})\in \mathcal{P}(\mathbb{D}(\mathbf{a}))$. By construction, $$a_{j_1}=d_{j_3^*-2}\le a_{j_2}=d_{j_3^*-1}<a_{j_3}-1=d_{j_3^*}-1\le a_{j_4}-1\le d_{j_3^*+1}-1;$$ that is, $d_{j_3^*-2},d_{j_3^*-1},d_{j_3^*},$ and $d_{j_3^*+1}$ occur as consecutive terms in $\mathbf{d}$ and form the pattern described in ($b$-$i$). The result follows.
\end{proof}

\subsection{Sufficiency}

In this section, we finish the proof of Theorem~\ref{thm:kshell} by establishing that the key diagrams of pure compositions generate (EL-)shellable Kohnert posets. To do so, given a pure composition $\mathbf{a}$, we show that the Kohnert poset $\mathcal{P}(\mathbb{D}(\mathbf{a}))$ decomposes into a direct product of graded, (EL-)shellable posets. In particular, we start by showing that, for a pure composition $\mathbf{a}$ with pure decomposition $(\alpha_1,\dots,\alpha_m),$ the poset $\mathcal{P}(\mathbb{D}(\mathbf{a}))$ is isomorphic to the direct product $\mathcal{P}(\mathbb{D}(\alpha_1))\times\dots\times\mathcal{P}(\mathbb{D}(\alpha_m))$ (see Proposition~\ref{prop:kdp}). Then we show that the Kohnert posets $\mathcal{P}(\mathbb{D}(\alpha_i))$ for $1\le i\le m$ are isomorphic to intervals within Kohnert posets of hook diagrams (see Proposition~\ref{prop:purecomppieces}) and, consequently, are graded and (EL-)shellable. Since the direct product of graded and (EL-)shellable is graded and (EL-)shellable by Theorem~\ref{thm:pshell} below, it will follow that $\mathcal{P}(\mathbb{D}(\mathbf{a}))$ is graded and (EL-)shellable.

\begin{prop}\label{prop:kdp}
    Let $\mathbf{a}$ be a pure composition and assume that $\mathbf{a}=(\alpha_1,\hdots,\alpha_m)$ is a pure decomposition of $\mathbf{a}$. Then $\mathcal{P}(\mathbb{D}(\mathbf{a}))\cong\mathcal{P}(\mathbb{D}(\alpha_1))\times\hdots\times \mathcal{P}(\mathbb{D}(\alpha_m))$.
\end{prop}
\begin{proof}
    Assume that $\alpha_j=(a_1^j,\hdots,a^j_{m_j})$ for $1\le j\le m$. Let $m_0=0$ and define $\psi$ on diagrams $D\in \mathcal{P}(\mathbb{D}(\mathbf{a}))$ by $\psi(D)=(D_1,\hdots,D_m)$, where $$D_j=\left\{(r,c)~\left|~1+\sum_{i=0}^{j-1}m_i\le r^*=r+\sum_{i=0}^{j-1}m_i\le \sum_{i=0}^{j}m_i~\text{and}~(r^*,c)\in D\right.\right\}$$ for $1\le j\le m$; that is, $D_j$ is the diagram formed by shifting the cells occupying rows $1+\sum_{i=0}^{j-1}m_i\le r\le\sum_{i=0}^{j}m_i$ of $\mathbb{D}(\mathbf{a})$ so that they occupy rows $1\le r\le m_j$. As an example, in Figure~\ref{fig:psi} we illustrate $\psi(\mathbb{D}(\mathbf{a}))$ where $$\mathbf{a}=(6,5,4,5,4,3,5,3,1,3,1,0,1)=(\alpha_1,\alpha_2,\alpha_3,\alpha_4)$$ with $\alpha_1=(6,5)$, $\alpha_2=(4,5,4,3,5)$, $\alpha_3=(3,1,3)$, and $\alpha_4=(1,0,1)$.
    \begin{figure}[H]
        \centering
        $$\scalebox{0.9}{\begin{tikzpicture}
            \node at (0,0) {\scalebox{0.4}{\begin{tikzpicture}
            \draw (0,13.5)--(0,0)--(6.5,0);
            \node at (0.5,0.5) {$\bigtimes$};
            \node at (1.5,0.5) {$\bigtimes$};
            \node at (2.5,0.5) {$\bigtimes$};
            \node at (3.5,0.5) {$\bigtimes$};
            \node at (4.5,0.5) {$\bigtimes$};
            \node at (5.5,0.5) {$\bigtimes$};
            \draw (0,1)--(6,1)--(6,0);
            \draw (1,0)--(1,2);
            \draw (2,0)--(2,2);
            \draw (3,0)--(3,2);
            \draw (4,0)--(4,2);
            \draw (5,0)--(5,1);
            \draw (0,2)--(5,2)--(5,1);
            \node at (0.5,1.5) {$\bigtimes$};
            \node at (1.5,1.5) {$\bigtimes$};
            \node at (2.5,1.5) {$\bigtimes$};
            \node at (3.5,1.5) {$\bigtimes$};
            \node at (4.5,1.5) {$\bigtimes$};
            \draw (0,3)--(4,3)--(4,2);
            \node at (0.5,2.5) {$\bigtimes$};
            \node at (1.5,2.5) {$\bigtimes$};
            \node at (2.5,2.5) {$\bigtimes$};
            \node at (3.5,2.5) {$\bigtimes$};
            \draw (1,2)--(1,3);
            \draw (2,2)--(2,3);
            \draw (3,2)--(3,3);
            \draw (0,4)--(5,4)--(5,3)--(4,3);
            \node at (0.5,3.5) {$\bigtimes$};
            \node at (1.5,3.5) {$\bigtimes$};
            \node at (2.5,3.5) {$\bigtimes$};
            \node at (3.5,3.5) {$\bigtimes$};
            \node at (4.5,3.5) {$\bigtimes$};
            \draw (1,3)--(1,4);
            \draw (2,3)--(2,4);
            \draw (3,3)--(3,4);
            \draw (4,3)--(4,4);
            \draw (0,5)--(4,5)--(4,4);
            \node at (0.5,4.5) {$\bigtimes$};
            \node at (1.5,4.5) {$\bigtimes$};
            \node at (2.5,4.5) {$\bigtimes$};
            \node at (3.5,4.5) {$\bigtimes$};
            \draw (1,4)--(1,5);
            \draw (2,4)--(2,5);
            \draw (3,4)--(3,5);
            \draw (0,6)--(3,6)--(3,5);
            \node at (0.5,5.5) {$\bigtimes$};
            \node at (1.5,5.5) {$\bigtimes$};
            \node at (2.5,5.5) {$\bigtimes$};
            \draw (1,5)--(1,6);
            \draw (2,5)--(2,6);
            \draw (3,5)--(3,6);
            \draw (0,7)--(5,7)--(5,6)--(3,6);
            \node at (0.5,6.5) {$\bigtimes$};
            \node at (1.5,6.5) {$\bigtimes$};
            \node at (2.5,6.5) {$\bigtimes$};
            \node at (3.5,6.5) {$\bigtimes$};
            \node at (4.5,6.5) {$\bigtimes$};
            \draw (1,6)--(1,7);
            \draw (2,6)--(2,7);
            \draw (3,6)--(3,7);
            \draw (4,6)--(4,7);
            \draw (0,8)--(3,8)--(3,7);
            \node at (0.5,7.5) {$\bigtimes$};
            \node at (1.5,7.5) {$\bigtimes$};
            \node at (2.5,7.5) {$\bigtimes$};
            \draw (1,7)--(1,8);
            \draw (2,7)--(2,8);
            \draw (3,7)--(3,8);
            \draw (0,9)--(1,9)--(1,8);
            \node at (0.5,8.5) {$\bigtimes$};
            \draw (0,10)--(3,10)--(3,9)--(1,9);
            \node at (0.5,9.5) {$\bigtimes$};
            \node at (1.5,9.5) {$\bigtimes$};
            \node at (2.5,9.5) {$\bigtimes$};
            \draw (1,9)--(1,10);
            \draw (2,9)--(2,10);
            \draw (3,9)--(3,10);
            \draw (0,11)--(1,11)--(1,10);
            \node at (0.5,10.5) {$\bigtimes$};
            \draw (0,12)--(1,12)--(1,13)--(0,13);
            \node at (0.5,12.5) {$\bigtimes$};
        \end{tikzpicture}}};
        \draw[->] (2,0)--(3,0);
        \node at (2.5, 0.5) {$\psi$};
        \node at (6,-0.63) {\scalebox{0.4}{\begin{tikzpicture}
            \draw (0,2.5)--(0,0)--(6.5,0);
            \node at (0.5,0.5) {$\bigtimes$};
            \node at (1.5,0.5) {$\bigtimes$};
            \node at (2.5,0.5) {$\bigtimes$};
            \node at (3.5,0.5) {$\bigtimes$};
            \node at (4.5,0.5) {$\bigtimes$};
            \node at (5.5,0.5) {$\bigtimes$};
            \draw (0,1)--(6,1)--(6,0);
            \draw (1,0)--(1,2);
            \draw (2,0)--(2,2);
            \draw (3,0)--(3,2);
            \draw (4,0)--(4,2);
            \draw (5,0)--(5,1);
            \draw (0,2)--(5,2)--(5,1);
            \node at (0.5,1.5) {$\bigtimes$};
            \node at (1.5,1.5) {$\bigtimes$};
            \node at (2.5,1.5) {$\bigtimes$};
            \node at (3.5,1.5) {$\bigtimes$};
            \node at (4.5,1.5) {$\bigtimes$};
        \end{tikzpicture}}};
        \node at (7.75,-1.2) {,};
        \node at (9.5,-0.05) {\scalebox{0.4}{\begin{tikzpicture}
            \draw (0,5.5)--(0,0)--(5.5,0);
            \draw (0,1)--(4,1)--(4,0);
            \node at (0.5,0.5) {$\bigtimes$};
            \node at (1.5,0.5) {$\bigtimes$};
            \node at (2.5,0.5) {$\bigtimes$};
            \node at (3.5,0.5) {$\bigtimes$};
            \draw (1,0)--(1,1);
            \draw (2,0)--(2,1);
            \draw (3,0)--(3,1);
            \draw (0,2)--(5,2)--(5,1)--(4,1);
            \node at (0.5,1.5) {$\bigtimes$};
            \node at (1.5,1.5) {$\bigtimes$};
            \node at (2.5,1.5) {$\bigtimes$};
            \node at (3.5,1.5) {$\bigtimes$};
            \node at (4.5,1.5) {$\bigtimes$};
            \draw (1,1)--(1,2);
            \draw (2,1)--(2,2);
            \draw (3,1)--(3,2);
            \draw (4,1)--(4,2);
            \draw (0,3)--(4,3)--(4,2);
            \node at (0.5,2.5) {$\bigtimes$};
            \node at (1.5,2.5) {$\bigtimes$};
            \node at (2.5,2.5) {$\bigtimes$};
            \node at (3.5,2.5) {$\bigtimes$};
            \draw (1,2)--(1,3);
            \draw (2,2)--(2,3);
            \draw (3,2)--(3,3);
            \draw (0,4)--(3,4)--(3,3);
            \node at (0.5,3.5) {$\bigtimes$};
            \node at (1.5,3.5) {$\bigtimes$};
            \node at (2.5,3.5) {$\bigtimes$};
            \draw (1,3)--(1,4);
            \draw (2,3)--(2,4);
            \draw (3,3)--(3,4);
            \draw (0,5)--(5,5)--(5,4)--(3,4);
            \node at (0.5,4.5) {$\bigtimes$};
            \node at (1.5,4.5) {$\bigtimes$};
            \node at (2.5,4.5) {$\bigtimes$};
            \node at (3.5,4.5) {$\bigtimes$};
            \node at (4.5,4.5) {$\bigtimes$};
            \draw (1,4)--(1,5);
            \draw (2,4)--(2,5);
            \draw (3,4)--(3,5);
            \draw (4,4)--(4,5);
        \end{tikzpicture}}};
        \node at (11.2,-1.2) {,};
        \node at (12.5,-0.45) {\scalebox{0.4}{\begin{tikzpicture}
            \draw (0,3.5)--(0,0)--(3.5,0);
            \draw (0,1)--(3,1)--(3,0);
            \node at (0.5,0.5) {$\bigtimes$};
            \node at (1.5,0.5) {$\bigtimes$};
            \node at (2.5,0.5) {$\bigtimes$};
            \draw (1,0)--(1,1);
            \draw (2,0)--(2,1);
            \draw (3,0)--(3,1);
            \draw (0,2)--(1,2)--(1,1);
            \node at (0.5,1.5) {$\bigtimes$};
            \draw (0,3)--(3,3)--(3,2)--(1,2);
            \node at (0.5,2.5) {$\bigtimes$};
            \node at (1.5,2.5) {$\bigtimes$};
            \node at (2.5,2.5) {$\bigtimes$};
            \draw (1,2)--(1,3);
            \draw (2,2)--(2,3);
            \draw (3,2)--(3,3);
        \end{tikzpicture}}};
        \node at (13.5,-1.2) {,};
        \node at (14.5,-0.45) {\scalebox{0.4}{\begin{tikzpicture}
            \draw (0,3.5)--(0,0)--(1.5,0);
            \draw (0,1)--(1,1)--(1,0);
            \node at (0.5,0.5) {$\bigtimes$};
            \draw (0,2)--(1,2)--(1,3)--(0,3);
            \node at (0.5,2.5) {$\bigtimes$};
        \end{tikzpicture}}};
        \draw (4.2, -1.5) to[bend left](4,1.5);
        \draw (15.2, -1.5) to[bend right](15.2,1.5);
        \end{tikzpicture}}$$
        \caption{$\psi(\mathbb{D}(\mathbf{a}))$}
        \label{fig:psi}
    \end{figure}
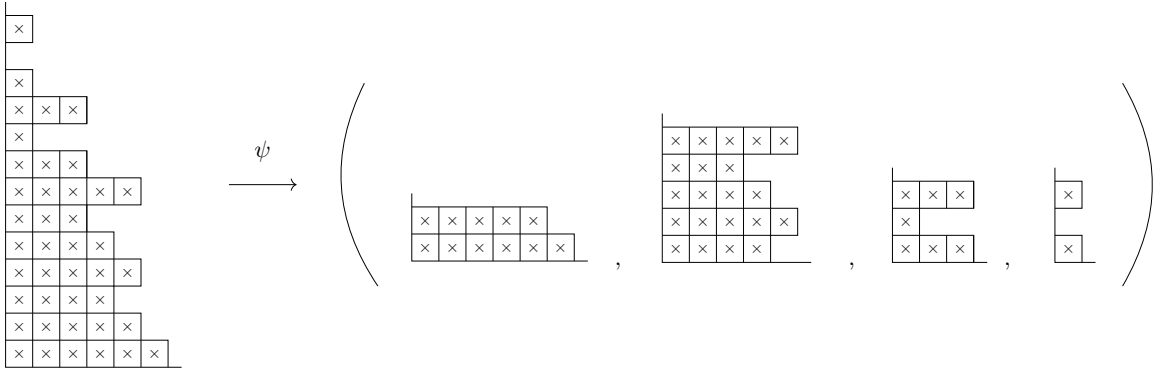
    \noindent
    We claim that $\psi$ defines an isomorphism between the posets $\mathcal{P}(\mathbb{D}(\mathbf{a}))$ and $\mathcal{P}(\mathbb{D}(\alpha_1))\times\hdots\times \mathcal{P}(\mathbb{D}(\alpha_m))$.

    First, we show that $$\mathcal{P}(\mathbb{D}(\alpha_1))\times\hdots\times \mathcal{P}(\mathbb{D}(\alpha_m))\subseteq \text{im}~\psi.$$ By definition, $\psi(\mathbb{D}(\mathbf{a}))=(\mathbb{D}(\alpha_1),\hdots,\mathbb{D}(\alpha_m))\in\text{im}~\psi$. Consequently, it suffices to show that if $$(D_1,\hdots, D_j,\hdots,D_m)\in \text{im}~\psi$$ and $D^*_j$ can be formed from $D_j$ by applying a single Kohnert move for $1\le j\le m$, then  $$(D_1,\hdots, D^*_j,\hdots,D_m)\in \text{im}~\psi.$$ Let $D\in \mathcal{P}(\mathbb{D}(\mathbf{a}))$ satisfy $\psi(D)=(D_1,\hdots, D_j,\hdots,D_m)$. Assume that $D_j^*$ can be formed from $D_j$ by applying a single Kohnert move at row $r$ with $D_j^*=D_j\ldownarrow^{(r,c)}_{(r',c)}$. Then $(r,c)$ is rightmost in row $r$ of $D_j$, $(\tilde{r},c)\in D_j$ for all $r'<\tilde{r}< r$, and $(r',c)\notin D_j$. Thus, by the definition of $\psi$, it follows that 
    \begin{enumerate}
        \item $(r'+\sum_{i=0}^{j-1}m_i,c)\notin D$,
        \item $(\tilde{r}+\sum_{i=0}^{j-1}m_i,c)\in D_j$ for all $r'+\sum_{i=0}^{j-1}m_i<\tilde{r}+\sum_{i=0}^{j-1}m_i< r+\sum_{i=0}^{j-1}m_i$, and
        \item  $(r+\sum_{i=0}^{j-1}m_i,c)\in D$ is rightmost in row $r+\sum_{i=0}^{j-1}m_i$ of $D$.
    \end{enumerate}
    Consequently, applying a Kohnert move at row $r+\sum_{i=0}^{j-1}m_i$ of $D$ results in the diagram $$D^*=D\ldownarrow^{(r+\sum_{i=0}^{j-1}m_i,c)}_{(r'+\sum_{i=0}^{j-1}m_i,c)}.$$ Evidently, $\psi(D^*)=(D_1,\hdots, D^*_j,\hdots,D_m)$. Therefore, $\mathcal{P}(\mathbb{D}(\alpha_1))\times\hdots\times \mathcal{P}(\mathbb{D}(\alpha_m))\subseteq \text{im}~\psi.$
    
    Now, to show that the above inclusion is in fact an equality, considering the definition of $\psi$, it suffices to show that $(\ast)$ for $1\le j\le m$, cells in rows $1+\sum_{i=0}^{j-1}m_i\le r\le \sum_{i=0}^{j}m_i$ in $\mathbb{D}(\mathbf{a})$ cannot be moved via Kohnert moves to rows $\tilde{r}<\sum_{i=0}^{j-1}m_i$ in $\mathbb{D}(\mathbf{a})$; that is, no sequence of Kohnert moves can move cells in rows corresponding to $\alpha_j$ in $\mathbb{D}(\mathbf{a})$, for $1\le j\le m$, into rows corresponding to $\alpha_i$ in $\mathbb{D}(\mathbf{a})$ for $1\le i<j$. Since $\min(\alpha_{j-1})\ge\max(\alpha_j)$ for $1<j\le m$, it follows that $\min(\alpha_{i})\ge \max(\alpha_{j})$ for $1\le i<j\le m$. Thus, $\{(r,c)~|~1\le c\le \max(\alpha_{j})\text{ and }1\le r\le \sum_{i=0}^{j-1}m_i\}\subset\mathbb{D}(\mathbf{a})$. Consequently, it follows that for all $D\in \mathcal{P}(\mathbb{D}(\mathbf{a}))$, there are no empty positions in columns 1 through $\max(\alpha_{j})$ below row $1+\sum_{i=0}^{j-1}m_i$. Therefore, since the cells in rows $1+\sum_{i=0}^{j-1}m_i\le r\le \sum_{i=0}^{j}m_i$ of $\mathbb{D}(\mathbf{a})$ only occupy columns $1\le c\le \max(\alpha_{j})$, property $(\ast)$ holds. Hence, $$\mathcal{P}(\mathbb{D}(\alpha_1))\times\hdots\times \mathcal{P}(\mathbb{D}(\alpha_m))=\text{im}~\psi.$$ As $\psi$ is clearly one-to-one, it follows that $\psi:\mathcal{P}(\mathbb{D}(\mathbf{a}))\to\mathcal{P}(\mathbb{D}(\alpha_1))\times\hdots\times \mathcal{P}(\mathbb{D}(\alpha_m))$ is a bijection.

    It remains to show that $\psi$ and $\psi^{-1}$ are order preserving. Since the proofs are similar, we consider only the proof for $\psi$. For $\psi$, take $D_1,D_2\in \mathcal{P}(\mathbb{D}(\mathbf{a}))$ such that $D_2\prec D_1$. Then there exists a sequence of rows $R=\{r_i\}_{i=1}^n$ such that if $T_0=D_1$ and, for $1\leq i\leq n,$ $T_i$ is formed from $T_{i-1}$ by applying a single Kohnert move at row $r_i$, then $T_{i-1}\neq T_i$ for $1\le i\le n$ and $T_n=D_1$. Let $R_j$ denote the (possibly empty) subsequence of $R$ consisting of all $r_i\in R$ such that $1+\sum_{i=0}^{j-1}m_i\le r_i\le\sum_{i=1}^{j}m_i$. For $1\le j\le m$, if $R_j$ is nonempty, then assume $R_j=\{r_i^j\}_{i=1}^{n_j}$. Let $\psi(D_i)=(D^i_1,\hdots,D^i_m)$ for $i=1,2$. By construction, if $T_{0,j}=D_1^j$ for $1\le j\le m$ and $T_{i,j}$ is the diagram formed from $T_{i-1,j}$ by applying a Kohnert move at row $r_i^j-\sum_{k=0}^{j-1}m_k$ for $1\le i\le n_j$, then $T_{n_j,j}=D^2_j$. Thus, $D^2_j\preceq D^1_j$ for all $1\le j\le m$, and $D^2_k\prec D^1_k$ for at least one $1\le k\le m$; that is, $$\psi(D_2)=(D^2_1,\hdots,D^2_m)\prec (D^1_1,\hdots,D^1_m)=\psi(D_1)$$ in $\mathcal{P}(\mathbb{D}(\alpha_1))\times\hdots\times \mathcal{P}(\mathbb{D}(\alpha_m))$ and $\psi$ is order preserving.
\end{proof}

\begin{prop}\label{prop:purecomppieces}
    If $\mathbf{a}=(a_1,\hdots,a_n)$ is a basic pure composition, then $\mathcal{P}(\mathbb{D}(\mathbf{a}))$ is EL-shellable.
\end{prop}
\begin{proof}
    Recall that if $\mathbf{a}=(a_1,\hdots,a_n)$ is a basic pure composition, then $\mathbf{a}$ is of one of the following forms:
    \begin{enumerate}
        \item[\textup{(i)}] $a_1\ge \hdots\ge a_n$,
        \item[\textup{(ii)}] there exists $p\in\mathbb{N}$ such that $a_1=p$ and $\{a_1,\hdots,a_n\}=\{p,p+1\}$,
        \item[\textup{(iii)}] $a_1\ge\hdots\ge a_{n-1}<a_n-1$, or
        \item[\textup{(iv)}] there exists $2<i^*\in\mathbb{N}$ and $p\in\mathbb{N}$ such that $a_1=p$, $\{a_1,\hdots,a_{i^*-1}\}=\{p,p+1\}$, $p>a_{i^*}\ge \hdots\ge a_{n-1}$, and $a_{n}=p+1$.
    \end{enumerate}
    Note that if $\mathbf{a}$ is of type (i) then, there are no empty positions below cells. Consequently, $\mathcal{P}(\mathbb{D}(\mathbf{a}))$ is a poset consisting of a single element which is trivially EL-shellable. To establish the result for the remaining cases, we show that the associated Kohnert poset $\mathcal{P}(\mathbb{D}(\mathbf{a}))$ is isomorphic to an interval in the Kohnert poset of a hook diagram. 
    
    Let $R$ denote the collection of cells $(r,c)\in\mathbb{D}(\mathbf{a})$ for $1\le r<n$ such that $(\tilde{r},c)\in\mathbb{D}(\mathbf{a})$ for $1\le \tilde{r}\le r$. When $\mathbf{a}$ is of type (ii) or (iv) with $\max(\mathbf{a})=p+1$ we have 
    \begin{equation}\label{eq:1}
        R=\{(r,c)~|~a_r=p+1,~1\le r<n,~1\le c\le p\}\cup \{(r,c)~|~a_r\le p,1\le c\le a_r\},
    \end{equation}
    while when $\mathbf{a}$ is of type (iii) we have that 
    \begin{equation}\label{eq:2}
        R=\mathbb{D}(a_1,\hdots,a_{n-1}).
    \end{equation}
    In Figure~\ref{fig:Rboxes} below we illustrate $R$ for the key diagrams of the basic pure compositions $(2,3,3,2,3)$, $(4,3,2,2,4)$, and $(3,4,3,4,3,2,1,4)$ of types $(ii)$, $(iii)$, and $(iv)$, respectively, where the cells of $R$ are marked with an ``$R$".
    \begin{figure}[H]
        \centering
        $$\scalebox{0.5}{\begin{tikzpicture}
            \draw (0,5.5)--(0,0)--(3.5,0);
            \node at (0.5,0.5) {$R$};
            \node at (1.5,0.5) {$R$};
            \draw (0,1)--(2,1)--(2,0);
            \draw (1,0)--(1,5);
            \node at (0.5,1.5) {$R$};
            \node at (1.5,1.5) {$R$};
            \node at (2.5,1.5) {$\bigtimes$};
            \draw (0,2)--(3,2)--(3,1)--(2,1);
            \draw (2,1)--(2,5);
            \node at (0.5,2.5) {$R$};
            \node at (1.5,2.5) {$R$};
            \node at (2.5,2.5) {$\bigtimes$};
            \draw (0,3)--(3,3)--(3,2);
            \node at (0.5,3.5) {$R$};
            \node at (1.5,3.5) {$R$};
            \draw (0,4)--(2,4);
            \node at (0.5,4.5) {$\bigtimes$};
            \node at (1.5,4.5) {$\bigtimes$};
            \node at (2.5,4.5) {$\bigtimes$};
            \draw (0,5)--(3,5)--(3,4)--(2,4);
            \node at (2,-1) {\Large $\mathbb{D}(2,3,3,2,3)$};
        \end{tikzpicture}}\quad\quad\quad\quad \scalebox{0.5}{\begin{tikzpicture}
            \draw (0,5.5)--(0,0)--(4.5,0);
            \node at (0.5,0.5) {$R$};
            \node at (1.5,0.5) {$R$};
            \node at (2.5,0.5) {$R$};
            \node at (3.5,0.5) {$R$};
            \draw (0,1)--(4,1)--(4,0);
            \draw (1,0)--(1,5);
            \draw (2,0)--(2,5);
            \draw (3,0)--(3,1);
            \node at (0.5,1.5) {$R$};
            \node at (1.5,1.5) {$R$};
            \node at (2.5,1.5) {$R$};
            \draw (0,2)--(3,2)--(3,1)--(2,1);
            \node at (0.5,2.5) {$R$};
            \node at (1.5,2.5) {$R$};
            \draw (0,3)--(2,3);
            \node at (0.5,3.5) {$R$};
            \node at (1.5,3.5) {$R$};
            \draw (0,4)--(2,4);
            \node at (0.5,4.5) {$\bigtimes$};
            \node at (1.5,4.5) {$\bigtimes$};
            \node at (2.5,4.5) {$\bigtimes$};
            \node at (3.5,4.5) {$\bigtimes$};
            \draw (0,5)--(4,5)--(4,4)--(2,4);
            \draw (3,4)--(3,5);
            \node at (2.5,-1) {\Large $\mathbb{D}(4,3,2,2,4)$};
        \end{tikzpicture}}\quad\quad\quad\quad \scalebox{0.5}{\begin{tikzpicture}
            \draw (0,7.5)--(0,0)--(4.5,0);
            \node at (0.5,0.5) {$R$};
            \node at (1.5,0.5) {$R$};
            \node at (2.5,0.5) {$R$};
            \draw (0,1)--(3,1)--(3,0);
            \draw (1,0)--(1,7);
            \draw (2,0)--(2,5);
            \draw (3,0)--(3,5);
            \node at (0.5,1.5) {$R$};
            \node at (1.5,1.5) {$R$};
            \node at (2.5,1.5) {$R$};
            \node at (3.5,1.5) {$\bigtimes$};
            \draw (0,2)--(4,2)--(4,1)--(3,1);
            \node at (0.5,2.5) {$R$};
            \node at (1.5,2.5) {$R$};
            \node at (2.5,2.5) {$R$};
            \draw (0,3)--(3,3)--(3,2);
            \node at (0.5,3.5) {$R$};
            \node at (1.5,3.5) {$R$};
            \node at (2.5,3.5) {$R$};
            \node at (3.5,3.5) {$\bigtimes$};
            \draw (0,4)--(4,4)--(4,3)--(3,3);
            \node at (0.5,4.5) {$R$};
            \node at (1.5,4.5) {$R$};
            \node at (2.5,4.5) {$R$};
            \draw (0,5)--(3,5);
            \draw (0,6)--(1,6);
            \node at (0.5,5.5) {$R$};
            \node at (0.5,6.5) {$\bigtimes$};
            \node at (1.5,6.5) {$\bigtimes$};
            \node at (2.5,6.5) {$\bigtimes$};
            \node at (3.5,6.5) {$\bigtimes$};
            \draw (0,7)--(4,7)--(4,6)--(1,6);
            \draw (2,6)--(2,7);
            \draw (3,6)--(3,7);
            \node at (2.5,-1) {\Large $\mathbb{D}(3,4,3,4,3,1,4)$};
        \end{tikzpicture}}$$
        \caption{Cells of $R$}
        \label{fig:Rboxes}
    \end{figure}
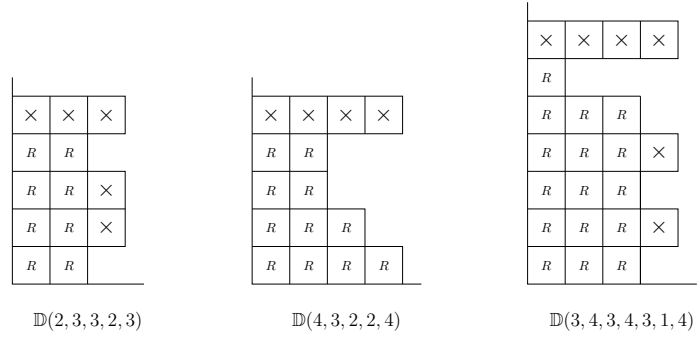
    
    \noindent
    Note that in all cases, considering (\ref{eq:1}) and (\ref{eq:2}), the cells of $R$ are both left and bottom justified. Consequently, $R\subset D$ for all $D\in \mathcal{P}(\mathbb{D}(\mathbf{a}))$. Now, for any $D\in \mathcal{P}(\mathbb{D}(\mathbf{a}))$, let $$Hk(D)=D\backslash R.$$ For $D=\mathbb{D}(\mathbf{a})$, we have the following options for the form of $Hk(D)$ depending on the type of $\mathbf{a}$:
    \begin{itemize}
        \item[] (ii) $Hk(D)$ consists of the $a_n$ cells in row $n$ of $D$ and the cells below in column $a_n$,
        \item[] (iii) $Hk(D)$ consists of the $a_n$ cells in row $n$ of $D$, or
        \item[] (iv) $Hk(D)$ consists of the $a_n$ cells in row $n$ of $D$ and the cells below in column $a_n$.
    \end{itemize}
    Note that, in each case, $Hk(\mathbb{D}(\mathbf{a}))$ is a hook diagram. We claim that $Hk$ defines a poset isomorphism between $\mathcal{P}(\mathbb{D}(\mathbf{a}))$ and an interval of $\mathcal{P}(H),$ where $H=Hk(\mathbb{D}(\mathbf{a}))$.

    First we show that $Hk$ is an order-preserving map from $\mathcal{P}(\mathbb{D}(\mathbf{a}))$ to $\mathcal{P}(H)$. Let $$\mathcal{P}_1=\{Hk(D)~|~D\in \mathcal{P}(\mathbb{D}(\mathbf{a}))\},$$ where we define a partial ordering on $\mathcal{P}_1$ by $D'\prec D$ if $D'$ can be formed from $D$ by applying some sequence of Kohnert moves. Since $Hk(\mathbb{D}(\mathbf{a}))=H\in\mathcal{P}_1$ and $D\prec\mathbb{D}(\mathbf{a})$ for all $D\in \mathcal{P}(\mathbb{D}(\mathbf{a}))$, if we can show that $Hk:\mathcal{P}(\mathbb{D}(\mathbf{a}))\to \mathcal{P}_1$ is order preserving, then it will follow that $\mathcal{P}_1\subseteq\mathcal{P}(H)$; that is, it will follow that $Hk$ is an order-preserving map from $\mathcal{P}(\mathbb{D}(\mathbf{a}))$ to $\mathcal{P}(H)$. Take $D_1,D_2\in \mathcal{P}(\mathbb{D}(\mathbf{a}))$ such that $D_2\precdot D_1$. Suppose that $D_2$ is formed from $D_1$ by applying a single Kohnert move at row $r$ and $D_2=D_1\ldownarrow^{(r,c)}_{(r',c)}$. We show that $Hk(D_2)$ can be formed from $Hk(D_1)$ by applying a single Kohnert move at row $r$ and $Hk(D_2)=Hk(D_1)\ldownarrow^{(r,c)}_{(r',c)}$, i.e., $Hk(D_2)\prec Hk(D_1)$. Since $D_2=D_1\ldownarrow^{(r,c)}_{(r',c)}$, it follows that $(r,c)\in D_1$ is rightmost in row $r$, $(\widehat{r},c)\in D_1$ for all $r'<\widehat{r}<r$, and $(r',c)\notin D_1$. 
    Thus, since the cells of $R\subset D_1$ are both bottom and left justified, we may conclude that 
    \begin{itemize}
        \item $(r,c),(r',c)\notin R$,
        \item $(r,c)\in D_1\backslash R= Hk(D_1)$ is rightmost in row $r$ of $Hk(D_1)$, and
        \item $(r',c)\notin D_1\backslash R=Hk(D_1)$.
    \end{itemize}
    Consequently, if applying a Kohnert at row $r$ of $Hk(D_1)$ does not result in $Hk(D_1)\ldownarrow^{(r,c)}_{(r',c)}$, then there must exist $\hat{r}$ such that $r'<\widehat{r}<r$, $(\widehat{r},c)\notin Hk(D_1)$, and $(\widehat{r},c)\in D_1$; however, this implies $(\widehat{r},c)\in R$, which is impossible since $(r',c)\notin R$ and the cells of $R$ are bottom justified. Hence, applying a Kohnert move at row $r$ of $Hk(D_1)$ results in $$Hk(D_1)\ldownarrow^{(r,c)}_{(r',c)}=(D_1\backslash R)\ldownarrow^{(r,c)}_{(r',c)}=\left(D_1\ldownarrow^{(r,c)}_{(r',c)}\right)\backslash R=D_2\backslash R=Hk(D_2),$$ where the second equality follows since $(r',c)\notin R$. Now, since $Hk(D_2)$ can be formed from $Hk(D_1)$ by applying a Kohnert move at row $r$, it follows that $Hk(D_2)\prec Hk(D_1)$. Therefore, if $D_1,D_2\in \mathcal{P}(\mathbb{D}(\mathbf{a}))$ satisfy $D_2\precdot D_1$, then $Hk(D_2)\prec H_k(D_1)$; that is, the map $Hk:\mathcal{P}(\mathbb{D}(\mathbf{a}))\to \mathcal{P}_1$ is order preserving. As noted above, it follows that $Hk$ is an order-preserving map from $\mathcal{P}(\mathbb{D}(\mathbf{a}))$ to $\mathcal{P}(H)$.

    Since $Hk$ is clearly one-to-one, in order to establish the claim, it only remains to show that $Hk$ maps $\mathcal{P}(\mathbb{D}(\mathbf{a}))$ onto an interval of $\mathcal{P}(H)$ and that the inverse of $Hk$ is order preserving. Let $\mathbf{a}'$ be the weak composition formed by sorting the entries of $\mathbf{a}$ into weakly decreasing order, and note that, by Corollary 6.2 of \textbf{\cite{KPoset1}}, $\mathbf{a}'$ is the unique minimal element of $\mathcal{P}(\mathbb{D}(\mathbf{a}))$. If $H'=Hk(\mathbb{D}(\mathbf{a}'))$, then we have that $Hk:\mathcal{P}(\mathbb{D}(\mathbf{a}))\to [H',H]$ since $Hk$ is order preserving. To show that $Hk$ maps $\mathcal{P}(\mathbb{D}(\mathbf{a}))$ onto $[H',H]$, we investigate the map $\phi$, which we define to send $D\in [H',H]$ to $D\cup R$. 

    Similar to the case of $Hk$, let $$\mathcal{P}_2=\{\phi(D)~|~D\in[H',H]\}$$ and define a partial ordering on $\mathcal{P}_2$ by $D'\prec D$ if $D'$ can be formed from $D$ by applying some sequence of Kohnert moves. We shall show that $\phi:[H',H]\to \mathcal{P}_2$ is order preserving. Since $\phi(H)=\mathbb{D}(\mathbf{a})$ and $\phi(H')=\mathbb{D}(\mathbf{a}')$, it will then follow that $$\mathbb{D}(\mathbf{a}')=\phi(H')\preceq \phi(\widehat{H})\preceq \phi(H)=\mathbb{D}(\mathbf{a})$$ for all $\widehat{H}\in [H',H]$, i.e., it will follow that $\mathcal{P}_2\subseteq \mathcal{P}(\mathbb{D}(\mathbf{a}))$ and $\phi$ is an order-preserving map from $[H',H]$ to $\mathcal{P}(\mathbb{D}(\mathbf{a}))$. Let $S=[\max(\mathbf{a})]^2\backslash R$. Since the cells of $R$ are both bottom and left justified, it follows that there exists $0<r^i_1\le r^i_2$ for $1\le i\le \max(\mathbf{a})$ such that $$S=\bigcup_{i=1}^{\max(\mathbf{a})}\{(\tilde{r},i)~|~r^i_1\le \tilde{r}\le r^i_2\}.$$ Moreover, since $H=Hk(\mathbb{D}(\mathbf{a}))=\mathbb{D}(\mathbf{a})\backslash R$ and $H'=Hk(\mathbb{D}(\mathbf{a}'))=\mathbb{D}(\mathbf{a}')\backslash R$, we have that $H\backslash S=\emptyset=H'\backslash S$. Consequently, applying Lemma~\ref{lem:strucprophelp}, it follows that $\widehat{H}\backslash S=\emptyset$ for all $\widehat{H}\in [H',H]$. In particular, 
    $$\widehat{H}\cap R=\emptyset\text{ for all }\widehat{H}\in [H',H].$$ Now, take $H_1,H_2\in [H',H]$ such that $H_2\precdot H_1$. Suppose that $H_2$ can be formed from $H_1$ by applying a single Kohnert move at row $r$ and $H_2=H_1\ldownarrow^{(r,c)}_{(r',c)}$. We show that $\phi(H_2)$ can be formed from $\phi(H_1)$ by applying a single Kohnert move at row $r$ and $\phi(H_2)=\phi(H_1)\ldownarrow^{(r,c)}_{(r',c)}$. Since $H_2=H_1\ldownarrow^{(r,c)}_{(r',c)}$, it follows that $(r,c)\in H_1$ is rightmost in row $r$, $(\widehat{r},c)\in H_1$ for all $r'<\widehat{r}<r$, and $(r',c)\notin H_1$. Thus, since the cells of $R$ are left justified and $H_1\cap R=\emptyset$, it follows that $(r,c)\notin R$ and $(r,c)\in H_1\cup R= \phi(H_1)$ is rightmost in row $r$ of $\phi(H_1)$. Moreover, since $H_2\cap R=\emptyset$, it follows that $(r',c)\notin R$ and, as a result, $(r',c)\notin H_1\cup R=\phi(H_1)$. Consequently, if applying a single Kohnert move at row $r$ of $\phi(H_1)$ does not result in $\phi(H_1)\ldownarrow^{(r,c)}_{(r',c)}$, then there must exist $r'<\widehat{r}<r$ such that $(\widehat{r},c)\notin \phi(H_1)$; but this is impossible since $H_1\subset H_1\cup R=\phi(H_1)$. Hence, applying a single Kohnert move at row $r$ of $\phi(H_1)$ results in $$\phi(H_1)\ldownarrow^{(r,c)}_{(r',c)}=(H_1\cup R)\ldownarrow^{(r,c)}_{(r',c)}=\left(H_1\ldownarrow^{(r,c)}_{(r',c)}\right)\cup R=H_2\cup R=\phi(H_2).$$ Since $\phi(H_2)$ can be formed from $\phi(H_1)$ by applying a Kohnert move at row $r$, it follows that $\phi(H_2)\prec \phi(H_1)$. Therefore, if $H_1,H_2\in \mathcal{P}(H)$ satisfy $H_2\precdot H_1$, then $\phi(H_2)\prec \phi(H_1)$; that is, $\phi:[H',H]\to\mathcal{P}_2$ is order preserving. As noted above, it follows that $\phi$ is an order-preserving map from $[H',H]$ to $\mathcal{P}(\mathbb{D}(\mathbf{a}))$.

    Now, our work above shows that $Hk:\mathcal{P}(\mathbb{D}(\mathbf{a}))\to [H',H]$ and $\phi:[H',H]\to \mathcal{P}(\mathbb{D}(\mathbf{a}))$ are both order preserving. Since $R\subset D$ for all $D\in \mathcal{P}(\mathbb{D}(\mathbf{a}))$ and $\widehat{H}\cap R=\emptyset$ for all $\widehat{H}\in [H',H]$, it follows that $\phi(Hk(D))=D$ for all $D\in \mathcal{P}(\mathbb{D}(\mathbf{a}))$ and $Hk(\phi(\widehat{H}))=\widehat{H}$ for all $\widehat{H}\in [H',H]$. Thus, $\phi=Hk^{-1}$ so that $Hk$ defines an isomorphism between $\mathcal{P}(\mathbb{D}(\mathbf{a}))$ and $[H',H]$. Combining Theorem~\ref{thm:ELinterval} and Theorem~\ref{thm:HrrCshell}, the result follows.
\end{proof}

Now, to prove Theorem~\ref{thm:kshell} using the results above, we require the following result of \textbf{\cite{Bjorner}}.

%[Theorem 4.3, \textbf{\cite{Bjorner}}]

\begin{theorem}[Bj\"orner \textbf{\cite{Bjorner}}]\label{thm:pshell}
    The product of graded posets is EL-shellable if and only if each of the posets is EL-shellable.
\end{theorem}

We are now in a position to complete the proof of Theorem~\ref{thm:kshell}.

\begin{proof}[Proof of Theorem~\ref{thm:kshell}]
Considering Theorem~\ref{thm:kgrade} and Proposition~\ref{prop:keynecrank}, it remains to show that if $\mathbf{a}$ is pure, then $\mathcal{P}(\mathbb{D}(\mathbf{a}))$ is EL-shellable. Let $(\alpha_1,\hdots,\alpha_m)$ be a pure decomposition of $\mathbf{a}$. Combining Theorems~\ref{thm:keybound} and~\ref{thm:kgrade} along with Proposition~\ref{prop:purecomppieces}, it follows that $\mathcal{P}(\mathbb{D}(\alpha_j))$ is graded and EL-shellable for $1\le j\le m$. Thus, since $\mathcal{P}(\mathbb{D}(\mathbf{a}))\cong\mathcal{P}(\mathbb{D}(\alpha_1))\times\hdots\times \mathcal{P}(\mathbb{D}(\alpha_m))$ by Proposition~\ref{prop:kdp}, the result follows from Theorem~\ref{thm:pshell}.
\end{proof}

In the theorem below, we consider the polynomial consequences of a key diagram having a pure and EL-shellable Kohnert poset.

\begin{theorem}\label{thm:keymf}
    Let $\mathbf{a}$ be a weak composition. If $\mathcal{P}(\mathbb{D}(\mathbf{a}))$ is pure and EL-shellable, then $\mathfrak{K}_{\mathbb{D}(\mathbf{a})}$ is multiplicity free.
\end{theorem}
\begin{proof}
    It is shown in \textbf{\cite{key}} that, given a weak composition $\mathbf{a}=(a_1,a_2,\hdots,a_n)$, the polynomial $\mathfrak{K}_{\mathbb{D}(\mathbf{a})}$ is multiplicity free if and only if there exists no three indices $1\le i_1<i_2<i_3\le n$ such that $a_{i_1}<a_{j_2}<a_{j_3}$
    and no four indices $1\le j_1<j_2<j_3<j_4\le n$ such that
    \begin{enumerate}
        \item $a_{j_1}=a_{j_2}<a_{j_3}-1<a_{j_4}-1$,
        \item $a_{j_1}=a_{j_2}<a_{j_4}<a_{j_3}$,
        \item $a_{j_2}<a_{j_1}<a_{j_4}<a_{j_3}$, and
        \item $a_{j_2}<a_{j_1}<a_{j_3}=a_{j_4}$.
    \end{enumerate}
    It is straightforward to verify that pure compositions avoid all five of the patterns above, so the result follows from Theorem~\ref{thm:kshell}.
\end{proof}

As previously noted, unlike in the case of the diagrams considered in Section~\ref{sec:hook}, a key diagram having a multiplicity-free Kohnert polynomial is not equivalent to the (EL-)shellability of the diagram's Kohnert poset. To see this, note that for $\mathbf{a}=(0,3,3)$ we have $$\mathfrak{K}_{\mathbb{D}(\mathbf{a})}=x_1^3x_2^3+x_1^3x_2^2x_3+x_1^2x_2^3x_3+x_1^3x_2x_3^2+x_1^2x_2^2x_3^2+x_1x_2^3x_3^2+x_1^3x_3^3+x_1^2x_2x_3^3+x_1x_2^2x_3^3+x_2^3x_3^3,$$ which is mulitplicity free. However, considering the Hasse diagram of $\mathcal{P}(\mathbb{D}(\mathbf{a}))$ illustrated in Figure~\ref{fig:keymf}, we see that $\mathcal{P}(\mathbb{D}(\mathbf{a}))$ contains an interval isomorphic to the poset $\mathcal{P}$ of Example~\ref{ex:hex}. Thus, $\mathcal{P}(\mathbb{D}(\mathbf{a}))$ is not shellable.

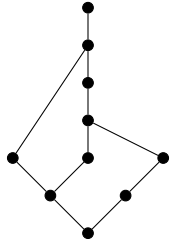
\begin{figure}[H]
    \centering
    $$\begin{tikzpicture}
	\node (1) at (0, 0) [circle, draw = black, fill = black, inner sep = 0.5mm]{};
	\node (2) at (-0.5, 0.5)[circle, draw = black, fill = black, inner sep = 0.5mm] {};
    \node (3) at (0.5, 0.5) [circle, draw = black, fill = black, inner sep = 0.5mm]{};
    \node (4) at (-1, 1) [circle, draw = black, fill = black, inner sep = 0.5mm]{};
    \node (5) at (0, 1) [circle, draw = black, fill = black, inner sep = 0.5mm]{};
    \node (6) at (1, 1) [circle, draw = black, fill = black, inner sep = 0.5mm]{};
    \node (7) at (0, 1.5) [circle, draw = black, fill = black, inner sep = 0.5mm]{};
    \node (8) at (0, 2) [circle, draw = black, fill = black, inner sep = 0.5mm]{};
    \node (9) at (0, 2.5) [circle, draw = black, fill = black, inner sep = 0.5mm]{};
    \node (10) at (0, 3) [circle, draw = black, fill = black, inner sep = 0.5mm]{};
    \draw (1)--(3)--(6)--(7)--(8)--(9)--(4)--(2)--(5)--(7);
    \draw (1)--(2);
    \draw (9)--(10);
\end{tikzpicture}$$
    \caption{Hasse diagram of $\mathcal{P}(\mathbb{D}(0,3,3))$}
    \label{fig:keymf}
\end{figure}

\section{Epilogue}\label{sec:ep}

In this article, our focus was a characterization of (EL-)shellable Kohnert posets. While we were able to establish some general results in Propositions~\ref{prop:strucasc} and~\ref{prop:strucblock}, we were not able to obtain a complete characterization. Instead, we determined characterizations for some particular families of diagrams, including those with at most one cell per column and those whose first two rows are empty. Restricting attention to pure, (EL-)shellable Kohnert posets, we were able to determine characterizations for those diagrams associated with key polynomials. With respect to a general characterization, we make the following conjectures.

\begin{conj}\label{conj:shell}
    There exists a finite number of families of subdiagrams $\mathcal{F}$ such that given any diagram $D$, $\mathcal{P}(D)$ is shellable if and only if there exists no $\tilde{D}\in \mathcal{P}(D)$ such that $\tilde{D}$ contains a subdiagram from $\mathcal{F}$. 
\end{conj}

\begin{conj}\label{conj:key}
    Let $\mathbf{a}=(a_1,\hdots,a_n)$ be a weak composition. Then $\mathcal{P}(\mathbb{D}(\mathbf{a}))$ is shellable if and only if there exist no $1\le i_1<i_2<i_3\le n$ for which
    \begin{itemize}
    \item $a_{i_1}<a_{i_2}<a_{i_3}$ or
    \item $a_{i_1}\le a_{i_3}-3\le a_{i_2}-3$,
    \end{itemize}
    and there exist no $1\le j_1<j_2<j_3<j_4\le n$ for which
    \begin{itemize}
    \item $a_{j_1}\le a_{j_2}<a_{j_3}-1\le a_{j_4}-1$,
    \item $a_{j_1}\le a_{j_2}<a_{j_4}< a_{j_3}$,
    \item $a_{j_2}< a_{j_1}<a_{j_4}< a_{j_3}$, or
    \item $a_{j_2}< a_{j_1}<a_{j_3}\le a_{j_4}$.
\end{itemize}
\end{conj}

\noindent
Note the similarity of Conjecture~\ref{conj:shell} with Conjecture 8.1 of \textbf{\cite{KPoset1}}. 

Along with the characterizations for shellability, we also found that, for the families of diagrams considered here, (EL-)shellability of the Kohnert poset had some interesting polynomial consequences. More specifically, for diagrams with either one cell per nonempty column or the first two rows empty, (EL-)shellability of the Kohnert poset was equivalent to the associated Kohnert polynomial being multiplicity free. On the other hand, for key diagrams, we found that the Kohnert poset being pure and (EL-)shellable only implied that the Kohnert polynomial was multiplicity free. Recall that the weak composition $\mathbf{a}=(0,3,3)$ generated a key polynomial $\mathfrak{K}_{\mathbb{D}(\mathbf{a})}$ which was mutliplicity free, but $\mathcal{P}(\mathbb{D}(\mathbf{a}))$ was not shellable. The authors wonder if there is a stronger polynomial property equivalent to (EL-)shellability in the case of key diagrams.

In addition to the conjectures listed above, many interesting questions remain concerning Kohnert posets. For example, it is of interest to the authors whether results similar to those contained in this article can be obtained for Kohnert posets arising from Rothe diagrams. Given a permutation $w=[w_1,\dots,w_n]$, the associated Rothe diagram is defined as $\mathbb{D}(w)=\{(i,w_j)~|~i<j\text{ and }w_i>w_j\}\subset\mathbb{N}\times\mathbb{N}$, and it was shown in \textbf{\cite{AssafSchu,Winkel2,Winkel1}} that the Kohnert polynomial $\mathfrak{K}_{\mathbb{D}(w)}$ is the Schubert polynomial corresponding to $w$. Thus, a result analogous to Theorem~\ref{thm:keymf} for Kohnert posets of Rothe diagrams would not only shed light on the relationship between the behaviors of Kohnert posets and Schubert polynomials, but it would also suggest a more general phenomenon. In fact, based on experimental evidence, it appears that EL-shellability of Kohnert posets associated with so-called ``southwest" diagrams -- a family of diagrams that contains key and Rothe diagrams as subfamilies (see \textbf{\cite{KP1}}) -- implies that each corresponding Kohnert polynomial is multiplicity free. In a slightly different direction, the authors are aware that E. Philips is working on identifying those Kohnert posets that are lattices \textbf{\cite{E24}}.

\printbibliography

\end{document}